\documentclass[11pt,english,openany,oneside]{article}

\usepackage[T1]{fontenc}
\usepackage[latin9]{inputenc}
\usepackage{fullpage}
\usepackage{units}
\usepackage{amsthm}
\usepackage{amssymb}
\usepackage{mathrsfs}
\usepackage{mathtools}
\usepackage[all]{xy}
\usepackage{wrapfig}
\usepackage[english]{babel}
\usepackage{indentfirst}
\usepackage{fancyhdr}
\usepackage{amssymb}
\usepackage{amsmath}
\usepackage{latexsym}
\usepackage{amsthm}
\usepackage{eucal}
\usepackage{eufrak}
\usepackage[pdftex]{graphicx}
\usepackage{amscd}
\usepackage{color}
\usepackage{lipsum}
\numberwithin{equation}{section} 
\numberwithin{figure}{section} 
\theoremstyle{plain}
\newtheorem{thm}{Theorem}[section]
  \theoremstyle{definition}
  \newtheorem{defn}[thm]{Definition}
  \theoremstyle{definition}
  \newtheorem{example}[thm]{Example}
  \theoremstyle{remark}
  \newtheorem{rem}[thm]{Remark}
  \theoremstyle{plain}
  \newtheorem{prop}[thm]{Proposition}
  \theoremstyle{plain}
  \newtheorem{cor}[thm]{Corollary}
  \theoremstyle{plain}
  \newtheorem{lemma}[thm]{Lemma}
  \theoremstyle{remark}
  
  \theoremstyle{plain}
  
  \theoremstyle{plain}
  
  \theoremstyle{remark}
  \newtheorem{war}[thm]{Warning}

\newcommand\phantomarrow[2]{%
  \setbox0=\hbox{$\displaystyle #1\to$}%
  \hbox to \wd0{%
    $#2\mapstochar
     \cleaders\hbox{$\mkern-1mu\relbar\mkern-3mu$}\hfill
     \mkern-7mu\rightarrow$}%
  \,}
\makeatletter
\newcommand*{\doublerightarrow}[2]{\mathrel{
  \settowidth{\@tempdima}{$\scriptstyle#1$}
  \settowidth{\@tempdimb}{$\scriptstyle#2$}
  \ifdim\@tempdimb>\@tempdima \@tempdima=\@tempdimb\fi
  \mathop{\vcenter{
    \offinterlineskip\ialign{\hbox to\dimexpr\@tempdima+1em{##}\cr
    \rightarrowfill\cr\noalign{\kern.5ex}
    \rightarrowfill\cr}}}\limits^{\!#1}_{\!#2}}}
\newcommand*{\triplerightarrow}[1]{\mathrel{
  \settowidth{\@tempdima}{$\scriptstyle#1$}
  \mathop{\vcenter{
    \offinterlineskip\ialign{\hbox to\dimexpr\@tempdima+1em{##}\cr
    \rightarrowfill\cr\noalign{\kern.5ex}
    \rightarrowfill\cr\noalign{\kern.5ex}
    \rightarrowfill\cr}}}\limits^{\!#1}}}
\newcommand\blfootnote[1]{%
  \begingroup
  \renewcommand\thefootnote{}\footnote{#1}%
  \addtocounter{footnote}{-1}%
  \endgroup
}

\makeatother

\title{A Period Map for Global Derived Stacks}
\author{Carmelo Di Natale}
\date{}

\begin{document}

\maketitle

\abstract{We develop the theory of Griffiths period map, which relates the classification of smooth projective varieties to the associated Hodge structures, in the framework of Derived Algebraic Geometry. We complete the description of the local period map as a morphism of derived deformation functors, following the path marked by Fiorenza, Manetti and Martinengo. In the end we show how to lift the local period map to a (non-geometric) morphism of derived stacks, in order to construct a global version of that.}

\tableofcontents{}

\section*{Introduction}

Let\blfootnote{This work was supported by the Engineering and Physical Sciences Research Council [grant number EP/I004130/1].} $X$ be a smooth complex projective variety of dimension $d$ and consider a family of deformations $\mathcal X\rightarrow S$ of $X$ over some contractible base $S$; in 1968 Griffiths observed that any such Kuranishi family induces canonically a variation of Hodge structures on $X$. More formally let 
\begin{equation*}
0=F^{k+1}H^k\left(X\right)\hookrightarrow F^kH^k\left(X\right)\hookrightarrow\cdots\hookrightarrow F^1H^k\left(X\right)\hookrightarrow F^0H^k\left(X
\right)=H^k\left(X\right)
\end{equation*}
be the Hodge filtration on cohomology and set $b^{p,k}:=\mathrm{dim}\,F^pH^k\left(X\right)$; define 
\begin{equation*}
\mathrm{Grass}\left(H^*\left(X\right)\right):=\prod_k\mathrm{Grass}\left(b^{p,k},H^k\left(X\right)\right)
\end{equation*}
which is a complex projective variety as so are the Grassmannians involved. Griffiths constructed the morphism
\begin{eqnarray} \label{intr G lpm}
\mathcal P^p:&S&\xrightarrow{\hspace*{0.5cm}}\mathrm{Grass}\left(H^*\left(X\right)\right) \nonumber \\ 
&t&\longmapsto \prod_kF^pH^k\left(X_t\right)
\end{eqnarray}
where $X_t$ is the fibre of the family $\mathcal X\rightarrow S$ over $t\in S$; map \eqref{intr G lpm} is said to be the \emph{$p$\textsuperscript{th} local period map} associated to $\mathcal X\rightarrow S$. In \cite{Grif} Griffiths proved that such a map is well-defined and holomorphic; he also computed its differential and showed that it is the same as the contraction map on the space $H^1\left(X,\mathscr T_X\right)$ of first-order deformations of $X$. Moreover it is possible to use map \eqref{intr G lpm} to derive some constraints on the obstructions of $X$. It is important to notice that, despite connecting two algebraically defined objects, map \eqref{intr G lpm} is not algebraic: as a matter of fact one can show that its image lies in the so-called \emph{period domain} (see \cite{Vo} Section 10.1.3), which is a subspace of the Grassmannian determined in general by transcendental equations. \\
The existence and holomorphicity of the local period map says that for any given Kuranishi family of a projective manifold $X$ there is a canonical way to construct a variation of its Hodge structures; moreover such a correspondence seems to be compatible with the general deformation theory of the variety $X$: 
prompted by this observation, in 2006 Fiorenza and Manetti described Griffiths period map in terms of deformation functors. Let 
\begin{eqnarray*}
\mathrm{Def}_X:&\mathfrak{Art}_{\mathbb C}&\xrightarrow{\hspace*{2cm}}\mathfrak{Set} \\
&A&\mapsto\frac{\left\{\text{deformations of }X\text{ over }A\right\}}{\text{isomorphism}}
\end{eqnarray*}
be the functor of Artin rings parametrising the deformations of the variety $X$ and recall that such a deformation functor is isomorphic to the deformation functor associated to the Kodaira-Spencer dgla $KS_X:=\mathbb R\Gamma\left(X,\mathscr T_X\right)$; in a similar way for all non-negative $p$ define the functor of Artin rings
\begin{eqnarray*}
\mathrm{Grass}_{F^pH^*\left(X\right),H^*\left(X\right)}:&\mathfrak{Art}_{\mathbb C}&\xrightarrow{\hspace*{3.5cm}}\mathfrak{Set} \\
&A&\mapsto\frac{\left\{A\text{-deformations of }F^pH^*\left(X\right)\text{ inside } H^*\left(X\right)\right\}}{\text{isomorphism}}
\end{eqnarray*}
which describes the deformations of the complex $F^pH^*\left(X\right)$ as a subcomplex of $H^*\left(X\right)$: this functor precisely encodes variations of Hodge structures on $X$. In \cite{FMan1}, \cite{FMan2} and \cite{FMan3} Fiorenza and Manetti proved the following facts:
\begin{itemize}
\item $\mathrm{Grass}_{F^pH^*\left(X\right),H^*\left(X\right)}$ is a deformation functor (in the sense of Schlessinger) and 
\begin{equation*}
\mathrm{Grass}_{F^pH^*\left(X\right),H^*\left(X\right)}\simeq\mathrm{Def}_{C_{\chi}}
\end{equation*}
where $C_{\chi}$ is the $L_{\infty}$-algebra defined as the cone of the inclusion of dgla's
\begin{equation*}
\chi:\mathrm{End}^{F^p}\left(H^*\left(X\right)\right)\hookrightarrow\mathrm{End}^*\left(H^*\left(X\right)\right)
\end{equation*}
with 
\begin{equation*}
\mathrm{End}^{F^p}\left(H^*\left(X\right)\right):=\left\{f\in\mathrm{End}^*\left(H^*\left(X\right)\right)\big|f\left(F^pH^*\left(X\right)\right)\subseteq F^pH^*\left(X\right)\right\};
\end{equation*}
\item the map
\begin{eqnarray*}
\mathrm{FM}^p:&KS_X&\xrightarrow{\hspace*{0.75cm}}C_{\chi} \\
&\xi&\longmapsto\left(l_{\xi},i_{\xi}\right)
\end{eqnarray*}
where $i$ is the contraction of differential forms with vector fields and $l$ stands for the holomorphic Lie derivative, is a $L_{\infty}$-morphism, thus it induces a morphism of deformation functors
\begin{equation*}
\mathrm{FM}^p:\mathrm{Def}_{KS_X}\longrightarrow\mathrm{Def}_{C_{\chi}};\footnote{Here, by a slight abuse of notation, the symbol $\mathrm{FM}^p$ is denoting both the $L_{\infty}$-map and the induced morphism of deformation functors.}
\end{equation*}
\item the natural transformation
\begin{eqnarray} \label{FM lpm intr}
\mathcal P^p:\mathrm{Def}_X\quad\;&\xrightarrow{\hspace*{0.5cm}}&\mathrm{Grass}_{F^pH^*\left(X\right),H^*\left(X\right)} \\
\forall A\in\mathfrak{Art}_{\mathbb C}\quad\left(\mathscr O_A\overset{\xi}{\rightarrow}\mathscr O_X\right)&\longmapsto& \quad F^pH^*\left(X,\mathscr  O_A\right) \nonumber
\end{eqnarray}
is a morphism of deformation functors extending Griffiths period map\footnote{This is why, by a slight abuse of notation, we are using the same symbol for both.} and the two morphisms $\mathrm{FM}^p$ and $\mathcal P^p$ are canonically isomorphic.
\end{itemize}
Fiorenza and Manetti's work shows that the $p$\textsuperscript{th} local period map is actually a \emph{morphism of deformation theories}, thus it commutes with all deformation-theoretic constructions: in particular all results of Griffiths about the differential of map \eqref{intr G lpm} are easily recovered as purely formal corollaries of the preceding statements. Moreover Fiorenza and Manetti's construction works for any proper smooth scheme of dimension $d$ over a field $k$ of characteristic $0$.\\
As we are able to interpret Griffiths period map as a natural transformation of deformation functors, the next step would be to look at it in the context of Derived Deformation Theory: more formally one could ask whether there exist derived enhancements of the functors $\mathrm{Def}_X$ and $\mathrm{Grass}_{F^pH^*\left(X\right),H^*\left(X\right)}$ for which it is possible to find some natural derived extension of morphism \eqref{FM lpm intr}. In 2012 Fiorenza and Martinengo approached this problem, tackling it from an entirely algebraic viewpoint. As a matter of fact they observed that the contraction of differential forms with vector fields $i$ (seen in the most general way, i.e. as a morphism of complexes of sheaves over $X$) and the Lie derivative $l$ give rise to a morphism of differential graded Lie algebras
\begin{equation} \label{FMM htpy fibre}
\mathrm{FMM}:\mathbb R\Gamma\left(X,\mathscr T_X\right)\overset{\left(l,e^i\right)}{\xrightarrow{\hspace*{1cm}}}\underset{\longleftarrow}{\mathrm{holim}}\left(\mathrm{End}^{\geq 0}\left(\mathbb R\Gamma\left(X,\Omega^*_X\right)\right)\doublerightarrow{\mathrm{incl.}}{0}\mathrm{End}^*\left(\mathbb R\Gamma\left(X,\Omega^*_X\right)\right)\right)
\end{equation}
where $\mathrm{End}^{\geq 0}\left(\mathbb R\Gamma\left(X,\Omega^*_X\right)\right)$ is the dgla made of non-negatively graded maps of the complex $\mathbb R\Gamma\left(X,\Omega^*_X\right)$ in itself, which can be interpreted as the dgla of all filtration-preserving maps. Notice also that the codomain of map \eqref{FMM htpy fibre} is nothing but the homotopy fibre over $0$ of the inclusion of $\mathrm{End}^{\geq 0}\left(\mathbb R\Gamma\left(X,\Omega^*_X\right)\right)$ into $\mathrm{End}\left(\mathbb R\Gamma\left(X,\Omega^*_X\right)\right)$. In \cite{FMar} Fiorenza and Martinengo showed that map \eqref{FMM htpy fibre} induces a morphism of derived deformation functors 
\begin{equation} \label{FMM lpm}
\mathrm{FMM}:\mathbb R\mathrm{Def}_{\mathbb R\Gamma\left(X,\mathscr T_X\right)}\overset{\left(l,e^i\right)}{\xrightarrow{\hspace*{1cm}}}\mathbb R\mathrm{Def}_{\underset{\longleftarrow}{\mathrm{holim}}\left(\mathrm{End}^{\geq 0}\left(\mathbb R\Gamma\left(X,\Omega^*_X\right)\right)\doublerightarrow{\mathrm{incl.}}{0}\mathrm{End}^*\left(\mathbb R\Gamma\left(X,\Omega^*_X\right)\right)\right)}
\end{equation}
whose $0$-truncation $\mathrm{FM}$ is very close to $\mathrm{FM}^p$ (actually $\mathrm{FM}$ is even more interesting than $\mathrm{FM}^p$ as it does not depend on the degree of the filtration, so it can be interpreted as a universal version of Griffiths period map). \\
The ultimate goal of this paper is to lift Fiorenza, Manetti and Martinengo's work to a global level, i.e. to find a morphism of derived stacks whose restriction to formal neighbourhoods gives back map \eqref{FMM lpm}. A first crucial step in order to do so consists of finding a more geometric description of such a map, thus we will define the morphism of derived deformation functors
\begin{eqnarray} \label{geom FMM}
&\scriptstyle{\mathcal P:\;\mathbb R\mathrm{Def}_X}&\scriptstyle{\longrightarrow\;\,\mathrm{hoFlag}\left(\mathbb R\Gamma\left(X,\Omega^*_X\right),F^{\bullet}\mathbb R\Gamma\left(X,\Omega^*_X\right)\right)} \nonumber \\
\scriptstyle{\forall A\in\mathfrak{dgArt}_k^{\leq0}}&\scriptstyle{\left[\mathscr O_{A,*}\overset{\varphi}{\longrightarrow}\mathscr O_X\right]\approx\left[\vcenter{\xymatrix{\ar@{} |{\Box^h}[dr]X\ar@{^{(}->}[r]\ar[d] & \mathcal X\ar[d] \\
\mathrm{Spec}\left(k\right)\ar[r] & \mathbb R\mathrm{Spec}\left(A\right)}}\right]}&\scriptstyle{\mapsto\qquad\left[\left(\left(\mathbb R\Gamma\left(\pi^0\mathcal X,\Omega^*_{\mathcal X/A}\right),F^{\bullet}\right),\tilde{\varphi}\right)\right]}
\end{eqnarray}
and prove that it is naturally isomorphic to map \eqref{FMM lpm}; in formula \eqref{geom FMM} $\mathbb R\mathrm{Def}_X$ parametrises derived deformations of the scheme $X$ (i.e. homotopy flat families of derived schemes deforming the underived scheme $X$), while $\mathrm{hoFlag}\left(\mathbb R\Gamma\left(X,\Omega^*_X\right),F^{\bullet}\mathbb R\Gamma\left(X,\Omega^*_X\right)\right)$ encodes derived deformations of the filtered complex $\left(\mathbb R\Gamma\left(X,\Omega^*_X\right),F^{\bullet}\right)$. Although it is intuitively quite clear what such functors should be, giving a careful definition of them reveals to be non-trivial at all and has actually lead us to develop the notions of affine differential graded category and affine simplicial category, which are probably interesting objects themselves to study. $\mathbb R\mathrm{Def}_X$ and $\mathrm{hoFlag}\left(\mathbb R\Gamma\left(X,\Omega^*_X\right),F^{\bullet}\mathbb R\Gamma\left(X,\Omega^*_X\right)\right)$ are by construction formal neighbourhoods of interesting derived stacks: the former is the formal neighbourhood at $X$ of the derived (non-geometric) stack $\mathcal{DS}ch_{d/k}$ of derived schemes of dimension $d$, which has been recently studied by Pridham in \cite{Pr3}, while the latter is the formal neighbourhood at $\left(\mathbb R\Gamma\left(X,\Omega^*_{X/k}\right),F^{\bullet}\right)$ of the homotopy flag variety $\mathcal{DF}lag_k\left(\mathbb R\Gamma\left(X,\Omega^*_{X/k}\right)\right)$, which is studied in \cite{dN2}. \\
The above tools provide us with a path towards a sensible global version of map \eqref{geom FMM}, i.e. a geometric morphism of derived geometric stacks inducing the latter on formal neighbourhoods. A very partial answer to such a question is discussed in the end of this paper, where we consider the non-geometric morphism of derived stacks
\begin{eqnarray} \label{non-geom derived Griff}
&\scriptstyle{\underline{\mathbb R\mathcal P}:\;\mathcal{DS}ch_{d/k}\times^h_{\mathbb R\mathcal Perf_k}\left\{\mathbb R\Gamma\left(X,\Omega^*_{X/k}\right)\right\}}&\scriptstyle{\longrightarrow\;\;\mathcal{DF}lag_k\left(\mathbb R\Gamma\left(X,\Omega^*_{X/k}\right)\right)} \nonumber \\
\scriptstyle{\forall A\in\mathfrak{dgAlg}_k^{\leq0}}&\scriptstyle{\left[Y,\theta:\mathbb R\Gamma\left(\pi^0 Y,\Omega^*_{Y/A}\right)\overset{\sim}{\rightarrow}\mathbb R\Gamma\left(X,\Omega^*_{X/k}\right)\otimes A\right]}&\scriptstyle{\longmapsto\;\;\;\,\left(\mathbb R\Gamma\left(\pi^0 Y,\Omega^*_{Y/A}\right),F^{\bullet}\right)}
\end{eqnarray}
and observe that it actually induces map \eqref{geom FMM} on formal neighbourhoods at $\left(X,\mathbb R\Gamma\left(X,\Omega^*_{X/k}\right)\right)$ and $\underline{\mathbb R\mathcal P}\left(\left(X,\mathbb R\Gamma\left(X,\Omega^*_{X/k}\right)\right)\right)=\left(\mathbb R\Gamma\left(X,\Omega^*_{X/k}\right),F^{\bullet}\right)$. However it is important to notice that map \eqref{non-geom derived Griff} cannot be geometric for at least two important reasons:
\begin{enumerate}
\item a global geometric morphism lifting Griffiths period map to the derived world cannot be algebraic as the classical underived map is already non-algebraic in general;
\item if we move the problem to the context of the newly-born Derived Analytic Geometry, we will have to take care of monodromy issues: in fact, given a family $p:\mathcal X\rightarrow\mathcal S$ of derived schemes (or even derived analytic spaces) globally deforming $X$, the presence of a (possibly higher) monodromy group acting non-trivially on the fibres of $\mathbb Rp_{*}\Omega^*_{\mathcal X/\mathcal S}$ may represent a serious obstruction to make map \eqref{non-geom derived Griff} into a derived analytic morphism of derived analytic stacks.
\end{enumerate}
The problem of constructing a derived geometric global period map in the context of Derived Analytic Geometry seems to be very interesting: it is now being jointly studied by Holstein and the author and will appear in \cite{dNH}. \\
$\mathbf{Acknowledgements}$ --- The author does wish to thank his PhD supervisor Jonathan P. Pridham for suggesting the problem and for his constant support and advise along all the preparation of this paper. The author is also deeply indebted to Domenico Fiorenza, Ian Grojnowski, Julian V. S. Holstein, Donatella Iacono, Dominic Joyce, Marco Manetti and Elena Martinengo for several inspiring discussions about the period mapping in Deformation Theory.

\section{The Period Map as a Holomorphic Function}

Let $X$ be a compact connected complex K\" ahler manifold of dimension $d$ and consider a family of deformations $\varphi:\mathcal X\rightarrow S$, i.e. a proper holomorphic submersion of complex manifolds (where the base $S$ is contractible) admitting a distinguished fibre $\varphi^{-1}\left(0\right)=:X_0\simeq X$. Recall that a famous result due to Ehresmann says that any such family is $C^{\infty}$-trivial, i.e. there exists a diffeomorphism 
\begin{equation} \label{Ehr}
T:\mathcal X\overset{\sim}{\longrightarrow}X_0\times S\simeq X\times S
\end{equation}
over $S$ (see \cite{Vo} Theorem 9.3). For all $t\in S$ let $X_t:=\varphi^{-1}\left(t\right)$: Ehresmann's trivialisation \eqref{Ehr} clearly induces a diffeomorphism $X_t\simeq X$ for all $t$, thus we can think of the morphism $\varphi$ as a collection of complex structures over the $C^{\infty}$-manifold underlying the complex variety $X$. This situation is the prototypical example of all deformation problems and was originally studied by Kodaira and Spencer. \\ 
A very natural question to ask is how the standard Hodge structures over $X$ vary with respect to the family $\varphi$; more formally, consider the cohomology algebra of $X$
\begin{equation*}
H^*\left(X,\mathbb C\right):=\bigoplus_{0\leq k\leq d}H^{k}\left(X,\mathbb C\right)
\end{equation*}
and recall that each cohomology group $H^k\left(X,\mathbb C\right)$ is endowed with a Hodge structure of weight $k$ defined by the Hodge decomposition
\begin{equation*}
H^k\left(X,\mathbb C\right)=\bigoplus_{p+q=k}H^{p.q}\left(X\right) \qquad\qquad H^{p,q}\left(X\right)\simeq H^q\left(X,\Omega^p_X\right)
\end{equation*}
or, equivalently, by the Hodge filtration
\begin{equation*}
0=F^{k+1}H^k\left(X,\mathbb C\right)\hookrightarrow F^kH^k\left(X,\mathbb C\right)\hookrightarrow\cdots\hookrightarrow F^1H^k\left(X,\mathbb C\right)\hookrightarrow F^0H^k\left(X,\mathbb C\right)=H^k\left(X,\mathbb C\right)
\end{equation*}
where
\begin{equation*}
F^mH^k\left(X,\mathbb C\right):=\bigoplus_{p\geq m}H^{p,q}\left(X\right).
\end{equation*}
The question we want to address is whether the family $\varphi$ induces any interesting structure on the cohomology of the fibres.
\subsection{Griffiths Period Map}
Observe that Ehresmann's trivialisation \eqref{Ehr} provides us with a diagram of isomorphisms of vector spaces
\begin{equation*}
\xymatrix{H^k\left(\mathcal X,\mathbb C\right)\ar[r]^{\sim}\ar@{=}[d] & H^k\left(X,\mathbb C\right)\ar[d]_{\wr} \\
H^k\left(\mathcal X,\mathbb C\right)\ar[r]^{\sim} & H^k\left(X_t,\mathbb C\right)}
\end{equation*}
which commutes for all $k\geq 0$ and $\forall t\in S$; actually much more is true, as for all $k\geq 0$ the sheaf $\mathbb R^k\varphi_*\underline{\mathbb C}$ -- where $\varphi_*:\mathfrak{Sh}\left(\mathcal X\right)\rightarrow\mathfrak{Sh}\left(S\right)$ is the push-forward functor -- is seen to be a local system over $S$ isomorphic to the constant sheaf $\underline{H^k\left(X,\mathbb C\right)}$ (see \cite{Vo} Section 9.2 for a more detailed explanation), thus the diagram above does not depend on the choice of the trivialisation. Denote
\begin{equation*}
h^k:=\mathrm{dim}\,H^k\left(X,\mathbb C\right)\qquad\qquad h^{p,q}:=\mathrm{dim}\,H^{p,q}\left(X\right)\qquad\qquad b^{p,k}:=\mathrm{dim}\,F^pH^k\left(X,\mathbb C\right).
\end{equation*} 
A standard argument based on the $E_1$-degeneration of the Hodge-to-De Rham spectral sequence of $X$ shows that there exists a neighbourhood of $0\in S$ such that 
\begin{equation*}
\mathrm{dim}\,H^{p,q}\left(X_t\right)=\mathrm{dim}\,H^{p,q}\left(X\right)=:h^{p,q}
\end{equation*}
thus the Hodge numbers of $X$ are invariant under (infinitesimal) deformation; moreover this immediately implies that the Hodge-to-De Rham spectral sequence of such fibres degenerates at its first page\footnote{Up to shrinking the base $S$, the fibres of $\varphi$ are K\" ahler manifolds themselves (see \cite{Vo} Theorem 9.23).}, as well (see \cite{Vo} Proposition 9.20). 
\begin{defn} (Griffiths) In the above notations define the \emph{$(p,k)$\textsuperscript{th} local period map} to be
\begin{eqnarray} \label{G plpm}
\mathcal P^{p,k}:&S&\xrightarrow{\hspace*{0.5cm}}\mathrm{Grass}\left(b^{p,k},H^k\left(X,\mathbb C\right)\right) \nonumber \\ 
&t&\longmapsto F^pH^k\left(X_t,\mathbb C\right).
\end{eqnarray}
\end{defn}
Since $H^k\left(X_t,\mathbb C\right)$ is canonically isomorphic to $H^k\left(X,\mathbb C\right)$ and the Hodge numbers of $X$ are invariant under deformation, map \eqref{G plpm} is well-defined; the following is a famous result of Griffiths.
\begin{thm} \emph{(Griffiths)}
The $(p,k)$\textsuperscript{th} local period map \eqref{G plpm} is holomorphic $\forall p\leq k$.
\end{thm}
\begin{proof}
See \cite{Vo} Theorem 10.9.
\end{proof}
\subsection{The Differential of the Period Mapping}
Griffiths deeply studied the differential of map \eqref{G plpm}, as well: in order to state his result let us review what the \emph{contraction of differential forms with vector fields} and the \emph{(holomorphic) Lie derivative} are. Recall that the tangent sheaf $\mathscr T_X$ is endowed with a natural structure of sheaf of Lie algebras (which can be considered as dgla's concentrated in degree $0$) induced by the canonical isomorphism $\mathscr T_X\simeq\mathcal Der\left(\mathscr O_X,\mathscr O_X\right)$, while $\mathcal End^*\left(\Omega_X^*\right)$ comes with a structure of sheaf of differential graded Lie algebras through the standard differential on $\mathrm{Hom}$ complexes and the standard Lie bracket. Now the contraction morphism is defined to be the {} ``shifted'' map of sheaves of differential graded Lie algebras
\begin{eqnarray} \label{contr map}
i:&\mathscr T_X&\xrightarrow{\hspace*{0.25cm}}\mathcal End^*\left(\Omega_X^*\right)\left[-1\right] \nonumber \\
&\xi&\longmapsto i_{\xi}\text{ such that }i_{\xi}\left(\omega\right):=\xi\lrcorner\,\omega\qquad\qquad\text{(on local sections)}
\end{eqnarray}
while the differential of map \eqref{contr map} (as an element of the complex $\mathrm{Hom}^*\left(\mathscr T_X,\mathcal End^*\left(\Omega_X^*\right)\left[-1\right]\right)$) is by definition the Lie derivative
\begin{eqnarray} \label{Lie deriv}
l:&\mathscr T_X&\xrightarrow{\hspace*{0.25cm}}\mathcal End^*\left(\Omega_X^*\right) \nonumber \\
&\xi&\longmapsto l_{\xi}\text{ such that }l_{\xi}\left(\omega\right):=d\left(\xi\lrcorner\,\omega\right)+\xi\lrcorner\,\left(d\omega\right)\quad\qquad\text{(on local sections)}
\end{eqnarray}
which is a genuine morphism of sheaves of dgla's\footnote{By a slight abuse of notation we will tend to denote by $i$ and $l$ the morphisms that maps \eqref{contr map} and \eqref{Lie deriv} induce on global sections and derived global sections, as well.}. 
\begin{thm} \emph{(Griffiths)}
The differential $d\mathcal P^{p,k}$ of map \eqref{G plpm} factors through the (cohomology) contraction map
\begin{equation} \label{dG plpm}
i:H^1\left(X,\mathscr T_X\right)\longrightarrow\mathrm{Hom}\left(F^pH^k\left(X,\mathbb C\right),\frac{H^k\left(X,\mathbb C\right)}{F^pH^k\left(X,\mathbb C\right)}\right).
\end{equation} 
Moreover map \eqref{dG plpm} actually takes values in $\mathrm{Hom}\left(F^pH^k\left(X,\mathbb C\right),\frac{F^{p-1}H^k\left(X,\mathbb C\right)}{F^pH^k\left(X,\mathbb C\right)}\right)$.\footnote{This last property is generally known as \emph{Griffiths transversality}.}
\end{thm}
\begin{proof}
The theorem has been stated in relatively modern terms, but a complete proof of it is given in \cite{Vo} Proposition 10.12, Lemma 10.19 and Theorem 10.21.  
\end{proof}
The $(p,k)$\textsuperscript{th} local period map \eqref{G plpm} depends by definition on two parameters, a cohomology one -- that is $k$ -- and a filtration one -- that is $p$; we would like to encode all cohomological information about the variations of Hodge structures induced by the family $\phi$ in a single morphism.
\begin{defn}
In the above notations define the \emph{$p$\textsuperscript{th} local period map} to be
\begin{eqnarray} \label{G lpm}
\mathcal P^p:&S&\xrightarrow{\hspace*{0.5cm}}\mathrm{Grass}\left(H^*\left(X,\mathbb C\right)\right) \nonumber \\ 
&t&\longmapsto \prod_kF^pH^k\left(X_t,\mathbb C\right).
\end{eqnarray}
\end{defn}
Notice that map \eqref{G lpm} is holomorphic and that its differential $d\mathcal P^p$ still factors through a contraction morphism
\begin{eqnarray} \label{per diff}
i:H^1\left(X,\mathscr T_X\right)\longrightarrow\bigoplus_k\mathrm{Hom}\left(F^pH^k\left(X,\mathbb C\right),\frac{H^k\left(X,\mathbb C\right)}{F^pH^k\left(X,\mathbb C\right)}\right).
\end{eqnarray}

\section{The Period Map as a Morphism of Deformation Functors}

The work of Griffiths which has been described in Section $1$ relates deformations of a complex smooth projective variety (or more generally complex K\" ahler manifold) to variations of its Hodge structures. Unfortunately the local period map \eqref{G lpm} is not really a morphism of deformation theories, as it depends on a given deformation of a complex variety $X$, nonetheless its differential \eqref{per diff} is very {} ``deformation-theoretic'' in nature, as it relates the space $H^1\left(X,\mathscr T_X\right)$, i.e. the tangent space to the deformation functor parametrising all deformations of $X$, to another cohomological invariant which depends only on $X$ rather than the special Kuranishi family over $X$ determining map \eqref{G lpm}. Observations like these led Fiorenza and Manetti to believe that Griffiths period map could be described as a morphism of deformation functors (in the sense of Schlessinger) whose induced tangent mapping coincided with map \eqref{per diff}.
\subsection{Deformations of $k$-Schemes}
Let $k$ be any (non-necessarily algebraically closed) field of characteristic $0$ and consider a smooth proper scheme $X$ of dimension $d$: these assumptions over $X$ just algebraically resemble the analytic framework in which Griffiths studied map \eqref{G lpm}, while the fact that the theory we are about to summarise works for any field of characteristic $0$ is a consequence of Deligne's views on Hodge Theory (for more details see \cite{Del1}, \cite{Del2} and \cite{Del3}). Notice also that by \cite{Del1} Theorem 5.5 the Hodge-to-De Rham spectral sequence of the scheme $X$ degenerates at its first page: such a property will be used several times in this paper. \\
Recall that the \emph{functor of deformations of $X$} is the functor of Artin rings
\begin{eqnarray} \label{Def_X}
\mathrm{Def}_X:&\mathfrak{Art}_k&\xrightarrow{\hspace*{2cm}}\mathfrak{Set} \nonumber \\
&A&\mapsto\frac{\left\{\text{deformations of }X\text{ over }A\right\}}{\text{isomorphism}}
\end{eqnarray}
where a deformation\footnote{From now on by deformation we will always mean infinitesimal deformation, i.e. a deformation over an Artinian base.} of $X$ over $A$ is a Cartesian diagram in $\mathfrak{Sch}_k$ 
\begin{equation*}
\xymatrix{\ar@{} |{\Box}[dr]X\ar@{^{(}->}[r]^i\ar[d] & \mathcal X\ar[d]^p \\
\mathrm{Spec}\left(k\right)\ar[r] & \mathrm{Spec}\left(A\right)}
\end{equation*}
with $i$ a closed immersion and $p$ flat and proper; equivalently a deformation of $X$ over $A$ can be viewed as a morphism of sheaves of $A$-algebras $\mathscr O_A\rightarrow\mathscr O_X$ such that $\mathscr O_A$ is flat over $A$ and $\mathscr O_A\otimes_Ak\simeq \mathscr O_X$. Of course, two $A$-deformations $\mathcal X_1\rightarrow\mathrm{Spec}\left(A\right)$ and $\mathcal X_2\rightarrow\mathrm{Spec}\left(A\right)$ of $X$ are said to be isomorphic if there is an isomorphism $\mathcal X_1\tilde{\longrightarrow}\mathcal X_2$ of schemes over $A$ inducing the identity on $X$: it is well-known that functor \eqref{Def_X} is a deformation functor in the sense of Schlessinger (see \cite{Man2} or \cite{Schl} for a definition). \\
Now let $\left(\mathfrak l,\langle\cdots\rangle_n\right)_{n>0}$ be a $L_{\infty}$-algebra over $k$ (see \cite{Man2} for a definition) and recall that the \emph{deformation functor associated to $\left(\mathfrak l,\langle\cdots\rangle_n\right)_{n>0}$} is defined to be
\begin{eqnarray*}
\mathrm{Def}_{\mathfrak l}:&\mathfrak{Art}_k&\xrightarrow{\hspace*{1.5cm}}\mathfrak{Set} \\
&A&\mapsto\frac{\mathrm{MC}_{\mathfrak l}\left(A\right)}{\text{homotopy equivalence}}
\end{eqnarray*}
where
\begin{equation*}
\mathrm{MC}_{\mathfrak l}\left(A\right):=\left\{x\in\mathfrak l^0\left[1\right]\otimes\mathfrak m_A\text{ s.t. }\sum_{n\geq 1}\frac{\left\langle x^{\odot n}\right\rangle_n}{n!}=0\right\}
\end{equation*}
is the set of solutions of the \emph{Maurer-Cartan equation} and two elements $x_0,x_1\in\mathrm{MC}_{\mathfrak l}\left(A\right)$ are said to be \emph{homotopy equivalent} if there exists a {} ``path'' $x\left(t,dt\right)\in\mathrm{MC}_{\mathfrak l\left[t,dt\right]}\left(A\right)$ such that $x\left(0\right)=x_0$ and $x\left(1\right)=x_1$; again, it is not hard to verify that $\mathrm{Def}_{\mathfrak l}$ is a deformation functor in the sense of Schlessinger. Notice that, if the higher products $\langle\cdots\rangle_n=0$ for all $n\geq3$, i.e. if the $L_{\infty}$-algebra is actually a differential graded Lie algebra (see \cite{dN1} or \cite{Man2} for a definition), we recover the more classical notion of deformation functor associated to a dgla. \\ 
A fundamental fact in Deformation Theory -- essentially due to Kodaira, Kuranishi and Spencer and developed in many ways by several other authors -- states that the functor of deformations $\mathrm{Def}_X$ associated to a scheme $X$ which satisfies the above conditions is isomorphic to the deformation functor associated to the \emph{Kodaira-Spencer dgla} of $X$, which is defined to be the differential graded Lie algebra $\left(KS_X,\left[-,-\right],D\right)$ where
\begin{eqnarray} \label{KS Lie struct}
&KS_X:=\mathbb R\Gamma\left(X,\mathscr T_X\right)\simeq\Gamma\left(X,\mathscr A^{0,*}_X\left(\mathscr T_X\right)\right)& \nonumber \\
&\left[fd\bar z_I\frac{\partial}{\partial z_i},gd\bar z_J\frac{\partial}{\partial z_j}\right]:=d\bar z_I\wedge d\bar z_J\left(f\frac{\partial g}{\partial z_i}\frac{\partial}{\partial z_j}-g\frac{\partial f}{\partial z_j}\frac{\partial}{\partial z_i}\right).& \nonumber \\
&D\left(\omega\frac{\partial}{\partial z_i}\right):=-\bar{\partial}\left(\omega\right)\frac{\partial}{\partial z_i}&
\end{eqnarray}
\begin{war} \label{KS resol}
In this paper the algebra $KS_X$ will always correspond to the specific resolution $\Gamma\left(X,\mathscr A^{0,*}_X\left(\mathscr T_X\right)\right)$ computing $\mathbb R\Gamma\left(X,\mathscr T_X\right)$, equipped with the Lie structure \eqref{KS Lie struct}. This becomes very relevant in comparisons with the work of Fiorenza, Manetti and Martinengo.
\end{war}
Now consider the natural transformation
\begin{eqnarray} \label{Def_X=Def_KS}
\mathscr O:\mathrm{Def}_{KS_X}&\xrightarrow{\hspace*{1cm}}&\quad\mathrm{Def}_X \nonumber \\
\forall A\in\mathfrak{Art}_k\quad\mathrm{Def}_{KS_X}\left(A\right)\ni\xi\qquad&\longmapsto&\left(\mathscr O_{\xi}\rightarrow\mathscr O_X\right)\in\mathrm{Def}_X\left(A\right).
\end{eqnarray}
where for all open $U\subseteq X$
\begin{equation*}
\mathscr O_{\xi}\left(U\right):=\left\{f\in \mathscr A^{0,0}_X\left(U\right)\otimes A\text{ s.t. }\bar{\partial}f=\xi\lrcorner\partial f\right\}
\end{equation*}
and the map $\mathscr O_{\xi}\rightarrow\mathscr O_X$ is induced by the projection $\mathscr A^{0,0}_X\otimes A\rightarrow\mathscr A^{0,0}_X$.
\begin{thm} \emph{(Kodaira-Spencer, Kuranishi, [...])} \label{Donatella}
In the above notations, map \eqref{Def_X=Def_KS} is an isomorphism of deformation functors.
\end{thm}
\begin{proof}
There is a variety of different proofs of this result in the literature: we refer to \cite{Iac1} Theorem II.7.3 for a very detailed algebraic one; see also \cite{Iac2} Theorem 3.4.
\end{proof}
\subsection{Mapping Cones and Deformations of Filtered Complexes}
The functor $\mathrm{Def}_X$ is the most natural candidate for the domain of a purely {}``deformation-theoretic'' version of Griffiths period map; now we wish to understand what the codomain of such a morphism should be, i.e. we seek a deformation functor which parametrises variations of Hodge structures over $X$. \\ 
Let $\left(V,d\right)$ be a differential graded $k$-vector space and $\left(W,d\right)$ a subcomplex of its; for any $A\in\mathfrak{Art}_k$, consider the groups\footnote{In this section, by a slight abuse of notation, the symbol $d$ may indifferently denote the differential of the complex $V$, the differential of the twisted complex $V\otimes A$ and the differential of the endomorphism complex $\mathrm{End}\left(\left(V,d\right)\right)$.}
\begin{eqnarray*}
\mathrm{Aut}^V\left(A\right)&:=&\left\{f\in\mathrm{Hom}^0_A\left(V\otimes A,V\otimes A\right)\text{ s.t. } f\equiv\mathrm{Id}_{\left(V,d\right)}\;\left(\mathrm{mod}\;\mathfrak m_A\right)\right\} \\
\mathrm{Aut}^{\left(V,d\right)}\left(A\right)&:=&\left\{f\in\mathrm{Aut}^V\left(A\right)\text{ s.t. }fd=df\right\} \\
\mathrm{Aut}^{W,V}\left(A\right)&:=&\left\{f\in\mathrm{Aut}^V\left(A\right)\text{ s.t. }f\left(W\otimes A\right)=W\otimes A\right\} \\
\tilde{\mathrm{Aut}}^{\left(V,d\right)}\left(A\right)&:=&\left\{f\in\mathrm{Aut}^{V,d}\left(A\right)\text{ s.t. }H^*\left(f\right)\text{ is the identity on }H^*\left(V\otimes A,d\right)\right\}
\end{eqnarray*} 
and define the functor of deformations of $\left(W,d\right)$ inside $\left(V,d\right)$ to be the functor of Artin rings
\begin{eqnarray} \label{Grass func}
\mathrm{Grass}_{W,V}:&\mathfrak{Art}_k&\xrightarrow{\hspace*{3cm}}\mathfrak{Set} \nonumber \\
&A&\mapsto\frac{\left\{f\in\mathrm{Aut}^V\left(A\right)\text{ s.t. }df\left(W\otimes A\right)\subseteq f\left(W\otimes A\right)\right\}}{\tilde{\mathrm{Aut}}^{\left(V,d\right)}\left(A\right)\times\mathrm{Aut}^{W,V}\left(A\right)}.
\end{eqnarray}
\begin{rem}
Formula \eqref{Grass func} is the original definition of the functor of deformations of the subcomplex $\left(W,d\right)$ as we find it in \cite{FMan1}; although it is quite elegant, it may not seem very intuitive, as there is no explicit reference to what a deformation of $\left(W,d\right)$ over a local Artinian $k$-algebra $A$ should be. Anyway a more careful look at it immediately shows that a deformation of $\left(W,d\right)$ over $A$ inside $\left(V,d\right)$ is a complex of free $A$-modules $\left(V\otimes A,d_A\right)$ such that its residue modulo $\mathfrak m_A$ equals $\left(V,d\right)$ and $d_A\left(W\otimes A\right)\subseteq W\otimes A$ (this is exactly what the {} ``numerator'' in formula \eqref{Grass func} parametrises); on the other hand two such deformations $\left(V\otimes A,d_A\right)$ and $\left(V\otimes A,d_A'\right)$ are isomorphic if there exists an isomorphism of cochain complexes $\varphi$ between them such that $\varphi\left(W\otimes A,d_A\right)=\varphi\left(W\otimes A,d_A'\right)$ and $H^i\left(\varphi\right)=\mathrm{Id}_{H^i\left(V\otimes A,d\right)}$ for all $i$ (this is exactly what the {} ``denominator'' in formula \eqref{Grass func} parametrises).
\end{rem}
Now consider the graded vector spaces
\begin{eqnarray*}
\mathrm{End}^*\left(\left(V,d\right)\right)&:=&\mathrm{Hom}^*\left(\left(V,d\right),\left(V,d\right)\right) \\
\mathrm{End}^{W}\left(\left(V,d\right)\right)&:=&\left\{f\in\mathrm{End}^*\left(\left(V,d\right)\right)\text{ s.t. }f\left(W\right)\subseteq W\right\}.
\end{eqnarray*}
They are endowed with natural structures of differential graded Lie algebras and there is an obvious inclusion
\begin{equation*}
\chi_{W,V}:\mathrm{End}^W\left(\left(V,d\right)\right)\hookrightarrow\mathrm{End}^*\left(\left(V,d\right)\right)
\end{equation*}
which is a morphism of dgla's; recall also that the \emph{mapping cone} $\left(C_{\chi_{W,V}},\delta\right)$ of the morphism $\chi_{W,V}$ is defined to be its homotopy cokernel, i.e. the complex 
\begin{equation*}
\underset{\longleftarrow}{\mathrm{holim}}\left(\mathrm{End}^W\left(\left(V,d\right)\right)\doublerightarrow{\chi_{W,V}}{0}\mathrm{End}^*\left(\left(V,d\right)\right)\right).
\end{equation*} 
More concretely, the mapping cone is given by the formulae 
\begin{eqnarray} \label{cone formula}
C_{\chi_{W,V}}^i&:=&{\mathrm{End}^{W}\left(\left(V,d\right)\right)}^i\oplus{\mathrm{End}^{i-1}\left(\left(V,d\right)\right)} \nonumber\\
\delta\left(\left(f,g\right)\right)&:=&\left(df,\chi\left(f\right)-dg\right).
\end{eqnarray}
\begin{prop} \label{mapping cone} \emph{(Fiorenza-Manetti)}
In the above notations, there is a canonical $L_{\infty}$-structure on the mapping cone $C_{\chi_{W,V}}$.
\end{prop}
\begin{proof}
See \cite{FMan2} Section 4 and Section 5.
\end{proof}
\begin{rem}
Fiorenza and Manetti gave two different proofs of Proposition \ref{mapping cone}: the first one is a very elegant but non-constructive proof based on the \emph{Homotopy Transfer Theorem} (see \cite{KS} and \cite{LV}, while \cite{Va} provides a gentler introduction), while the second proof relies on a careful explicit description of all the higher products defining the $L_{\infty}$-structure of $C_{\chi_{W,V}}$; anyway, we are not reporting such formulae since they are not really needed for the sake of this paper.
\end{rem}
Consider the natural transformation
\begin{eqnarray*} \label{Def_C=Grass}
\Psi_{\chi_{W,V}}:&\mathrm{Def}_{C_{\chi_{W,V}}}&\xrightarrow{\hspace*{1cm}}\mathrm{Grass}_{W,V} \\
\forall A\in\mathfrak{Art}_k\qquad\mathrm{Def}_{C_{\chi_{W,V}}}\left(A\right)\ni&\eta&\longmapsto \left(\eta\left(W\otimes A\right),d\right)\in\mathrm{Grass}_{W,V}\left(A\right).
\end{eqnarray*}
\begin{thm} \label{Grass=Def_C} \emph{(Fiorenza-Manetti)}
In the above notations, map \eqref{Def_C=Grass} is an isomorphism of deformation functors; in particular $\mathrm{Grass}_{W,V}$ is a deformation functor in the sense of Schlessinger.
\end{thm}
\begin{proof}
See \cite{FMan1} Proposition 9.2.
\end{proof}
\subsection{Cartan Homotopies and Period Maps}
The work of Fiorenza and Manetti, especially Theorem \ref{Grass=Def_C}, suggests that a good candidate for the codomain of a purely deformation-theoretic version of Griffiths $p$\textsuperscript{th} local period map should be the functor $\mathrm{Grass}_{F^pH^*\left(X,k\right),H^*\left(X,k\right)}$, where
\begin{equation*}
H^*\left(X,k\right):=\mathbb H^*\left(X,\Omega^*_{X/k}\right) 
\end{equation*}
is the algebraic De Rham cohomology of the scheme $X$ and $F^{\bullet}$ is the Hodge filtration over it. Now we are almost ready to describe the actual morphism that Fiorenza and Manetti constructed in order to translate Griffiths period map in terms of deformation functors.
\begin{defn} \label{Cart htpy} (Fiorenza-Manetti)
Let $\left(\mathfrak g,d,[-,-]\right)$ and $\left(\mathfrak l,d,[-,-]\right)$ be two differential graded Lie algebras over $k$; a linear map $i\in\mathrm{Hom}^{-1}\left(\mathfrak g,\mathfrak l\right)$ is said to be a \emph{Cartan homotopy} if 
\begin{equation*}
\forall a,b\in\mathfrak g\qquad\qquad i\left(\left[a,b\right]\right)=\left[i\left(a\right),di\left(b\right)\right]\quad\text{and}\quad\left[i\left(a\right),i\left(b\right)\right]=0.\footnote{Again, we are denoting by the same symbol the differential and the bracket of the dgla's $\mathfrak g$, $\mathfrak l$ and $\mathrm{Hom}^*\left(\mathfrak g,\mathfrak l\right)$.}
\end{equation*}
\end{defn}
\begin{rem} \label{rem Cart htpy}
The following facts directly follow from Definition \ref{Cart htpy}
\begin{enumerate}
\item The differential of a Cartan homotopy is a morphism of differential graded Lie algebras (i.e. it preserves grading and differentials); 
\item The notion of Cartan homotopy is stable under composition with a dgla map and under tensorisation with a differential graded commutative algebra; \item The notion of Cartan homotopy generalises to maps of sheaves of dgla's;
\item Let $i:\mathfrak g\rightarrow\mathfrak l\left[-1\right]$ be a Cartan homotopy and $l:\mathfrak g\rightarrow\mathfrak l$ its differential: $e^{i}$ is an homotopy between $l$ and the zero dgla morphism $0$.
\end{enumerate}
\end{rem}
\begin{example}
The contraction map associated to the scheme $X$ is a Cartan homotopy of sheaves of dgla's (see Section 1.2 for a definition in the context of complex manifolds), while its derived globalisation provides us with an honest Cartan homotopy of dgla's: the latter will be a key ingredient of this paper (see Section 3.2).
\end{example}
The reason why we are interested in Cartan homotopies is that they behave very well with respect to mapping cones.
\begin{prop} \emph{(Fiorenza-Manetti)}
In the notations of Definition \ref{Cart htpy}, let $l:=di$; then the linear map
\begin{eqnarray*}
\check i:&\mathfrak g&\xrightarrow{\hspace*{0.75cm}}C_{l} \\
&a&\longmapsto\left(a,i\left(a\right)\right)
\end{eqnarray*}
is a $L_{\infty}$-morphism; in particular it induces a morphism between the associated deformation functors.
\end{prop}
\begin{proof}
See \cite{FMan1} Proposition 7.4.
\end{proof}
Now, in the notations of Proposition \ref{mapping cone}, for all $p>0$ set
\begin{equation*} 
V:=\mathbb R\Gamma\left(X,\Omega^*_{X/k}\right)\qquad\qquad W:=F^p\mathbb R\Gamma\left(X,\Omega^*_{X/k}\right)\qquad\qquad\chi^p:=\chi_{V,W}
\end{equation*}
while denote by $i$ the contraction map associated to $X$ and by $l$ its differential, i.e. the Lie derivative.
\begin{war} \label{RGamma resol}
As we did in Warning \ref{KS resol} in the case of the Kodaira-Spencer dgla, we will always fix a specific choice of functor for $\mathbb R\Gamma\left(X,\Omega^*_{X/k}\right)$, i.e. the one given by the Dolbeaut resolution: in other words throughout the paper we will have
\begin{equation} \label{RGamma}
\mathbb R\Gamma\left(X,\Omega^*_{X/k}\right)\simeq\Gamma\left(X,\mathscr A_X^{*,*}\right).
\end{equation}
\end{war}
\begin{thm} \emph{(Fiorenza-Manetti)}
The linear map
\begin{eqnarray*}
\mathrm{fm}^p:&KS_X&\xrightarrow{\hspace*{0.75cm}}C_{\chi^p} \\
&\xi&\longmapsto\left(l_{\xi},i_{\xi}\right)
\end{eqnarray*}
is a $L_{\infty}$-morphism; in particular it induces a morphism of deformation functors
\begin{equation} \label{fm}
\mathrm{fm}^p:\mathrm{Def}_{KS_X}\xrightarrow{\hspace*{0.75cm}}\mathrm{Def}_{C_{\chi^p}}.\footnote{Here, by a slight abuse of notation, the symbol $\mathrm{fm}^p$ is denoting both the $L_{\infty}$-map and the induced morphism of deformation functors.}
\end{equation}
\end{thm}
\begin{proof}
See \cite{FMan1} Theorem 12.1.
\end{proof}
\begin{rem} \label{formal}
Recall that, as a consequence of the $E_1$-degeneration of the Hodge-to De Rham spectral sequence of $X$, the canonical inclusion of complexes $F^p\mathbb R\Gamma\left(X,\Omega^*_{X/k}\right)\hookrightarrow\mathbb R\Gamma\left(X,\Omega^*_{X/k}\right)$ descends to cohomology, i.e. the induced linear map $H^*\left(F^p\mathbb R\Gamma\left(X,\Omega^*_{X/k}\right)\right)\rightarrow H^*\left(X,k\right)$ is injective. This is equivalent to say that for all $p$ there is a quasi-isomorphism of complexes between $F^p\mathbb R\Gamma\left(X,\Omega^*_{X/k}\right)$ and $F^pH^*\left(X,k\right)$\footnote{Notice that the case $p=0$ is trivial.}, which in turn induces a quasi-isomorphism of dgla's between $\mathrm{End}^{F^p\mathbb R\Gamma\left(X,\Omega^*_{X/k}\right)}\left(\mathbb R\Gamma\left(X,\Omega^*_{X/k}\right)\right)$ and $\mathrm{End}^{F^pH^*\left(X,k\right)}\left(H^*\left(X,k\right)\right)$.
\end{rem}
Now denote $\hat{\chi}^p:=\chi_{H^*\left(X,k\right),F^pH^*\left(X,k\right)}$: Remark \ref{formal} entails in particular the existence of a homotopy equivalence of $L_{\infty}$-algebras
\begin{equation*}
h: C_{\chi^p}\longrightarrow C_{\hat{\chi}^p}
\end{equation*}
which induces, by the \emph{Basic Theorem of Deformation Theory} (see \cite{Man2}), an isomorphism 
\begin{equation} \label{h}
h: \mathrm{Def}_{C_{\chi^p}}\longrightarrow \mathrm{Def}_{C_{\hat{\chi}^p}}\footnote{Again, the symbol $h$ is denoting both the $L_{\infty}$-map and the induced morphism of deformation functors.}
\end{equation}
between the corresponding deformation functors. In the same fashion, the natural transformation
\begin{equation} \label{H Grass}
H^*:\mathrm{Grass}_{F^p\mathbb R\Gamma\left(X,\Omega^*_{X/k}\right),\mathbb R\Gamma\left(X,\Omega^*_{X/k}\right)}\xrightarrow{\hspace*{0.75cm}}\mathrm{Grass}_{F^pH^*\left(X,k\right),H^*\left(X,k\right)}.
\end{equation}
induced by the algebraic De Rham cohomology functor is an isomorphism: for a proof see \cite{FMan1} Theorem 10.6.
\begin{defn} \label{alg FM}
For all $p>0$ define the \emph{algebraic $p$\textsuperscript{th} Fiorenza-Manetti local period map} to be the morphism
\begin{equation*}
\mathrm{FM}^p:\mathrm{Def}_{KS_X}\xrightarrow{\hspace*{0.75cm}}\mathrm{Def}_{C_{\hat{\chi}^p}}
\end{equation*}
given by the composition of maps \eqref{h} and \eqref{fm}.
\end{defn}
\begin{defn} \label{geom FM}
For all $p>0$ define the \emph{geometric $p$\textsuperscript{th} Fiorenza-Manetti local period map} to be the morphism
\begin{eqnarray*} \label{FM lpm}
\mathcal P^p:\mathrm{Def}_X\qquad&\xrightarrow{\hspace*{0.5cm}}&\mathrm{Grass}_{F^pH^*\left(X\right),H^*\left(X\right)} \\
\forall A\in\mathfrak{Art}_k\qquad\mathrm{Def}_X\left(A\right)\ni\left(\mathscr O_A\overset{\xi}{\rightarrow}\mathscr O_X\right)&\longmapsto& \quad F^pH^*\left(X,\mathscr  O_A\right)\in\mathrm{Grass}_{F^pH^*\left(X\right),H^*\left(X\right)}\left(A\right).
\end{eqnarray*}
\end{defn}
Now we are finally ready to lift Griffiths period map to a morphism of deformation functors.
\begin{thm} \emph{(Fiorenza-Manetti)} \label{Fiorenza-Manetti}
There is a natural isomorphism between maps $\mathrm{FM^p}$ and $\mathcal P^p$, meaning that the diagram
\begin{equation*}
\xymatrix{\mathrm{Def}_{KS_X}\ar[rr]^{\mathrm{FM}^p}\ar[d]_{\mathscr O}^{\wr} & & \mathrm{Def}_{C_{\hat{\chi}^p}}\ar[d]^{\Psi_{\hat{\chi}^p}}_{\wr} \\
\mathrm{Def}_X\ar[rr]^{\mathcal P^p} & & \mathrm{Grass}_{F^pH^*\left(X\right),H^*\left(X\right)}}
\end{equation*}
commutes. Moreover the tangent morphism to the functor $\mathcal P^p$ is the same as map \eqref{per diff}.
\end{thm}
\begin{proof}
See \cite{FMan1} Theorem 12.3 and Corollary 12.5.
\end{proof}
\subsection{Flag Functors and the Fiorenza-Manetti Period Map}
Both $\mathrm{FM^p}$ and $\mathcal P^p$ depend on a filtration parameter: we would like to get rid of it, in order to define universal versions of the algebraic and geometric Fiorenza-Manetti period map. \\
Observe that the target functor of any universal version of the geometric Fiorenza-Manetti period map should not be simply the product of the deformation functors $\mathrm{Grass}_{F^pH^*\left(X,k\right),H^*\left(X,k\right)}$, because the only deformations of the sequence $\left(F^pH^*\left(X,k\right)\right)_p$ of subcomplexes of $H^*\left(X,k\right)$ which may belong to its image are those preserving the property that $F^{\bullet}$ is a filtration. \\
For this reason, let $\left(V,\mathcal F^{\bullet}\right)$ be a filtered complex and define the \emph{flag functor associated to $\left(V,\mathcal F^{\bullet}\right)$} to be
\begin{eqnarray*}
\mathrm{Flag}_V^{\mathcal F^{\bullet}}:&\mathfrak{Art}_k&\xrightarrow{\hspace*{4.5cm}}\mathfrak{Set} \\
&A&\mapsto\left\{\left(\left(U,\mathcal G^p\right)\right)_p\text{ s.t. }\left( U,\mathcal G^p\right)\in\mathrm{Grass}_{\mathcal F^pV,V}\left(A\right),\mathcal G^p U\hookrightarrow\mathcal G^{p-1} U\right\}.
\end{eqnarray*}
Consider the complex
\begin{equation} \label{filt-pres dgla definition}
\mathrm{End}^{\mathcal F^{\bullet}}\left(V\right):=\underset{p}{\bigcap}\mathrm{End}^{\mathcal F^pV}\left(V\right)
\end{equation}
which may be seen as the subcomplex of $\mathrm{End}\left(V\right)$ made of filtration-preserving endomorphisms. 
\begin{example}
Notice that if our base filtered complex $\left(V,\mathcal F^{\bullet}\right)$ is the algebraic De Rham complex (or cohomology) of a scheme $X$ equipped with the Hodge filtration, than the complex of filtration-preserving endomorphisms is nothing but the complex of non-negatively graded endomorphisms. In other words, in the notations of formula \eqref{filt-pres dgla definition} we have that 
\begin{equation*}
\mathrm{End}^{F^{\bullet}}\left(\mathbb R\Gamma\left(X,\Omega^*_{X/k}\right)\right)=\mathrm{End}^{\geq0}\left(\mathbb R\Gamma\left(X,\Omega^*_{X/k}\right)\right).
\end{equation*}
\end{example}
Again $\mathrm{End}^{\mathcal F^{\bullet}}\left(V\right)$ is endowed with a natural structure of differential graded Lie algebra: this comes with a natural inclusion of dgla's
\begin{equation*}
\chi:\mathrm{End}^{\mathcal F^{\bullet}}\left(V\right)\hookrightarrow\mathrm{End}\left(V\right).
\end{equation*}
and let 
\begin{equation*}
C^{\mathcal F^{\bullet}}_V:=\underset{\longleftarrow}{\mathrm{holim}}\left(\mathrm{End}^{\mathcal F^{\bullet}}\left(V\right)\doublerightarrow{\chi}{0}\mathrm{End}^*\left(V\right)\right).
\end{equation*}
be its homotopy cokernel.
\begin{prop} \label{gen und Flag dgla} \emph{(Fiorenza-Martinengo)}
In the above notations there is an isomorphism of functors
\begin{equation*}
\mathrm{Flag}_V^{\mathcal F^{\bullet}}\simeq\mathrm{Def}_{C^{\mathcal F^{\bullet}}_V}
\end{equation*}
In particular $\mathrm{Flag}_V^{\mathcal F^{\bullet}}$ is a deformation functor.
\end{prop}
\begin{proof}
See \cite{FMar} Section 5 and Section 6.
\end{proof}
Now consider the functors $\mathrm{Flag}_{\mathbb R\Gamma\left(X,\Omega^*_{X/k}\right)}^{F^{\bullet}}$ and $\mathrm{Flag}_{H^*\left(X,k\right)}^{F^{\bullet}}$: the same arguments used to deal with map \eqref{H Grass} imply that the morphism
\begin{equation} \label{flag htpy stability}
H^*:\mathrm{Flag}_{\mathbb R\Gamma\left(X,\Omega^*_{X/k}\right)}^{F^{\bullet}}\xrightarrow{\hspace*{0.75cm}}\mathrm{Flag}_{H^*\left(X,k\right)}^{F^{\bullet}}
\end{equation}
is well-defined and an isomorphism. \\
\begin{rem} \label{Martinengo}
In the language of \cite{FMar} a pair of differential graded Lie algebras $\left(\mathfrak g,\mathfrak l\right)$ is said to be a \emph{formal pair} if there is an inclusion of dgla's $\mathfrak g\hookrightarrow\mathfrak l$ inducing an injective morphism $H^*\left(\mathfrak g\right)\hookrightarrow H^*\left(\mathfrak l\right)$ on cohomology: in particular by Remark \ref{formal} we have that for any smooth proper $k$-scheme $X$ and for all $p\geq 0$ the pair $\left(\mathrm{End}^{F^p}\left(\mathbb R\Gamma\left(X,\Omega^*_{X/k}\right)\right),\mathrm{End}\left(\mathbb R\Gamma\left(X,\Omega^*_{X/k}\right)\right)\right)$ is formal. Moreover The formality argument of Remark \eqref{formal} is uniform in $p$, therefore there is a filtered quasi-isomorphism between $\left(\mathbb R\Gamma\left(X,\Omega^*_{X/k}\right),F^{\bullet}\right)$ and $\left(H^*\left(X,k\right),F^{\bullet}\right)$, providing in turn a week equivalence between the dgla's $\mathrm{End}^{\geq 0}\left(\mathbb R\Gamma\left(X,\Omega^*_{X/k}\right)\right)$ and $H^*\left(\mathcal End^{\geq 0}\left(\Omega^*_{X/k}\right)\right)$. It follows that the pair $\left(\mathrm{End}^{\geq0}\left(\mathbb R\Gamma\left(X,\Omega^*_{X/k}\right)\right),\mathrm{End}\left(\mathbb R\Gamma\left(X,\Omega^*_{X/k}\right)\right)\right)$ is formal as well.
\end{rem}
It is easy to see that if $\left(\mathfrak g,\mathfrak l\right)$ is a formal pair of dgla's then the homotopy fibre $\underset{\longleftarrow}{\mathrm{holim}}\left(\mathfrak g\doublerightarrow{\mathrm{incl.}}{0}\mathfrak l\right)$ is quasi-abelian and in fact a model for it is given by the complex $\nicefrac{\mathfrak l}{\mathfrak g}\left[-1\right]$ endowed with the trivial bracket: see \cite{FMar} Section 5 for a more detailed explanation. In particular, if we apply this to the De Rham-theoretic case we have that
\begin{eqnarray} \label{tgt dgrass} 
C_{\mathbb R\Gamma\left(X,F^p\Omega^*_{X/k}\right),\mathbb R\Gamma\left(X,\Omega^*_{X/k}\right)}&\simeq&\underset{\longleftarrow}{\mathrm{holim}}\left(\mathrm{End}^{F^p}\left(\mathbb R\Gamma\left(X,\Omega^*_{X/k}\right)\right)\doublerightarrow{\quad}{\quad}\mathrm{End}\left(\mathbb R\Gamma\left(X,\Omega^*_{X/k}\right)\right)\right) \nonumber \\
&\simeq&\frac{\mathrm{End}^*\left(\mathbb R\Gamma\left(X,\Omega^*_{X/k}\right)\right)}{\mathrm{End}^{F^p}\left(\mathbb R\Gamma\left(X,\Omega^*_{X/k}\right)\right)}\left[-1\right] \nonumber \\
&\simeq&\frac{H^*\left(\mathcal End^*\left(\Omega^*_{X/k}\right)\right)}{H^*\left(\mathcal End^{F^p}\left(\Omega^*_{X/k}\right)\right)}\left[-1\right]
\end{eqnarray}
and
\begin{eqnarray} \label{tgt hoflag}
C^{F^{\bullet}}_{\mathbb R\Gamma\left(X,\Omega^*_{X/k}\right)}&\simeq&\underset{\longleftarrow}{\mathrm{holim}}\left(\mathrm{End}^{\geq0}\left(\mathbb R\Gamma\left(X,\Omega^*_{X/k}\right)\right)\doublerightarrow{\quad}{\quad}\mathrm{End}\left(\mathbb R\Gamma\left(X,\Omega^*_{X/k}\right)\right)\right) \nonumber \\
&\simeq&\frac{\mathrm{End}^*\left(\mathbb R\Gamma\left(X,\Omega^*_{X/k}\right)\right)}{\mathrm{End}^{\geq0}\left(\mathbb R\Gamma\left(X,\Omega^*_{X/k}\right)\right)}\left[-1\right] \nonumber \\
&\simeq&\frac{H^*\left(\mathcal End^*\left(\Omega^*_{X/k}\right)\right)}{H^*\left(\mathcal End^{\geq0}\left(\Omega^*_{X/k}\right)\right)}\left[-1\right]
\end{eqnarray}
where $\mathcal End^*\left(\Omega^*_{X/k}\right)$, $\mathcal End^{\geq0}\left(\Omega^*_{X/k}\right)$ and $\mathcal End^{F^p}\left(\Omega^*_{X/k}\right)$ denote respectively the endomorphism sheaf of $\Omega^*_{X/k}$, the sheaf of its non-negatively graded endomorphisms and the sheaf of those endomorphisms preserving the $p^{\mathrm{th}}$ piece of the Hodge filtration. \\
Summing up the preceding considerations, we obtain a very explicit description of the flag functor associated to the algebraic De Rham complex of $X$.
\begin{cor} \label{und Flag dgla}
There is a chain of isomorphism of deformation functors
\begin{equation} \label{chain}
\mathrm{Flag}_{H^*\left(X,k\right)}^{F^{\bullet}}\simeq\mathrm{Flag}_{\mathrm{End}^*\left(\mathbb R\Gamma\left(X,\Omega^*_{X/k}\right)\right)}^{F^{\bullet}}\simeq\mathrm{Def}_{\frac{\mathrm{End}^*\left(\mathbb R\Gamma\left(X,\Omega^*_{X/k}\right)\right)}{\mathrm{End}^{\geq 0}\left(\mathbb R\Gamma\left(X,\Omega^*_{X/k}\right)\right)}\left[-1\right]}.
\end{equation}
\end{cor}
\begin{proof}
The first isomorphism in chain \eqref{chain} is map \eqref{flag htpy stability}, the other one follows combining Proposition \ref{gen und Flag dgla} and formula \eqref{tgt hoflag}. 
\end{proof}
Now we are ready to define some universal version of the Fiorenza-Manetti morphism.
\begin{defn} \label{univ geom Fio-Man}
Define the \emph{universal geometric Fiorenza-Manetti period map} to be the natural transformation
\begin{eqnarray*} 
\mathcal P:&\mathrm{Def}_X&\xrightarrow{\hspace*{0.75cm}}\mathrm{Flag}_{H^*\left(X,k\right)}^{F^{\bullet}} \\
\forall A\in\mathfrak{Art}_k\quad\quad&\left(\mathscr O_A\overset{\xi}{\rightarrow}\mathscr O_X\right)&\longmapsto\left(\mathcal P^p\left(\left(\mathscr O_A\overset{\xi}{\rightarrow}\mathscr O_X\right)\right)\right)_p.
\end{eqnarray*}
\end{defn}
Notice that Definition \ref{geom FM} ensures that $\mathcal P$ is a well-defined morphism of functors. \\ 
Map $\mathcal P$ is a good universal version of the geometric Fiorenza-Manetti period map; we would like to complete the picture with a natural universal version of the algebraic Fiorenza-Manetti map, that is we would like to construct a morphism of differential graded Lie algebras 
\begin{equation*}
\mathrm{FM}:KS_X\longrightarrow\frac{\mathrm{End}^*\left(\mathbb R\Gamma\left(X,\Omega^*_{X/k}\right)\right)}{\mathrm{End}^{\geq 0}\left(\mathbb R\Gamma\left(X,\Omega^*_{X/k}\right)\right)}\left[-1\right]
\end{equation*}
such that the diagram
\begin{equation*}
\xymatrix{\mathrm{Def}_{KS_X}\ar[rrr]^{\mathrm{FM}}\ar[d]^{\wr} & & & \mathrm{Def}_{\frac{\mathrm{End}^*\left(\mathbb R\Gamma\left(X,\Omega^*_{X/k}\right)\right)}{\mathrm{End}^{\geq 0}\left(\mathbb R\Gamma\left(X,\Omega^*_{X/k}\right)\right)}\left[-1\right]}\ar[d]_{\wr} \\
\mathrm{Def}_X\ar[rrr]^{\mathcal P} & & & \mathrm{Flag}_{H^*\left(X,k\right)}^{F^{\bullet}}}
\end{equation*}
commutes: we will construct it in Section 3.2.

\section{The Period Map as a Morphism of $\infty$-Groupoids}

Theorem \ref{Fiorenza-Manetti} attests two very interesting facts: the first one is that Definition \ref{alg FM} and Definition \ref{geom FM} are naturally equivalent (and this enables us to simply talk about the \emph{Fiorenza-Manetti local period map}, dropping any further adjective) and the second one is that map \eqref{FM lpm} really extends the period mapping \eqref{G lpm} to a morphism of deformation theories, as the tangent maps are the same. In this perspective, the period map is seen to play a remarkable unifying role in Deformation Theory and Hodge Theory: as a matter of fact a number of highly non-trivial classical results such as \emph{Kodaira Principle} and \emph{Bogomolov-Tian-Todorov Theorem} are recovered as corollaries of Theorem \ref{Fiorenza-Manetti} (see \cite{FMan1}, \cite{FMan3}, \cite{FMar} and \cite{IacMan} for more details). \\
Anyway the contemporary viewpoint on Deformation Theory claims that Schlessinger's deformation functors are not the most suitable tools in order to study general local moduli problems, as they are often unable to capture most of the hidden geometry of such problems. As a matter of fact Schlessinger's functors do not generally take into account automorphisms and higher autoequivalences of the objects they classify and in most cases they do not give a proper description of obstructions, either. Moreover the correspondence between differential graded Lie algebras and deformation functors in the context of classical Deformation Theory is not fully satisfying\footnote{Notice that an instance of such a drawback has already appeared in Section 2.2, since the mapping cone \eqref{cone formula} is endowed with a non-trivial $L_{\infty}$-structure.}, but the most important drawback of Schlessinger's functors for the sake of this paper is that in general they are not formal neighbourhoods of any global moduli space; this is precisely the case of the functor $\mathrm{Def}_X$ defined in Section 2.1: there does not exist any (classical) moduli space of proper smooth schemes of dimension $d>1$, thus for a general choice of the scheme $X$ the functor $\mathrm{Def}_X$ cannot be describing infinitesimally any algebraic space.
\subsection{Quick Review of Derived Deformation Theory}
The critical aspects we have briefly listed above mark some of the reasons that have been leading to the development of Derived Deformation Theory: the rough idea behind this subject is that Deformation Theory is not really a {} ``categorical'' subject, but rather an {} ``$\left(\infty,1\right)$-categorical'' one, meaning that its constructions and invariants should be homotopical (or derived) in nature. In particular the basic objects of Derived Deformation Theory should be homotopy analogues of Schlessinger's functors -- i.e. functors defined over (some subcategory of) $\mathfrak{dgArt}_k$ rather than $\mathfrak{Art}_k$ -- satisfying homotopical versions of Schlessinger's axioms and preserving the homotopical structure of the category of Artinian dg-algebras. Foundational work on Derived Deformation Theory includes \cite{Getz}, \cite{Hin2}, \cite{Kon}, \cite{Lu1}, \cite{Lu3}, \cite{Man1}, \cite{Pr1} and \cite{TV}, while a gentle introduction to the subject can be found in \cite{dN1}: here we quickly review some of the main concepts just to fix notations. \\
There are several different ways to enhance a classical deformation functor to a derived one, giving rise to various consistent derived deformation theories; in \cite{Pr1} Pridham proved that all these variants are homotopy equivalent\footnote{All approaches to Derived Deformation Theory are described by a well-defined $\left(\infty,1\right)$-category: Pridham proved that all such $\left(\infty,1\right)$-categories are equivalent; for more details see \cite{Pr1}.}, thus in this paper by derived deformation functor we will always mean a \emph{Hinich derived deformation functor}\footnote{In the literature people also refer to such functors as \emph{formal moduli problems} or \emph{formal stacks}.}. The latter is a functor 
\begin{equation*}
\mathbf F:\mathfrak{dgArt}_k^{\leq 0}\longrightarrow\mathfrak{sSet}
\end{equation*}
satisfying weaker versions of Schlessiger's axioms for classical deformation problems: for a precise definition see \cite{dN1} or the original paper \cite{Hin2}, but essentially $\mathbf F$ is required to be \emph{homotopic} -- i.e. to map quasi-isomorphisms in $\mathfrak{dgArt}_k^{\leq 0}$ to weak equivalences in $\mathfrak{sSet}$ -- and \emph{homotopy-homogeneous} -- i.e. such that for all surjections $A\twoheadrightarrow B$ and all maps $C\rightarrow B$ in $\mathfrak{dgArt}_k^{\leq 0}$ the natural map
\begin{equation*}
\mathbf F\left(A\times_B C\right)\longrightarrow\mathbf F\left(A\right)\times^{h}_{\mathbf F\left(B\right)}\mathbf F\left(C\right)\footnote{The symbol $-\times_-^h-$ denotes the homotopy fibre product in $\mathfrak{sSet}$.}
\end{equation*}
is a weak equivalence. In case $\mathbf F$ is only \emph{homotopy-surjecting} -- i.e. for all tiny acyclic extension $A\rightarrow B$ in $\mathfrak{dgArt}_k^{\leq 0}$ the induced map $\pi_0\left(\mathbf F\left(A\right)\right)\rightarrow\pi_0\left(\mathbf F\left(B\right)\right)$ is surjective -- we will say that it is a \emph{derived pre-deformation functor}. \\
All the geometry of Hinich functors is captured by certain cohomological invariants which generalise tangent spaces and obstruction theories for classical deformation functors: let us briefly recall how to construct them. Given a derived deformation functor $\mathbf F:\mathfrak{dgArt}_k^{\leq 0}\longrightarrow\mathfrak{sSet}$, consider as in \cite{Pr1} Section 1.6 the functor
\begin{eqnarray*}
\mathrm{tan}\,\mathbf F:&\mathfrak{dgVect}_k^{\leq 0}&\xrightarrow{\hspace*{0.50cm}}\mathfrak{sVect}_k \\
&V&\longmapsto\mathbf F\left(k\oplus V\right)
\end{eqnarray*}
and recall that the \emph{$j$-th generalised tangent space} of $\mathbf F$ is said to be the group
\begin{equation*}
H^j\left(\mathbf F\right):=\pi_i\left(\mathrm{tan}\,\mathbf F\left(k\left[-n\right]\right)\right)\qquad\qquad\text{where }n-i=j
\end{equation*}
and the definition is well-given because of \cite{Pr1} Corollary 1.46. Generalised tangent spaces extend the underived notions of tangent and obstruction spaces in the sense that if $\mathbf F$ is a derived deformation functor, the group $H^{j}\left(\mathbf F\right)$ parametrises infinitesimal $j$-automorphisms associated to it; in particular $H^0\left(\mathbf F\right)$ encodes first-order derived deformations and $H^1\left(\mathbf F\right)$ encodes second-order derived deformations, i.e. all obstructions (see \cite{Pr1} Section 1.6).\\
One of the properties of derived deformation functors which are most interesting to us is that they provide the right notion of formal stack, i.e. that they describe derived geometric stacks infinitesimally. Foundational work on higher stacks and Derived Algebraic Geometry includes \cite{Lu1}, \cite{Lu2}, \cite{Pr2}, \cite{Toe} and \cite{TV}: here we only recall that given a (possibly non-geometric) derived stack over $k$
\begin{equation*}
\mathcal F:\mathfrak{dgAlg}_k^{\leq 0}\rightarrow\mathfrak{sSet}
\end{equation*}
a point $x$ over it, the \emph{formal neighbourhood} of $F$ at $x$ is defined as
\begin{eqnarray*}
\hat{\mathbf F}_x:&\mathfrak{dgArt}^{\leq 0}_k&\xrightarrow{\hspace*{1cm}}\mathfrak{sSet}\\
&A&\longmapsto \mathcal F\left(A\right)\times_{\mathcal F\left(k\right)}^h\left\{x\right\}.
\end{eqnarray*}
A well-known folklore result in Derived Algebraic Geometry is that $\hat{\mathbf F}_x$ is a derived deformation functor: a proof of it is hidden somewhere in \cite{Lu1} and \cite{TV}; see also \cite{Toe} and \cite{Pr1}.\\
Now denote by $\Omega^*_{\mathrm{DR}}\left(\Delta^*\right)$ the simplicial differential graded commutative algebra of \emph{polynomial differential forms}, given in simplicial level $n$ by
\begin{equation*}
\Omega^*_{\mathrm{DR}}\left(\Delta^n\right):=\frac{k\left[x_0,x_1,\ldots,x_n,dx_0,dx_1,\ldots,dx_n\right]}{\sum x_i=1,\sum dx_i=0}
\end{equation*}
where $x_0,x_1,\ldots,x_n$ live in cochain degree $0$ and $dx_0,dx_1,\ldots,dx_n$ in cochain degree 1; more generally, given a simplicial set $S$, the symbol $\Omega^*_{\mathrm{DR}}\left(S\right)$ will stand for the simplicial differential graded commutative algebra of \emph{polynomial differential forms on} $S$, which is defined in dg level $p$ by
\begin{equation*}
\Omega^p_{\mathrm{DR}}\left(S\right):=\mathrm{Hom}_{\mathfrak{sSet}}\left(S,\Omega^p_{\mathrm{DR}}\left(\Delta^*\right)\right).
\end{equation*} 
Also recall that the \emph{Hinich nerve} of a dgla $\mathfrak g$ is defined to be the derived deformation functor
\begin{align*}
\mathbb R\mathrm{Def}_{\mathfrak g}:\mathfrak{dg}&\mathfrak{Art}^{\leq 0}_k\xrightarrow{\hspace*{1.25cm}}\mathfrak{sSet} \\
&\phantomarrow{\mathfrak{dgArt}^{\leq 0}_k}{A} \mathbb R\mathrm{MC}_{\mathfrak g\otimes\Omega^*_{\mathrm{DR}}\left(\Delta^*\right)}\left(A\right)
\end{align*}
where $\mathbb R\mathrm{MC}_{\mathfrak g\otimes\Omega^*_{\mathrm{DR}}\left(\Delta^*\right)}\left(A\right)$ is the simplicial set determined in level $n$ by the set
\begin{equation*}
\mathbb R\mathrm{MC}_{\mathfrak g\otimes\Omega^*_{\mathrm{DR}}\left(\Delta^n\right)}\left(A\right):=\left\{x\in \left(\mathfrak g\otimes\Omega^*_{\mathrm{DR}}\left(\Delta^n\right)\otimes\mathfrak m_A\right)^1\text{ s.t. } d\left(x\right)+\frac{1}{2}\left[x,x\right]=0\right\}.
\end{equation*}
\begin{thm} \label{compare def theories} \emph{(Hinich, Lurie, Pridham)}
The functor
\begin{eqnarray*}
\mathbb R\mathrm{Def}:&\mathfrak{dgLie}_k&\xrightarrow{\hspace*{0.5cm}}\mathfrak{Def}^{\mathrm{Hin}}_k \\
&\mathfrak g&\longmapsto\mathbb R\mathrm{Def}_{\mathfrak g}
\end{eqnarray*}
is an equivalence of $\left(\infty,1\right)$-categories, thus it induces an equivalence on the homotopy categories
\begin{equation*}
\mathrm{Ho}\left(\mathfrak{dgLie}_k\right)\simeq\mathrm{Ho}\left(\mathfrak{Def}^{\mathrm{Hin}}_k\right).
\end{equation*}
\end{thm}
\begin{proof}
See \cite{Pr1} Corollary 4.56.
\end{proof}
Despite its great theoretical properties, the Hinich nerve is seldom handy enough to make concrete computations. For this reason, recall that the \emph{(derived) Deligne groupoid} associated to a differential graded Lie algebra $\mathfrak g$ is defined to be the formal groupoid
\begin{eqnarray*}
\mathrm{Del}_{\mathfrak g}:&\mathfrak{dgArt}^{\leq 0}&\xrightarrow{\hspace*{1cm}}\mathfrak{Grpd} \\
&A&\longmapsto\left[\nicefrac{\widetilde{\mathrm{MC}}_{\mathfrak g}\left(A\right)}{\widetilde{\mathrm{Gg}}_{\mathfrak g}\left(A\right)}\right]
\end{eqnarray*}
where
\begin{eqnarray} \label{MC}
\widetilde{\mathrm{MC}}_{\mathfrak g}:&\mathfrak{dgArt}_k^{\leq 0}&\xrightarrow{\hspace*{3cm}}\mathfrak{Set} \nonumber \\
&A&\mapsto\left\{x\in \left(\mathfrak g\otimes\mathfrak m_A\right)^1\quad\text{s.t.}\quad d\left(x\right)+\frac{1}{2}\left[x,x\right]=0\right\}
\end{eqnarray}
\begin{eqnarray} \label{Gg}
\widetilde{\mathrm{Gg}}_{\mathfrak g}:&\mathfrak{dgArt}_k^{\leq 0}&\xrightarrow{\hspace*{1cm}}\mathfrak{Grp} \nonumber \\
&A&\mapsto\mathrm{exp}\left(\left(\mathfrak g\otimes \mathfrak m_A\right)^0\right)
\end{eqnarray}
and let
\begin{eqnarray*}
\mathrm{BDel}_{\mathfrak g}:\mathfrak{dgArt}_k^{\leq 0}\longrightarrow\mathfrak{sSet}
\end{eqnarray*}
denote its nerve.
\begin{rem}
Notice that formula \eqref{MC} and formula \eqref{Gg} are just straightforward generalisations of the notions of Maurer-Cartan and gauge functor in underived Deformation Theory; these objects are used to define \emph{extended deformation functors} in the sense of Manetti (see \cite{Man1} or \cite{Man2}). In \cite{Pr1} Pridham also proved that there is an equivalence of $\left(\infty,1\right)$-categories between $\mathfrak{Def}_k^{\mathrm{Man}}$ and $\mathfrak{Def}_k^{\mathrm{Hin}}$.
\end{rem}
\begin{war}
The nerve of the Deligne groupoid associated to a differential graded Lie algebra is a derived pre-deformation functor but not a derived deformation functor: as a matter of fact it is not homotopic in general. Moreover, although it might be a bit confusing, we will tend to refer to both $\mathrm{Del}_{\mathfrak g}$ and $\mathrm{BDel}_{\mathfrak g}$ as the Deligne groupoid associated to the differential graded Lie algebra $\mathfrak g$.
\end{war}
Fix $\mathfrak g\in\mathfrak{dgLie}_k$: we can define the functor
\begin{eqnarray*}
&\underline{\mathrm{BDel}}_{\mathfrak g}:\mathfrak{dgArt}_k^{\leq 0}\xrightarrow{\hspace*{5cm}}\mathfrak{sSet}\qquad\qquad\qquad\qquad\qquad\qquad\qquad\qquad& \\
&\qquad\qquad\; A\longmapsto\mathrm{diag}\left(\xymatrix{\mathrm{BDel}_{\mathfrak g}\left(A\right)\ar[r] & \mathrm{BDel}_{\mathfrak g\otimes\Omega^*\left(\Delta^1\right)}\left(A\right)\ar@<2.5pt>[l]\ar@<-2.5pt>[l]\ar@<2.5pt>[r]\ar@<-2.5pt>[r] & \mathrm{BDel}_{\mathfrak g\otimes\Omega^*\left(\Delta^2\right)}\left(A\right)\ar[l]\ar@<5pt>[l]\ar@<-5pt>[l]\ar[r]\ar@<5pt>[r]\ar@<-5pt>[r] & \cdots\ar@<-7.5pt>[l]\ar@<-2.5pt>[l]\ar@<2.5pt>[l]\ar@<7.5pt>[l]}\right)&
\end{eqnarray*}
which is sometimes called the \emph{simplicial Deligne groupoid} of $\mathfrak g$.
\begin{thm} \label{Del grpd} \emph{(Pridham)}
Let $\mathfrak g$ be a differential graded Lie algebra concentrated in non-negative degrees; we have that
\begin{itemize}
\item the functor $\underline{\mathrm{BDel}}_{\mathfrak g}$ is a derived deformation functor;
\item the functor $\underline{\mathrm{BDel}}_{\mathfrak g}$ is the universal derived deformation functor under $\mathrm{BDel}_{\mathfrak g}$;
\item the functors $\underline{\mathrm{BDel}}_{\mathfrak g}$ and $\mathbb R\mathrm{Def}_{\mathfrak g}$ are weakly equivalent.
\end{itemize}
\end{thm}
\begin{proof}
See \cite{Pr4} Section 3 for the proof of the first two claims, while the last statement is proved in \cite{Hin4} Section 3.
\end{proof}
As a consequence of Theorem \ref{Del grpd} we have that all geometric and homotopy-theoretic information concerning the Hinich nerve of a differential graded Lie algebra $\mathfrak g$ are completely determined by its associated Deligne groupoid, which is a much more down-to-earth object as it is essentially a formal groupoid. Unfortunately, as $\mathrm{BDel}_{\mathfrak g}$ does not map quasi-isomorphisms to weak equivalences, the description of higher tangent spaces we gave above in this section is no longer valid; nonetheless Pridham found a coherent way to define good cohomology theories for derived pre-deformation functors. As a matter of fact fix a derived pre-deformation functor $\mathbf F:\mathfrak{dgArt}_k^{\leq 0}\rightarrow\mathfrak{sSet}$ and define as in \cite{Pr2bis} Section 3.3 the $j$-th generalised tangent space of $\mathbf F$ to be 
\begin{equation*}
H^j\left(\mathbf F\right):=
\begin{cases} 
\pi_{-j}\left(\mathbf F\left(\frac{k\left[\varepsilon\right]}{\left(\varepsilon^2\right)}\right)\right) & \text{ if }j\leq 0 \\
\nicefrac{\pi_0\left(\mathrm{tan}\left(\mathbf F\left(k\left[i\right]\right)\right)\right)}{\pi_0\left(\mathrm{tan}\left(\mathbf F\left(\mathrm{cone}\left(k\left[i\right]\right)\right)\right)\right)} & \text{ otherwise } 
\end{cases}
\end{equation*}
which is seen to be consistent with the definition given above in this section in case $\mathbf F$ is also homotopic (see \cite{Pr2bis} Lemma 3.15). \\ Now fix $\mathfrak g$ to be a differential graded Lie algebra over $k$ concentrated in non-negative degrees and apply the above definitions to its Deligne groupoid. We have that
\begin{eqnarray*}
&H^{-1}\left(\mathrm{BDel}_{\mathfrak g}\right)=\pi_1\left(\mathrm{BDel}_{\mathfrak g}\left(\frac{k\left[\varepsilon\right]}{\left(\varepsilon^2\right)}\right)\right)=\pi_1\left(\left[\nicefrac{\widetilde{\mathrm{MC}}_{\mathfrak g}\left(\frac{k\left[\varepsilon\right]}{\left(\varepsilon^2\right)}\right)}{\widetilde{\mathrm{Gg}}_{\mathfrak g}\left(\frac{k\left[\varepsilon\right]}{\left(\varepsilon^2\right)}\right)}\right]\right)=& \\
&\pi_1\left(\left[\nicefrac{\mathrm{MC}_{\mathfrak g}\left(\frac{k\left[\varepsilon\right]}{\left(\varepsilon^2\right)}\right)}{\mathrm{Gg}_{\mathfrak g}\left(\frac{k\left[\varepsilon\right]}{\left(\varepsilon^2\right)}\right)}\right]\right)\simeq\mathrm{Stab}_{\mathrm{Gg}_{\mathfrak g}\left(\frac{k\left[\varepsilon\right]}{\left(\varepsilon^2\right)}\right)}\left(0\right)&
\end{eqnarray*}
but
\begin{eqnarray*}
&\mathrm{MC}_{\mathfrak g}\left(\frac{k\left[\varepsilon\right]}{\left(\varepsilon^2\right)}\right)=\left\{x\otimes\varepsilon\in\mathfrak g^1\otimes\frac{k\left[\varepsilon\right]}{\left(\varepsilon^2\right)}\text{ s.t. }d\left(x\right)=0\right\}=Z^1\left(\mathfrak g\right)\varepsilon& \\
&\mathrm{Gg}_{\mathfrak g}\left(\frac{k\left[\varepsilon\right]}{\left(\varepsilon^2\right)}\right)=\mathrm{exp}\left(\mathfrak g^0\otimes\left(\varepsilon\right)\right)\simeq\mathrm{Id}+\mathfrak g^0\varepsilon& 
\end{eqnarray*}
and notice that the gauge action just reduces to
\begin{eqnarray} \label{gauge level 0}
\mathrm{Id}+\mathfrak g^0\varepsilon\times Z^1\left(\mathfrak g\right)\varepsilon&\overset{*}{\xrightarrow{\hspace*{1cm}}}& \quad Z^1\left(\mathfrak g\right)\varepsilon \nonumber \\
\left(\mathrm{Id}+a\varepsilon,x\varepsilon\right)\quad&\longmapsto& \left(x+d\left(a\right)\right)\varepsilon
\end{eqnarray}
therefore we get
\begin{eqnarray*}
&\mathrm{Stab}_{\mathrm{Gg}_{\mathfrak g}\left(\frac{k\left[\varepsilon\right]}{\left(\varepsilon^2\right)}\right)}\left(0\right)=\left\{\left(\mathrm{Id}+a\varepsilon\right)\in\mathrm{Id}+\mathfrak g^0\varepsilon\,\text{ s.t. }\left(\mathrm{Id}+a\varepsilon\right)*0=0\right\}\simeq& \\
&\left\{a\in\mathfrak g^0\text{ s.t. }d\left(a\right)=0\right\}=Z^0\left(\mathfrak g\right)\simeq H^0\left(\mathfrak g\right)&
\end{eqnarray*}
where the last identification follows from the fact that $\mathfrak g$ lives in non-negative degrees. \\ 
Similarly we see that
\begin{eqnarray*}
&H^0\left(\mathrm{BDel}_{\mathfrak g}\right)=\pi_0\left(\mathrm{BDel}_{\mathfrak g}\left(\frac{k\left[\varepsilon\right]}{\left(\varepsilon^2\right)}\right)\right)=&\\
&\pi_0\left(\left[\nicefrac{\widetilde{\mathrm{MC}}_{\mathfrak g}\left(\frac{k\left[\varepsilon\right]}{\left(\varepsilon^2\right)}\right)}{\widetilde{\mathrm{Gg}}_{\mathfrak g}\left(\frac{k\left[\varepsilon\right]}{\left(\varepsilon^2\right)}\right)}\right]\right)= \nicefrac{\mathrm{MC}_{\mathfrak g}\left(\frac{k\left[\varepsilon\right]}{\left(\varepsilon^2\right)}\right)}{\mathrm{Gg}_{\mathfrak g}\left(\frac{k\left[\varepsilon\right]}{\left(\varepsilon^2\right)}\right)}\simeq\nicefrac{Z^1\left(\mathfrak g\right)\varepsilon}{\mathrm{Id}+\mathfrak g^0\varepsilon}&
\end{eqnarray*}
thus the quotient of $\mathrm{MC}_{\mathfrak g}\left(\frac{k\left[\varepsilon\right]}{\left(\varepsilon^2\right)}\right)$ under the gauge action \eqref{gauge level 0} is given by $H^1\left(\mathfrak g\right)$. \\
At last observe that $\mathrm{cone}\left(k\oplus k\left[j-1\right]\right)$ is a path object for $k\oplus k\left[j\right]$, so the same kind of computation gives us that for all $j\geq 0$
\begin{eqnarray*}
&H^j\left(\mathrm{BDel}_{\mathfrak g}\right)=\nicefrac{\pi_0\left(\mathrm{tan}\left(\mathrm{BDel}_{\mathfrak g}\left(k\left[i\right]\right)\right)\right)}{\pi_0\left(\mathrm{tan}\left(\mathrm{BDel}_{\mathfrak g}\left(\mathrm{cone}\left(k\left[i\right]\right)\right)\right)\right)}=&\\
&\nicefrac{Z^0\left(\mathfrak g\otimes\left(k\oplus k\left[j\right]\right)\right)}{Z^0\left(\mathfrak g\otimes\left(k\oplus\mathrm{cone}\left(k\oplus k\left[j-1\right]\right)\right)\right)}=\nicefrac{Z^{j+1}\left(\mathfrak g\right)}{\mathfrak g^j}&
\end{eqnarray*}
and again $\mathfrak g^j$ acts on $Z^{j+1}\left(\mathfrak g\right)$ by differentials, so the quotient is $H^{j+1}\left(\mathfrak g\right)$.
\begin{rem} \label{Del grpd coho}
Let $\mathfrak g$ be any differential graded Lie algebra; by combining Theorem \ref{Del grpd} and the above observations we have that  
\begin{eqnarray*}
&H^i\left(\mathbb R\mathrm{Def}_{\mathfrak g}\right)\simeq H^i\left(\mathrm{BDel}_{\mathfrak g}\right)=\overset{\overset{\quad d}{\quad\curvearrowleft}}{\nicefrac{Z^{i+1}}{g^i}}\simeq H^{i+1}\left(\mathfrak g\right)&\qquad\forall i\geq 0. \\
&H^{-1}\left(\mathbb R\mathrm{Def}_{\mathfrak g}\right)\simeq H^{-1}\left(\mathrm{BDel}_{\mathfrak g}\right)\simeq\mathrm{Stab}_{\mathrm{Gg}_{\mathfrak g^0}\left(\frac{k\left[\varepsilon\right]}{\left(\varepsilon^2\right)}\right)}\left(0\right)\simeq H^0\left(\mathfrak g\right)&
\end{eqnarray*}
\end{rem}
\subsection{The Algebraic Fiorenza-Manetti-Martinengo Period Map}
The above considerations give many motivations to try to lift the period map from a morphism of classical deformation functors to the context of Derived Deformation Theory; Fiorenza and Martinengo started to address such a question, tackling it from an entirely algebraic viewpoint. \\
Let $X$ still be a proper smooth scheme of dimension $d$ over a field $k$ of characteristic $0$ and, again, take the Cartan homotopy defined by the contraction of differential forms with vector fields
\begin{equation*}
i:\mathscr T_X\longrightarrow\mathcal End^*\left(\Omega^*_{X/k}\right)\left[-1\right]
\end{equation*}
and the Lie derivative
\begin{equation*}
l:\mathscr T_X\longrightarrow\mathcal End^*\left(\Omega^*_{X/k}\right)
\end{equation*}
which corresponds to the differential of $i$ in the $\mathrm{Hom}$ complex. Now consider the linear map of dgla's
\begin{equation*}
\tilde i:KS_X\simeq\mathbb R\Gamma\left(X,\mathscr T_X\right)\longrightarrow\mathrm{End}^*\left(\mathbb R\Gamma\left(X,\Omega^*_{X/k}\right)\right)\left[-1\right]
\end{equation*}
defined as the composition of $\mathbb R\Gamma\left(X,i\right)$ with the map 
\begin{equation*}
\mathbb R\Gamma\left(X,\mathcal End^*\left(\Omega^*_{X/k}\right)\right)\rightarrow\mathrm{End}^*\left(\mathbb R\Gamma\left(X,\Omega^*_{X/k}\right)\right)
\end{equation*}
induced by the action of derived global sections of the endomorphism sheaf of $\Omega^*_{X/k}$ on derived global sections of $X$: this is still a Cartan homotopy; denote by $\tilde l$ the associated morphism of dgla's, which is essentially the derived globalisation of the Lie derivative. Recall that $\mathrm{End}^{\geq 0}\left(\mathbb R\Gamma\left(X, \Omega^*_{X/k}\right)\right)$ is the differential graded Lie subalgebra of $\mathrm{End}^*\left(\mathbb R\Gamma\left(X,\Omega^*_{X/k}\right)\right)$ consisting of non-negatively graded endomorphisms of the (derived global sections of the) algebraic De Rham complex and that this is the same as the subalgebra of endomorphisms preserving the Hodge filtration. Notice also that the image of $\tilde l$ is contained in $\mathrm{End}^{\geq 0}\left(\mathbb R\Gamma\left(X,\Omega^*_{X/k}\right)\right)$, therefore in the end there is a diagram of dgla's
\begin{equation*}
\xymatrix{KS_X\simeq\mathbb R\Gamma\left(X,\mathscr T_X\right)\ar[r]^{\tilde l} & \mathrm{End}^{\geq 0}\left(\mathbb R\Gamma\left(X,\Omega^*_{X/k}\right)\right)\ar@<0.75ex>[r]^{\text{incl.}}\ar@<-0.75ex>[r]_0 & \mathrm{End}^*\left(\mathbb R\Gamma\left(X,\Omega^*_{X/k}\right)\right)}
\end{equation*}
where $0$ stands for the zero map. Since $\tilde i$ is a Cartan homotopy, by Remark \ref{rem Cart htpy} $e^{\tilde i}$ gives an homotopy between $\tilde l$ and the zero map, thus there is an induced morphism of dgla's to the homotopy fibre
\begin{equation} \label{FMM dgla}
KS_X\overset{\left(\tilde l,e^{\tilde i}\right)}{\xrightarrow{\hspace*{1cm}}}\underset{\longleftarrow}{\mathrm{holim}}\left(\mathrm{End}^{\geq 0}\left(\mathbb R\Gamma\left(X,\Omega^*_{X/k}\right)\right)\doublerightarrow{\quad}{\quad}\mathrm{End}^*\left(\mathbb R\Gamma\left(X,\Omega^*_{X/k}\right)\right)\right)
\end{equation}
as observed in \cite{FMar} Section 6; moreover, recall from formula \eqref{tgt hoflag} that a model for the above homotopy fibre is the abelian dgla
\begin{equation*}
\frac{\mathrm{End}^*\left(\mathbb R\Gamma\left(X,\Omega^*_{X/k}\right)\right)}{\mathrm{End}^{\geq 0}\left(\mathbb R\Gamma\left(X,\Omega^*_{X/k}\right)\right)}\left[-1\right].
\end{equation*}
\begin{defn} \label{alg Fio-Man}
Define the \emph{universal algebraic Fiorenza-Manetti local period map} to be the morphism of deformation functors
\begin{equation*}
\mathrm{FM}:\mathrm{Def}_{KS_X}\xrightarrow{\hspace*{1cm}}\mathrm{Def}_{\frac{\mathrm{End}^*\left(\mathbb R\Gamma\left(X,\Omega^*_{X/k}\right)\right)}{\mathrm{End}^{\geq 0}\left(\mathbb R\Gamma\left(X,\Omega^*_{X/k}\right)\right)}\left[-1\right]}
\end{equation*}
induced by map \eqref{FMM dgla}.
\end{defn}
\begin{defn} \label{alg Fio-Man-Mar}
Define the \emph{(universal) algebraic Fiorenza-Manetti-Martinengo local period map} to be the morphism of derived deformation functors
\begin{equation*}
\mathrm{FMM}:\mathbb R\mathrm{Def}_{KS_X}\xrightarrow{\hspace*{1cm}}\mathbb R\mathrm{Def}_{\frac{\mathrm{End}^*\left(\mathbb R\Gamma\left(X,\Omega^*_{X/k}\right)\right)}{\mathrm{End}^{\geq 0}\left(\mathbb R\Gamma\left(X,\Omega^*_{X/k}\right)\right)}\left[-1\right]}
\end{equation*}
induced by map \eqref{FMM dgla}.
\end{defn}
In Section 2.4 we described a universal version of the geometric period map (see Definition \ref{univ geom Fio-Man}), but we did not construct its Lie-theoretic counterpart: Fiorenza and Martinengo showed that this is precisely given by map \eqref{FMM dgla}.
\begin{thm} \label{Fiorenza-Martinengo} \emph{(Fiorenza-Martinengo)} 
The diagram
\begin{equation*}
\xymatrix{\mathbb R\mathrm{Def}_{KS_X}\ar[rrr]^{\mathrm{FMM}}\ar[d]_{\pi^0\pi_{\leq 0}} & & & \mathbb R\mathrm{Def}_{\frac{\mathrm{End}^*\left(\mathbb R\Gamma\left(X,\Omega^*_{X/k}\right)\right)}{\mathrm{End}^{\geq 0}\left(\mathbb R\Gamma\left(X,\Omega^*_X/k\right)\right)}\left[-1\right]}\ar[d]^{\pi^0\pi_{\leq 0}} \\
\mathrm{Def}_{KS_X}\ar[rrr]^{\mathrm{FM}}\ar[d]^{\wr} & & & \mathrm{Def}_{\frac{\mathrm{End}^*\left(\mathbb R\Gamma\left(X,\Omega^*_{X/k}\right)\right)}{\mathrm{End}^{\geq 0}\left(\mathbb R\Gamma\left(X,\Omega^*_{X/k}\right)\right)}\left[-1\right]}\ar[d]_{\wr} \\
\mathrm{Def}_X\ar[rrr]^{\mathcal P} & & & \mathrm{Flag}_{H^*\left(X,k\right)}^{F^{\bullet}}}
\end{equation*}
is well-defined and commutes.
\end{thm}
\begin{proof}
See \cite{FMar} Section 6.
\end{proof}
\subsection{Affine DG$_{\geq 0}$-Categories and the Dold-Kan Correspondence}
Theorem \ref{Fiorenza-Martinengo} says that morphism $\mathrm{FMM}$ is the correct derived enhancement of the universal Fiorenza-Manetti local period map; however the geometric interpretation of such a result is somehow indirect, thus it would be worth to find an equivalent morphism of derived deformation functors having more evident geometric meaning. Of course the key step in order to do this consists of finding the right domain and codomain for such a morphism, i.e. defining two derived deformation functors $\mathbb R\mathrm{Def}_X$ and $\mathrm{hoFlag}_{\mathbb R\Gamma\left(X,\Omega^*_{X/k}\right)}^{F^{\bullet}}$ such that
\begin{itemize}
\item $\mathbb R\mathrm{Def}_X$ is weakly equivalent to $\mathbb R\mathrm{Def}_{KS_X}$ and similarly $\mathrm{hoFlag}_{\mathbb R\Gamma\left(X,\Omega^*_{X/k}\right)}^{F^{\bullet}}$ is weakly equivalent to $\mathbb R\mathrm{Def}_{\frac{\mathrm{End}^*\left(\mathbb R\Gamma\left(X,\Omega^*_{X/k}\right)\right)}{\mathrm{End}^{\geq 0}\left(\mathbb R\Gamma\left(X,\Omega^*_{X/k}\right)\right)}\left[-1\right]}$;
\item $\mathbb R\mathrm{Def}_X$ and $\mathrm{hoFlag}_{\mathbb R\Gamma\left(X,\Omega^*_{X/k}\right)}^{F^{\bullet}}$ are derived enhancements of $\mathrm{Def}_X$ and $\mathrm{Flag}_{H^*\left(X,k\right)}^{F^{\bullet}}$, respectively.
\end{itemize}
In order to construct such functors we need some homotopy-theoretic background. 
\begin{war}
In this section we will deal with non-negatively graded differential graded chain structures rather than non-positively graded cochain ones, though the pictures they provide are largely equivalent; the reason for this lies in the fact that -- at least in the framework of this paper -- the codomain of a derived deformation functor is the simplicial model category of simplicial sets, which is more directly related to chain structures than cochain ones.
\end{war}
First of all, recall that the \emph{normalisation} of a simplicial $k$-vector space $\left(V_{\bullet},\partial_i,\sigma_j\right)$ is defined to be the non-negatively graded chain complex of $k$-vector spaces $\left(\mathbf NV,\delta\right)$ where
\begin{equation} \label{norm}
\left(\mathbf NV\right)_n:=\bigcap_i\mathrm{ker}\left(\partial_i:V_n\rightarrow V_{n-1}\right)
\end{equation}
and $\delta_n:=\left(-1\right)^n\partial_n$. Moreover, given a map $f:V_{\bullet}\rightarrow W_{\bullet}$ of simplicial $k$-vector spaces, we can define the chain map
\begin{equation*}
\mathbf N\left(f\right):\mathbf NV_{\bullet}\longrightarrow\mathbf NW_{\bullet}
\end{equation*}
identified by the relation $\mathbf N\left(f\right)_n:=f_n|_{\mathbf NV_n}$; notice that this construction gives us a well-defined morphism of chain complexes. In the end, there is a normalisation functor
\begin{equation*}
\mathbf N:\mathfrak{sVect}_k\longrightarrow\mathfrak{Ch}_{\geq 0}\left(\mathfrak{Vect}_k\right).
\end{equation*}
at our disposal. \\
On the other hand, let $V$ be a chain complex of $k$-vector spaces and recall that its \emph{denormalisation} is defined to be the simplicial vector space $\left(\left(\mathbf KV\right)_{\bullet},\partial_i,\sigma_j\right)$ given in level $n$ by the vector space 
\begin{equation*}
\left(\mathbf KV\right)_n:=\underset{\eta\text{ surjective}}{\underset{\eta\in\mathrm{Hom}_{\Delta}\left(\left[p\right],\left[n\right]\right)}{\prod}}V_p\left[\eta\right]\qquad\qquad\qquad \left(V_p\left[\eta\right]\simeq V_p\right).
\end{equation*}
\begin{rem} \label{denorm}
Notice that
\begin{equation*}
\left(\mathbf KV\right)_n\simeq V_0\oplus V_1^{\oplus n}\oplus V_2^{\oplus\binom{n}{2}}\oplus\cdots\oplus V_k^{\oplus\binom{n}{k}}\oplus\cdots\oplus V_n^{\oplus\binom{n}{n}}.
\end{equation*}
\end{rem}
In order to complete the definition of the denormalisation of $V$ we need to define face and degeneracy maps: we will describe a combinatorial procedure to determine all of them. For all morphisms $\alpha:\left[m\right]\rightarrow\left[n\right]$ in $\Delta$, we want to define a linear map $\mathbf K\left(\alpha\right):\left(\mathbf KV\right)_n\rightarrow\left(\mathbf KV\right)_m$; this will be done by describing all restrictions $\mathbf K\left(\alpha,\eta\right):V_p\left[\eta\right]\rightarrow\left(\mathbf KV\right)_m$, for any surjective non-decreasing map $\eta\in\mathrm{Hom}_{\Delta}\left(\left[p\right],\left[n\right]\right)$. \\
For all such $\eta$, take the composite $\eta\circ\alpha$ and consider its epi-monic factorisation\footnote{The existence of such a decomposition is one of the key properties of the category $\Delta$.} $\epsilon\circ\eta'$, as in the diagram
\begin{equation*}
\xymatrix{\left[m\right]\ar[r]^{\alpha}\ar[d]_{\eta'} & \left[n\right]\ar[d]^{\eta} \\
\left[q\right]\ar[r]^{\epsilon} & \left[p\right].}
\end{equation*}
Now
\begin{itemize}
\item if $p=q$ (in which case $\epsilon$ is just the identity map), then set $\mathbf K\left(\alpha,\eta\right)$ to be the natural identification of $V_p\left[\eta\right]$ with the summand $V_p\left[\eta'\right]$ in $\left(\mathbf KV\right)_m$;
\item if $p=q+1$ and $\epsilon$ is the unique injective non-decreasing map from $\left[p\right]$ to $\left[p+1\right]$ whose image misses $p$, then set $\mathbf K\left(\alpha,\eta\right)$ to be the differential $d_p:V_p\rightarrow V_{p-1}$;
\item in all other cases set $\mathbf K\left(\alpha,\eta\right)$ to be the zero map.
\end{itemize}
The above constructions determine the whole of the simplicial vector space $\left(\left(\mathbf KV\right)_{\bullet},\partial_i,\sigma_j\right)$. As done for normalisation, for any chain map $f:V\rightarrow W$ we can define a morphism of simplicial $k$-vector spaces
\begin{equation*}
\mathbf K\left(f\right):\mathbf KV\longrightarrow\mathbf KW
\end{equation*}
by setting
\begin{eqnarray*}
&V_0\times V_1^{\oplus n}\times V_2^{\oplus\binom{n}{2}}\times\cdots\times V_n^{\oplus\binom{n}{n}}&\overset{\mathbf K\left(f\right)_n}{\xrightarrow{\hspace*{1cm}}}\qquad\;\; W_0\times W_1^{\oplus n}\times W_2^{\oplus\binom{n}{2}}\times\cdots\times W_n^{\oplus\binom{n}{n}} \\
&\left(v_0,\left(v_1^i\right)_i,\left(v_2^j\right)_j,\ldots,v_n\right)&\longmapsto\qquad\left(f_0\left(v_0\right),\left(f_1\left(v_1^i\right)\right)_i,\left(f_2\left(v_2^j\right)\right)_j,\ldots,f_n\left(v_n\right)\right).
\end{eqnarray*}
Again, there is a denormalisation functor 
\begin{equation*}
\mathbf K:\mathfrak{Ch}_{\geq 0}\left(\mathfrak{Vect}_k\right)\longrightarrow\mathfrak{sVect}_k.
\end{equation*}
at our disposal.
\begin{thm} \label{Dold-Kan} \emph{(Dold, Kan)}
The functors $\mathbf N$ and $\mathbf K$ form an equivalence of categories between $\mathfrak{sVect}_k$ and $\mathfrak{Ch}_{\geq 0}\left(\mathfrak{Vect}_k\right)$. 
\end{thm}
\begin{proof}
See \cite{GJ} Corollary 2.3 or \cite{We} Theorem 8.4.1.
\end{proof}
The Dold-Kan correspondence described in Theorem \ref{Dold-Kan} is known to induce a number of very interesting $\infty$-equivalences; for instance the Eilenberg-Zilber shuffle product and the Alexander-Whitney map, which we will discuss in more details later in this section, allow us to extend normalisation and denormalisation to a pair of functors
\begin{equation*}
\mathbf N:\mathfrak{sAlg}_k\rightleftarrows\mathfrak{dg}_{\geq 0}\mathfrak{Alg}_k:\mathbf K
\end{equation*}
which is seen to be a Quillen equivalence. Moreover recall that 
\begin{itemize}
\item a $\mathrm{dg}_{\geq 0}$-category over $k$ is a category enriched in $\mathfrak{Ch}_{\geq 0}\left(\mathfrak{Vect}_k\right)$;
\item a $k$-simplicial category is a category enriched in $\mathfrak{sVect}_k$;
\item a simplicial category is a category enriched in $\mathfrak{sSet}$;
\item a simplicial groupoid is a simplicial object in $\mathfrak{Grpd}$: equivalently a simplicial groupoid is a simplicial category in which all $1$-morphisms are invertible.
\end{itemize}
All the above structures form well-understood model categories; furthermore it is well-known in the homotopy-theoretic folklore that Theorem \ref{Dold-Kan} induces a Quillen equivalence
\begin{equation} \label{DKcat}
\mathbf N:\mathfrak{sCat}_k\rightleftarrows\mathfrak{dg}_{\geq 0}\mathfrak{Cat}_k:\mathbf K.\footnote{There is some abuse of notation in this statement.}
\end{equation}
Tabuada also constructed an explicit Quillen equivalence between $\mathfrak{dg}_{\geq 0}\mathfrak{Cat}_k$ and $\mathfrak{sCat}$ (see \cite{Tab}, where an explicit proof of formula \eqref{DKcat} can be found, as well). \\
We will use slightly more general versions of the Dold-Kan correspondence provided by Theorem \ref{Dold-Kan} and its corollaries, so we need to develop a few technical tools. \\
Define $\mathfrak{Aff}_k$ to be the category whose objects are $k$-vector spaces and whose morphisms are affine maps between $k$-vector spaces, i.e.
\begin{equation*}
\mathrm{Hom}_{\mathfrak{Aff}_k}\left(V,W\right):=\left\{v\mapsto f\left(v\right)+b\text{ s.t. }f\text{ linear},b\in W\right\}\simeq\mathrm{Hom}_{\mathfrak{Vect}_k}\left(V,W\right)\times W.
\end{equation*}
$\mathfrak{Aff}_k$ can be thought of as the category of affine spaces over $k$ and affine maps. Given $V,W\in\mathfrak{Aff}_k$, define their tensor product to be
\begin{equation} \label{tens aff}
V\tilde{\otimes}W:=V\oplus W\oplus\left(V\otimes W\right)
\end{equation}
where the tensor product $V\otimes W$ in the right-hand side of formula \eqref{tens aff} is just the tensor product as vector spaces; in a similar way, given two affine maps
\begin{eqnarray*}
&\phi\in\mathrm{Hom}_{\mathfrak{Aff}_k}\left(V,W\right)\quad\text{where}\quad\phi\left(v\right):=f\left(v\right)+b& \\ 
&\psi\in\mathrm{Hom}_{\mathfrak{Aff}_k}\left(U,Z\right)\quad\text{where}\quad\psi\left(u\right):=g\left(u\right)+d&
\end{eqnarray*}
the tensor product map is given by
\begin{eqnarray} \label{und tens morph aff}
\phi\otimes\psi:&V\oplus W\oplus \left(V\otimes W\right)&\longrightarrow\quad U\oplus Z\oplus \left(U\otimes Z\right)\nonumber \\
&\left(u,v,x\otimes y\right)&\mapsto\left(u+b,v+d,f\left(x\right)\otimes g\left(y\right)\right).
\end{eqnarray}
Formula \eqref{tens aff} and formula \eqref{und tens morph aff} determine a monoidal structure on $\mathfrak{Aff}_k$: we will be more precise about this a little bit further in this section, when dealing with dg$_{\geq 0}$-affine spaces.
\begin{defn} \label{quasi dg vect}
Define a \emph{(chain) differential graded affine space} over $k$ in non-negative degrees (\emph{$dg_{\geq 0}$-affine space} for short) to be a pair $\left(A_0,V\right)$ where
\begin{equation*}
V:\qquad V_0\overset{d}{\longleftarrow} V_1\overset{d}{\longleftarrow} V_2\overset{d}{\longleftarrow}\cdots\overset{d}{\longleftarrow} V_n\overset{d}{\longleftarrow}
\end{equation*}
is a non-negatively graded chain complex of $k$-vector spaces and $A_0$ is an affine space over $k$ whose difference vector space is $V_0$.
\end{defn}
If $\left(A_0,V\right)$ and $\left(B_0,W\right)$ are $\mathrm{dg}_{\geq 0}$-affine spaces over $k$, a morphism $\phi:\left(A_0,V\right)\rightarrow \left(B_0,W\right)$ will be a chain map which is affine in degree $0$ and linear in higher degrees: more formally the set of morphisms between $\left(A_0,V\right)$ and $\left(B_0,W\right)$ is defined to be
\begin{equation*}
\mathrm{Hom}_{\mathfrak{Ch}_{\geq 0}\left(\mathfrak{Aff}_k\right)}\left(\left(A_0,V\right),\left(B_0,W\right)\right):=\left\{\underline v\mapsto f\left(\underline v\right)+b\text{ s.t. }f\in\mathrm{Hom}_{\mathfrak{Ch}_{\geq 0}\left(\mathfrak{Vect}_k\right)}\left(V,W\right),b\in W_0\right\}.
\end{equation*}
In the end we have a well-defined category of $\mathrm{dg}_{\geq 0}$-affine spaces over $k$, which we will denote as $\mathfrak{Ch}_{\geq 0}\left(\mathfrak{Aff}_k\right)$.
\begin{rem}
We have defined the objects of $\mathfrak{Ch}_{\geq 0}\left(\mathfrak{Aff}_k\right)$ as pairs where the first term is an affine space and the second term is a chain complex of vector spaces just to make the affine structure explicit; an equivalent and more compact characterisation of $\mathfrak{Ch}_{\geq 0}\left(\mathfrak{Aff}_k\right)$ is 
\begin{eqnarray*}
&\mathrm{Ob}\left(\mathfrak{Ch}_{\geq 0}\left(\mathfrak{Aff}_k\right)\right):=\mathrm{Ob}\left(\mathfrak{Ch}_{\geq 0}\left(\mathfrak{Vect}_k\right)\right)& \\
&\mathrm{Hom}_{\mathfrak{Ch}_{\geq 0}\left(\mathfrak{Aff}_k\right)}\left(\left(A_0,V\right),\left(B_0,W\right)\right)\simeq\mathrm{Hom}_{\mathfrak{Ch}_{\geq 0}\left(\mathfrak{Vect}_k\right)}\left(V,W\right)\times W_0.&
\end{eqnarray*}
In particular $\mathfrak{Ch}_{\geq 0}\left(\mathfrak{Aff}_k\right)$ is a $k$-linear category.
\end{rem}
The category $\mathfrak{Ch}_{\geq 0}\left(\mathfrak{Aff}_k\right)$ is both complete and cocomplete: limits and colimits are constructed from those in $\mathfrak{Ch}_{\geq 0}\left(\mathfrak{Vect}_k\right)$. For example if $\left(A_0,V\right)$ and $\left(B_0,W\right)$ are $\mathrm{dg}_{\geq 0}$-affine spaces their product will be just $\left(A_0\times B_0,V\times W\right)$, where $V\times W$ is the product of $V$ and $W$ in $\mathfrak{Ch}_{\geq 0}\left(\mathfrak{Vect}_k\right)$ and $A_0\times B_0$ is the affine space over $k$ whose difference vector space is $V_0\times W_0$. \\
We can also put a tensor structure over $\mathfrak{Ch}_{\geq 0}\left(\mathfrak{Aff}_k\right)$: given two $\mathrm{dg}_{\geq 0}$-affine spaces $\left(A_0,V\right)$ and $\left(B_0,W\right)$, define their tensor product $\left(A_0,V\right)\otimes\left(B_0,W\right)$ to be the $\mathrm{dg}_{\geq 0}$-affine space determined by the chain complex
\begin{equation} \label{tens struc aff}
V\oplus W\oplus \left(V\otimes W\right).
\end{equation}
Similarly, given 
\begin{eqnarray*}
&\phi\in\mathrm{Hom}_{\mathfrak{Ch}_{\geq 0}\left(\mathfrak{Aff}_k\right)}\left(\left(A_0,V\right),\left(B_0,W\right)\right)\quad\text{where}\quad\phi\left(\underline v\right):=f\left(\underline v\right)+b& \\ 
&\psi\in\mathrm{Hom}_{\mathfrak{Ch}_{\geq 0}\left(\mathfrak{Aff}_k\right)}\left(\left(C_0,U\right),\left(D_0,Z\right)\right)\quad\text{where}\quad\psi\left(\underline u\right):=g\left(\underline u\right)+d&
\end{eqnarray*}
the tensor product map is given by
\begin{eqnarray} \label{tens morph aff}
\phi\otimes\psi:&V\oplus W\oplus \left(V\otimes W\right)&\longrightarrow\quad U\oplus Z\oplus \left(U\otimes Z\right)\nonumber \\
&\left(\underline u,\underline v,\underline x\otimes \underline y\right)&\mapsto\left(\underline u+b,\underline v+d,f\left(\underline x\right)\otimes g\left(\underline y\right)\right).
\end{eqnarray}
Formula \eqref{tens struc aff} and formula \eqref{tens morph aff} determine a monoidal structure on $\mathfrak{Ch}_{\geq 0}\left(\mathfrak{Aff}_k\right)$: in particular the unit is given by the object $\left(\{*\},0\right)$, the associator is induced by the monoidal structure on $\mathfrak{Ch}_{\geq 0}\left(\mathfrak{Vect}_k\right)$ and the unitors are simply given by
\begin{equation*}
\left(V\oplus 0\oplus\left(V\otimes 0\right)\right)\overset{\sim}{\longrightarrow}V
\end{equation*}
\begin{equation*}
\left(0\oplus V\oplus\left(0\otimes V\right)\right)\overset{\sim}{\longrightarrow}V.
\end{equation*}
The reader can check that the above definitions verify the pentagon and the triangle identities.
\begin{rem}
Let $\left(A_0,V\right)$ and $\left(B_0,W\right)$ be dg$_{\geq 0}$-affine spaces: notice that
\begin{equation} \label{coh}
\left(A_0,V\right)\otimes\left(B_0,W\right)\simeq A_0\tilde{\otimes}B_0
\end{equation}
where formula \eqref{coh} is a canonical identification in $\mathfrak{Aff}_k$; an analogous coherence statement holds for morphisms.
\end{rem}
\begin{defn}
Define a \emph{simplicial affine space} over $k$ to be just a simplicial object in $\mathfrak{Aff}_k$.
\end{defn}
Let $\mathfrak{sAff}_k$ be the category of simplicial affine spaces over $k$, i.e.
\begin{equation*}
\mathfrak{sAff}_k:=\mathfrak{Aff}_k^{\Delta^{\mathrm{op}}}.
\end{equation*}
\begin{rem} \label{sadj}
There is a natural linearisation functor
\begin{equation*}
\mathbf L:\mathfrak{sAff}_k\longrightarrow\mathfrak{sVect}_k
\end{equation*}
which just deletes the non-linear part in the face and degeneracy maps defining a simplicial affine space, as well as the non-linear part of morphisms between simplicial affine spaces; in the same fashion there is a forgetful functor
\begin{equation*}
\mathbf U:\mathfrak{sVect}_k\longrightarrow\mathfrak{sAff}_k
\end{equation*} 
which just takes (maps of) simplicial vector spaces and looks at them as (maps of) simplicial affine ones.
\end{rem}
\begin{war}
The pair of functors $(\mathbf U,\mathbf L)$ does not provide an adjunction between $\mathfrak{sAff}_k$ and $\mathfrak{sVect}_k$.
\end{war}
The category $\mathfrak{sAff}_k$ has all small limits and colimits, which are just taken levelwise; moreover define the tensor product in $\mathfrak{sAff}_k$ to be constructed by simply taking the tensor product in $\mathfrak{Aff}_k$ in all levels: it is straightforward to check that this equips such a category with a monoidal structure.\\
Now define the normalisation of a $\mathrm{dg}_{\geq 0}$-affine space over $k$ to be the functor
\begin{eqnarray*}
\breve{\mathbf N}:&\mathfrak{sAff}_k&\xrightarrow{\hspace*{1cm}}\quad\mathfrak{Ch}_{\geq 0}\left(\mathfrak{Aff}_k\right) \\
&A_{\bullet}&\longmapsto\qquad\left(A_0,\mathbf N\left(\mathbf L\left(A_{\bullet}\right)\right)\right) \\
&\vcenter{\xymatrix{A_{\bullet}\ar[d]^{\phi} \\
B_{\bullet}}}&\mapsto
\left(\vcenter{\xymatrix{A_0\ar[d]^{\phi_0} \\
B_0}}\;\;,\;\;\vcenter{\xymatrix{\mathbf N\left(\mathbf L\left(A_{\bullet}\right)\right)\ar[d]^{\mathbf N\left(\mathbf L\left(\phi\right)\right)} \\ \mathbf N\left(\mathbf L\left(B_{\bullet}\right)\right)}}\right)
\end{eqnarray*}
and observe that such a definition is well-given as the $0$-th term of the chain complex $\mathbf N\left(\mathbf L\left(A_{\bullet}\right)\right)$ is precisely the difference vector space of $A_0$; in other words, the normalisation of a simplicial vector space does not affect the object in degree $0$, as follows from formula \eqref{norm}. \\
Analogously, define the denormalisation of a simplicial affine space over $k$ to be the functor
\begin{eqnarray*}
\breve{\mathbf K}:&\mathfrak{Ch}_{\geq 0}\left(\mathfrak{Aff}_k\right)&\xrightarrow{\hspace*{3.5cm}}\mathfrak{sAff}_k \\
&\left(A_0,V\right)&\longmapsto\xymatrix{A_0\ar[r] & A_0\times V_1\ar@<2.5pt>[l]\ar@<-2.5pt>[l]\ar@<2.5pt>[r]\ar@<-2.5pt>[r] & A_0\times V_1^{\oplus 2}\times V_2 \ar[l]\ar@<5pt>[l]\ar@<-5pt>[l]\ar[r]\ar@<5pt>[r]\ar@<-5pt>[r] & \cdots\ar@<-7.5pt>[l]\ar@<-2.5pt>[l]\ar@<2.5pt>[l]\ar@<7.5pt>[l]}
\end{eqnarray*}
where the maps involving $A_0$ and $A_0\times V_1$ are
\begin{eqnarray*}
&A_0\times V_1\ni\left(a,v\right)\mapsto a+d\left(v\right)\in A_0& \\
&A_0\times V_1\ni\left(a,v\right)\mapsto a\in A_0& \\
&A_0\ni a\mapsto\left(a,0\right)\in A_0\times V_1&
\end{eqnarray*}
and all other faces and degeneracies -- which do not involve the affine space $A_0$, but rather only the vector spaces $V_i$ -- are defined as done for classical denormalisation (see Remark \ref{denorm} and subsequent discussion). \\
In a similar way, given
\begin{equation*}
\phi\in\mathrm{Hom}_{\mathfrak{Ch}_{\geq 0}\left(\mathfrak{Aff}_k\right)}\left(\left(A_0,V\right),\left(B_0,W\right)\right)\quad\text{where}\quad\phi\left(\underline v\right):=f\left(\underline v\right)+b
\end{equation*}
the morphism $\breve{\mathbf K}\left(\phi\right)$ of simplicial affine spaces is defined in level $n$ by the affine map
\begin{eqnarray*}
&A_0\times V_1^{\oplus n}\times V_2^{\oplus\binom{n}{2}}\times\cdots\times V_n^{\oplus\binom{n}{n}}&\longrightarrow \qquad\;\; B_0\times W_1^{\oplus n}\times W_2^{\oplus\binom{n}{2}}\times\cdots\times W_n^{\oplus\binom{n}{n}} \\
&\left(a_0,\left(v_1^i\right)_i,\left(v_2^j\right)_j,\ldots,v_n\right)&\mapsto\left(f_0\left(a_0\right)+b,\left(f_1\left(v_1^i\right)\right)_i,\left(f_2\left(v_2^j\right)\right)_j,\ldots,f_n\left(v_n\right)\right)
\end{eqnarray*}
We are ready to describe the generalisation of Theorem \ref{Dold-Kan} we mentioned before.
\begin{prop} \label{ext Dold-Kan}
The functors $\breve{\mathbf N}$ and $\breve{\mathbf K}$ form an equivalence of categories between $\mathfrak{sAff}_k$ and $\mathfrak{Ch}_{\geq 0}\left(\mathfrak{Aff}_k\right)$.
\end{prop}
\begin{proof}
The arguments used in \cite{We} Theorem 8.4.1 to prove the classical Dold-Kan correspondence given by Theorem \ref{Dold-Kan} carry over to this context. 
\end{proof}
As follows for instance from the discussion in \cite{SchShi} Section 2.3, the normalisation functor $\breve{\mathbf N}$ can be made into a lax monoidal functor via the Eilenberg-Zilber shuffle map, which is the natural transformation 
\begin{equation*}
\mathrm{EZ}:\breve{\mathbf N}\left(-\right)\otimes\breve{\mathbf N}\left(-\right)\longrightarrow\breve{\mathbf N}\left(-\otimes-\right)
\end{equation*}
determined for all $A_{\bullet},B_{\bullet}\in\mathfrak{sAff}_k$ by the morphisms
\begin{eqnarray*}
\mathrm{EZ}^{p,q}_{A_{\bullet},B_{\bullet}}:&\breve{\mathbf N}\left(A_{\bullet}\right)_p\otimes\breve{\mathbf N}\left(B_{\bullet}\right)_q&\xrightarrow{\hspace*{1.5cm}}\breve{\mathbf N}\left(A_{\bullet}\otimes B_{\bullet}\right)_{p+q} \\
&a\otimes b&\longmapsto\underset{\left(\mu,\nu\right)}{\sum}\mathrm{sign}\left(\mu,\nu\right)\,\sigma_{\nu}\left(a\right)\otimes\sigma_{\mu}\left(b\right)
\end{eqnarray*}
where the sum runs over all $(p,q)$-shuffles
\begin{equation*}
\left(\mu,\nu\right):=\left(\mu_1,\ldots,\mu_p,\nu_1,\ldots,\nu_q\right)
\end{equation*}
and the corresponding degeneracy maps are
\begin{equation*}
s_{\mu}:=s_{\mu_p}\circ\ldots s_{\mu_1}\qquad\qquad s_{\nu}:=s_{\nu_q}\circ\ldots s_{\nu_1}.
\end{equation*}
In the same fashion, again from \cite{SchShi} Section 2.3, the denormalisation functor $\tilde{\mathbf K}$ can also be made into a lax monoidal functor by means of the Alexander-Whitney map. The latter is defined to be the natural transformation
\begin{eqnarray*}
\mathrm{AW}:\tilde{\mathbf N}\left(-\otimes -\right)\longrightarrow\tilde{\mathbf N}\left(-\right)\otimes\tilde{\mathbf N}\left(-\right)
\end{eqnarray*}
given for all $A_{\bullet},B_{\bullet}\in\mathfrak{sAff}_k$ by the morphisms
\begin{eqnarray} \label{AW K}
\mathrm{AW}_{A_{\bullet},B_{\bullet}}^{n}:&\tilde{\mathbf N}\left(A_{\bullet}\otimes B_{\bullet}\right)_{n}&\xrightarrow{\hspace*{1.5cm}}\left(\tilde{\mathbf N}\left(A_{\bullet}\right)\otimes\tilde{\mathbf N}\left(B_{\bullet}\right)\right)_n \nonumber \\
&a\otimes b&\longmapsto\underset{p+q=n}{\bigoplus}\left(\tilde d^p\left(a\right)\otimes d_0^q\left(b\right)\right)
\end{eqnarray}
where the {} ``front face'' $\tilde d^p$ and the {} ``back face'' $d_0^q$ are induced respectively by the injective monotone maps $\tilde{\delta^p}:\left[p\right]\rightarrow\left[p+q\right]$ and $\delta_0^p:\left[q\right]\rightarrow\left[p+q\right]$; in particular the Alexander-Whitney map makes the normalisation functor $\tilde{\mathbf N}$ into a comonoidal one (again, see \cite{SchShi} Section 2.3, whose considerations adapt to these context). Notice also that by setting $A':=\tilde{\mathbf N}\left(A\right)$ and $B':=\tilde{\mathbf N}\left(A\right)$ in formula \eqref{AW K} and using the equivalence provided by Proposition \ref{ext Dold-Kan}, we get a version of the Alexander-Whitney transformation
\begin{equation*}
\mathrm{AW}:\tilde{\mathbf K}\left(-\right)\otimes\tilde{\mathbf K}\left(-\right)\longrightarrow\tilde{\mathbf K}\left(-\otimes-\right)\footnote{There is some abuse of notation in this formula}
\end{equation*}
which makes the denormalisation $\tilde{\mathbf K}$ into a lax monoidal functor. Also we have that the composite $\mathrm{AW}\circ\mathrm{EZ}$ is the same as the identity, while the transformation $\mathrm{EZ}\circ\mathrm{AW}$ is chain homotopic to the identity: in particular the Dold-Kan equivalence provided by Proposition \ref{ext Dold-Kan} is lax monoidal. \\
Now we are ready to introduce the notions of affine $\mathrm{dg}_{\geq 0}$-category and affine simplicial category, which will be crucial technical tools to develop a good derived version of the period map.
\begin{defn}
An \emph{affine differential graded category} over $k$ $\left(\mathfrak C,\underline{\mathfrak C}_{\bullet}\right)$ (\emph{affine dg$_{\geq 0}$-category} for short) is a category $\mathfrak C$ enriched over $\mathfrak{Ch}_{\geq 0}\left(\mathfrak{Aff}_k\right)$.
\end{defn}
We will denote by $\mathfrak{dg}_{\geq 0}\mathfrak{Cat}_k^{\mathfrak{Aff}}$ the $\infty$-category of affine $\mathrm{dg}_{\geq 0}$-categories.\\
Let $\left(\mathfrak C,\underline{\mathfrak C}_{\bullet}\right)$ be an affine $\mathrm{dg}_{\geq0}$-category and denote by $H_0\left(\left(\mathfrak C,\underline{\mathfrak C}_{\bullet}\right)\right)$ the (honest) category defined by the relations
\begin{eqnarray*}
\mathrm{Ob}\left(H_0\left(\left(\mathfrak C,\underline{\mathfrak C}_{\bullet}\right)\right)\right)&:=&\mathfrak C \\
\forall X,Y\in\mathfrak C\qquad\mathrm{Hom}_{H_0\left(\left(\mathfrak C,\underline{\mathfrak C}_{\bullet}\right)\right)}\left(X,Y\right)&:=&H_0\left(\underline{\mathfrak C}_{\bullet}\left(X,Y\right)\right).
\end{eqnarray*}
\begin{defn}
An \emph{affine differential graded groupoid} over $k$ (\emph{affine dg$_{\geq 0}$-groupoid} for short) will be an affine $\mathrm{dg}_{\geq 0}$-category $\left(\mathfrak C,\underline{\mathfrak C}_{\bullet}\right)$ such that the category $H_0\left(\left(\mathfrak C,\underline{\mathfrak C}_{\bullet}\right)\right)$ is a groupoid.
\end{defn}
We will denote by $\mathfrak{dg}_{\geq 0}\mathfrak{Grpd}_k^{\mathfrak{Aff}}$ the $\infty$-category of affine $\mathrm{dg}_{\geq 0}$-groupoids.
\begin{rem} \label{obsss}
The notion of $\mathrm{dg}_{\geq 0}$-affine space allows us to define a notion of $\infty$-groupoid in the differential graded context: as a matter of fact a more naive notion of $\mathrm{dg}_{\geq 0}$-groupoid -- intended as a $\mathrm{dg}_{\geq 0}$-category where all morphisms in level $0$ are isomorphism -- would not really make sense as every $\mathrm{dg}_{\geq 0}$-category comes with a zero morphism, which is seldom an isomorphism. 
\end{rem}
\begin{defn}
An \emph{affine simplicial category} over $k$ $\left(\mathfrak C,\underline{\mathfrak C}_{\bullet}\right)$ is a category $\mathfrak C$ enriched over $\mathfrak{sAff}_k$
\end{defn}
We will denote by $\mathfrak{sCat}_k^{\mathfrak{Aff}}$ the $\infty$-category of affine simplicial categories. \\
Let $\left(\mathfrak C,\underline{\mathfrak C}_{\bullet}\right)$ be an affine simplicial category and denote by $\pi_0\left(\left(\mathfrak C,\underline{\mathfrak C}_{\bullet}\right)\right)$ the (honest) category defined by the relations
\begin{eqnarray*}
\mathrm{Ob}\left(\pi_0\left(\left(\mathfrak C,\underline{\mathfrak C}_{\bullet}\right)\right)\right)&:=&\mathfrak C \\
\forall X,Y\in\mathfrak C\qquad\mathrm{Hom}_{\pi_0\left(\left(\mathfrak C,\underline{\mathfrak C}_{\bullet}\right)\right)}\left(X,Y\right)&:=&\pi_0\left(\underline{\mathfrak C}_{\bullet}\left(X,Y\right)\right).
\end{eqnarray*} 
\begin{defn}
An \emph{affine simplicial groupoid} over $k$ will be an affine simplicial category $\left(\mathfrak C,\underline{\mathfrak C}_{\bullet}\right)$ such that the category $\pi_0\left(\left(\mathfrak C,\underline{\mathfrak C}_{\bullet}\right)\right)$ is a groupoid. 
\end{defn}
We will denote by $\mathfrak{sGrpd}_k^{\mathfrak{Aff}}$ the $\infty$-category of affine simplicial groupoids.
\begin{rem}
The notion of simplicial affine space allows us to define a notion of $\infty$-groupoid in the $k$-simplicial context, just like $\mathrm{dg}_{\geq 0}$-affine spaces give rise to a good notion of differential graded groupoid, as observed in Remark \ref{obsss}
\end{rem}
Of course any simplicial affine space has an underlying simplicial set, so an affine simplicial category over $k$ is in particular a simplicially enriched category: more formally, there is a natural forgetful functors from $\mathfrak{sCat}_k^{\mathfrak{Aff}}$ to $\mathfrak{sCat}$. \\
The slightly extended version of the Dold-Kan equivalence given by Proposition \ref{ext Dold-Kan} induces a pair of functors 
\begin{equation} \label{aff cat DK}
\breve{\mathbf N}:\mathfrak{sCat}_k^{\mathfrak{Aff}}\rightleftarrows\mathfrak{dg}_{\geq 0}\mathfrak{Cat}_k^{\mathfrak{Aff}}:\breve{\mathbf K}\footnote{There is some abuse of notation in this formula.}
\end{equation}
where
\begin{eqnarray*}
&\breve{\mathbf N}:\qquad\mathfrak{sCat}_k^{\mathfrak{Aff}}\xrightarrow{\hspace*{3cm}}\mathfrak{dg}_{\geq 0}\mathfrak{Cat}_k^{\mathfrak{Aff}}\qquad& \\
&\mathfrak C\xmapsto{\hspace*{4cm}}\mathfrak C& \\
\forall P,Q\in\mathfrak C&\quad\underline{\mathfrak C}_{\bullet}\left(P,Q\right)\xmapsto{\hspace*{2.5cm}}\breve{\mathbf N}\left(\underline{\mathfrak C}_{\bullet}\left(P,Q\right)\right)& \\
\forall P,Q,R\in\mathfrak C&\qquad\left(\vcenter{\xymatrix{\underline{\mathfrak C}_{\bullet}\left(P,Q\right)\otimes\underline{\mathfrak C}_{\bullet}\left(Q,R\right)\ar[d]^{\circ} \\ \underline{\mathfrak C}_{\bullet}\left(P,R\right)}}\right)\mapsto\left(\vcenter{\xymatrix{\breve{\mathbf N}\left(\underline{\mathfrak C}_{\bullet}\left(P,Q\right)\right)\otimes\breve{\mathbf N}\left(\underline{\mathfrak C}_{\bullet}\left(Q,R\right)\right)\ar[d]^{\mathrm{EZ}_{\underline{\mathfrak C}_{\bullet}\left(P,Q\right),\underline{\mathfrak C}_{\bullet}\left(Q,R\right)}} \\ \breve{\mathbf N}\left(\underline{\mathfrak C}_{\bullet}\left(P,Q\right)\otimes\underline{\mathfrak C}_{\bullet}\left(Q,R\right)\right)\ar[d]^{\breve{\mathbf N}\left(\circ\right)} \\ \breve{\mathbf N}\left(\underline{\mathfrak C}_{\bullet}\left(P,R\right)\right)}}\right)&
\end{eqnarray*}
and 
\begin{eqnarray*}
&\breve{\mathbf K}:\qquad\mathfrak{dg}_{\geq 0}\mathfrak{Cat}_k^{\mathfrak{Aff}}\xrightarrow{\hspace*{3cm}}\mathfrak{sCat}_k^{\mathfrak{Aff}}\qquad& \\
&\mathfrak C\xmapsto{\hspace*{4cm}}\mathfrak C& \\
\forall P,Q\in\mathfrak C&\quad\underline{\mathfrak C}_{\bullet}\left(P,Q\right)\xmapsto{\hspace*{2.5cm}}\breve{\mathbf K}\left(\underline{\mathfrak C}_{\bullet}\left(P,Q\right)\right)& \\
\forall P,Q,R\in\mathfrak C&\qquad\left(\vcenter{\xymatrix{\underline{\mathfrak C}_{\bullet}\left(P,Q\right)\otimes\underline{\mathfrak C}_{\bullet}\left(Q,R\right)\ar[d]^{\circ} \\ 
\underline{\mathfrak C}_{\bullet}\left(P,R\right)}}\right)\mapsto\left(\vcenter{\xymatrix{\breve{\mathbf K}\left(\underline{\mathfrak C}_{\bullet}\left(P,Q\right)\right)\otimes\breve{\mathbf K}\left(\underline{\mathfrak C}_{\bullet}\left(Q,R\right)\right)\ar[d]^{\mathrm{AW_{\underline{\mathfrak C}_{\bullet}\left(P,Q\right),\underline{\mathfrak C}_{\bullet}\left(Q,R\right)}}} \\ 
\breve{\mathbf K}\left(\underline{\mathfrak C}_{\bullet}\left(P,Q\right)\otimes\underline{\mathfrak C}_{\bullet}\left(Q,R\right)\right)\ar[d]^{\breve{\mathbf K}\left(\circ\right)} \\ 
\breve{\mathbf K}\left(\underline{\mathfrak C}_{\bullet}\left(P,R\right)\right)}}\right)&
\end{eqnarray*}
Notice also that the $\infty$-equivalence given by formula \eqref{aff cat DK} restricts to an $\infty$-equivalence
\begin{equation*}
\breve{\mathbf N}:\mathfrak{sGrpd}_k^{\mathfrak{Aff}}\rightleftarrows\mathfrak{dg}_{\geq 0}\mathfrak{Grpd}_k^{\mathfrak{Aff}}:\breve{\mathbf K}.
\end{equation*}
At last, let us recall that there is a natural functor
\begin{equation} \label{bar W}
\bar W:\mathfrak{sCat}\longrightarrow\mathfrak{sSet}
\end{equation}
given by the right adjoint to Illusie's Dec functor; we are not describing it explicitly as its construction is slightly technical and not really needed for the sake of this paper: the definition of $\bar W$ can be found in \cite{GJ} Section V.7 or \cite{Pr3} Section 1. Moreover in \cite{CR} Cegarra and Remedios proved that $\bar W$ is weakly equivalent to the diagonal of the simplicial nerve functor. Functor \eqref{bar W} is also known to induce a right Quillen equivalence
\begin{equation*}
\bar W:\mathfrak{sGrpd}\longrightarrow\mathfrak{sSet}
\end{equation*}
and -- as a corollary of the results in \cite{Pr3} Section 1 -- we also have that functor \eqref{bar W} restricts to an equivalence
\begin{equation*}
\bar W:\mathfrak{sGrpd}_k^{\mathfrak{Aff}}\longrightarrow\mathfrak{sSet}.
\end{equation*} 
In Section 3.4 we will apply the functor $\bar W$ to interesting affine simplicial groupoids in order to define rigorously the derived deformation functor $\mathrm{hoFlag}_{\mathbb R\Gamma\left(X,\Omega^*_{X/k}\right)}^{F^{\bullet}}$, while the functor $\mathbb R\mathrm{Def}_X$ will be constructed in Section 3.5 by using different techniques.
\subsection{Derived Deformations of Filtered Complexes}
In this section we will study in detail how to enhance the formal Grassmannian and the flag functor analysed in Section 2.2 and Section 2.4 to the world of derived deformation functors; in particular we will define the notions of derived formal total Grassmannian and formal homotopy flag variety and make comparisons with the corresponding derived stacks studied in \cite{dN2}. \\
Throughout this section fix $V$ to be a complex of $k$-vector spaces, $S$ a subcomplex of it and $\mathcal F^{\bullet}$ a filtration on $V$. Also recall from \cite{dN2} that for any $R$ in $\mathfrak{Alg}_k$ or even $\mathfrak{dgAlg}^{\leq 0}_k$ there is a model structure on $\mathfrak{FdgMod}_R$ modelled on the projective model structure over $\mathfrak{dgMod}_R$.
\begin{defn} \label{der def V}
A \emph{derived deformation of $V$ over $A\in\mathfrak{dgArt}_k^{\leq 0}$} is the datum of a cofibrant complex of $A$-modules $\mathscr V_A$ and a surjective quasi-isomorphism $\varphi:\mathscr V_A\otimes_Ak\rightarrow V$.
\end{defn}
\begin{defn} \label{der def W}
A \emph{derived deformation of $S$ inside $V$ over $A\in\mathfrak{dgArt}_k^{\leq 0}$} is the datum of a derived deformation $\left(\mathscr V_A,\varphi\right)$ of $V$ and a subcomplex $\mathscr S_A\subseteq\mathscr V_A$ such that
\begin{equation*}
\varphi\big|_{\mathscr S_A}:\mathscr S_A\otimes_Ak\longrightarrow S
\end{equation*}
is still a surjective quasi-isomorphism.
\end{defn}
\begin{defn} \label{der def filt cmplx}
A \emph{derived deformation of $\left(V,\mathcal F^{\bullet}\right)$ over $A\in\mathfrak{dgArt}_k^{\leq 0}$} is a pair $\left(\left(\mathscr V_A, \mathcal F^{\bullet}_A\right),\varphi\right)$, where $\left(\mathscr V_A,\mathcal F^{\bullet}_A\right)$ is a cofibrant filtered complex of $A$-modules and $\varphi:\left(\mathscr V_A, \mathcal F^{\bullet}_A\right)\rightarrow\left(V,\mathcal F^{\bullet}\right)$ is a surjective morphism such that the maps $\mathcal F_A^p\mathscr V_A\otimes_Ak\rightarrow \mathcal F^pV$ induced by $\varphi$ are quasi-isomorphisms for all $p$.
\end{defn}
\begin{rem} \label{fib-cof flag}
Notice that the key technical assumptions in Definition \ref{der def V}, Definition \ref{der def W} and Definition \ref{der def filt cmplx} are 
\begin{enumerate}
\item restricting to cofibrant objects while defining derived deformations of (filtered) derived modules;
\item requiring that the structure morphism $\varphi$ is surjective.
\end{enumerate}
Assumption (1) allows us to use tensor products instead of derived tensor products. Assumption (2) makes objects in te slice category $\nicefrac{\mathfrak{dgMod}_A}{V}$ (or $\nicefrac{\mathfrak{FdgMod}_A}{\left(V,\mathcal F^{\bullet}\right)}$) fibrant.  Actually in the main geometric example we will consider, that is derived deformations of (filtered) perfect complexes of modules over a scheme, the perfectness condition will imply cofibrancy wherever relevant.
\end{rem}
Derived deformation of complexes, subcomplexes and filtrations are governed by nice derived deformation functors, already hidden in the Derived Deformation Theory folklore: we now propose a rigorous way to define them  via the language of affine $\mathrm{dg}_{\geq 0}$-categories. \\
In the above notations, let
\begin{equation*}
\left(hDef_V\left(A\right),\underline{hDef_V}\left(A\right)_{\bullet}\right) 
\end{equation*}
be the affine $\mathrm{dg}_{\geq 0}$-category defined by the formulae
\begin{equation*}
hDef_V\left(A\right):=\left\{\text{derived }A\text{-deformations of }V\right\}
\end{equation*}
and for all $\left(\mathscr V_A,\varphi\right),\left(\mathscr W_A,\phi\right)\in hDef_V\left(A\right)$
\begin{eqnarray*}
&\underline{hDef_V}\left(A\right)_0:=\left\{\Psi\in\mathrm{Hom}^0\left(\mathscr V_A,\mathscr W_A\right)\text{ s.t. }\phi\circ\Psi=\varphi\right\}&  \\
&\underline{hDef_V}\left(A\right)_1:=\left\{\Psi\in\mathrm{Hom}^{-1}\left(\mathscr V_A,\mathscr W_A\right)\text{ s.t. }\phi\circ\Psi=0\right\}& \\
&\vdots& \\
&\underline{hDef_V}\left(A\right)_n:=\left\{\Psi\in\mathrm{Hom}^{-n}\left(\mathscr V_A,\mathscr W_A\right)\text{ s.t. }\phi\circ\Psi=0\right\}&  \\
&\vdots&
\end{eqnarray*}
with the differential induced by the standard differential on $\mathrm{Hom}$ complexes.
Similarly, let
\begin{equation*}
\left(hDef_{S,V}\left(A\right),\underline{hDef_{S,V}}\left(A\right)_{\bullet}\right) 
\end{equation*}
be the affine $\mathrm{dg}_{\geq 0}$-category defined by the formulae
\begin{equation*}
hDef_{S,V}\left(A\right):=\left\{\text{(derived) }A\text{-deformations of } S \text{ inside } V\right\}
\end{equation*}
and for all $\left(\left(\mathscr S_A, \mathscr V_A\right),\varphi\right),\left(\left(\mathscr T_A, \mathscr W_A\right),\phi\right)\in hDef_{S,V}\left(A\right)$
\begin{equation*}
\underline{hDef_{S,V}}\left(A\right)_0:=\left\{\Psi\in\mathrm{Hom}^0\left(\left(\mathscr S_A,\mathscr V_A\right),\left(\mathscr T_A,\mathscr W_A\right)\right)\text{ s.t. }\phi\circ\Psi=\varphi\right\}
\end{equation*}
\begin{eqnarray*}
&\underline{hDef_{S,V}}\left(A\right)_1:=\left\{\Psi\in\mathrm{Hom}^{-1}\left(\left(\mathscr S_A,\mathscr V_A\right),\left(\mathscr T_A, \mathscr W_A\right)\right)\text{ s.t. }\phi\circ\Psi=0\right\}& \\
&\vdots& \\
&\underline{hDef_{S,V}}\left(A\right)_n:=\left\{\Psi\in\mathrm{Hom}^{-n}\left(\left(\mathscr S_A,\mathscr V_A\right),\left(\mathscr T_A, \mathscr W_A\right)\right)\text{ s.t. }\phi\circ\Psi=0\right\}&  \\
&\vdots&
\end{eqnarray*}
with the differential induced by the standard differential on $\mathrm{Hom}$ complexes. \\
Finally let
\begin{equation*}
\left(hDef_V^{\mathcal F^{\bullet}}\left(A\right),\underline{hDef_V^{\mathcal F^{\bullet}}}\left(A\right)_{\bullet}\right) 
\end{equation*}
be the affine $\mathrm{dg}_{\geq 0}$-category defined by the formulae
\begin{equation*}
hDef_V^{\mathcal F^{\bullet}}\left(A\right):=\left\{\text{(derived) }A\text{-deformations of }\left(V,\mathcal F^{\bullet}\right)\right\}
\end{equation*}
and for all $\left(\left(\mathscr V_A, \mathcal F^{\bullet}_A\right)\varphi\right),\left(\left(\mathscr W_A,\mathcal G^{\bullet}_A\right),\phi\right)\in hDef_V^{\mathcal F^{\bullet}}\left(A\right)$
\begin{eqnarray*}
&\underline{hDef_V^{\mathcal F^{\bullet}}}\left(A\right)_0:=\left\{\Psi\in\mathrm{Hom}^0\left(\left(\mathscr V_A,\mathcal F^{\bullet}_A\right),\left(\mathscr W_A,\mathcal G^{\bullet}_A\right)\right)\text{ s.t. }\phi\circ\Psi=\varphi\right\}&  \\
&\underline{hDef_V^{\mathcal F^{\bullet}}}\left(A\right)_1:=\left\{\Psi\in\mathrm{Hom}^{-1}\left(\left(\mathscr V_A,\mathcal F^{\bullet}_A\right),\left(\mathscr W_A,\mathcal G^{\bullet}_A\right)\right)\text{ s.t. }\phi\circ\Psi=0\right\}& \\
&\vdots& \\
&\underline{hDef_V^{\mathcal F^{\bullet}}}\left(A\right)_n:=\left\{\Psi\in\mathrm{Hom}^{-n}\left(\left(\mathscr V_A,\mathcal F^{\bullet}_A\right),\left(\mathscr W_A,\mathcal G^{\bullet}_A\right)\right)\text{ s.t. }\phi\circ\Psi=0\right\}&  \\
&\vdots&
\end{eqnarray*}
with the differential induced by the standard differential on $\mathrm{Hom}$ complexes.
\begin{rem}
The $2$-out-of-$3$ property implies that morphisms in $\underline{hDef_V}\left(A\right)_0$, $\underline{hDef_{S,V}}\left(A\right)_0$ and $\underline{hDef_V^{\mathcal F^{\bullet}}}\left(A\right)_0$ are all weak equivalences; in particular 
\begin{itemize}
\item $H_0\left(\left(hDef_V\left(A\right),\underline{hDef_V}\left(A\right)_{\bullet}\right)\right)$
\item $H_0\left(\left(hDef_{S,V}\left(A\right),\underline{hDef_{S,V}}\left(A\right)_{\bullet}\right)\right)$  
\item $H_0\left(\left(hDef_V^{\mathcal F^{\bullet}}\left(A\right),\underline{hDef_V^{F^{\bullet}}}\left(A\right)_{\bullet}\right)\right)$
\end{itemize} 
are groupoids, so the affine $\mathrm{dg}_{\geq 0}$-categories 
\begin{itemize}
\item $\left(hDef_V\left(A\right),\underline{hDef_V}\left(A\right)_{\bullet}\right)$ 
\item $\left(hDef_{S,V}\left(A\right),\underline{hDef_{S,V}}\left(A\right)_{\bullet}\right)$
\item $\left(hDef_V^{\mathcal F^{\bullet}}\left(A\right),\underline{hDef_V^{\mathcal F^{\bullet}}}\left(A\right)_{\bullet}\right)$ 
\end{itemize}
are really affine $\mathrm{dg}_{\geq 0}$-groupoids.
\end{rem}
Now define the derived deformation functors 
\begin{eqnarray*}
\mathbb R\mathrm{Def}_V:&\mathfrak{dgArt}_k^{\leq 0}&\xrightarrow{\hspace*{3cm}}\mathfrak{sSet} \\
&A&\longmapsto\bar W\left(\breve{\mathbf K}\left(\left(hDef_V\left(A\right),\underline{hDef_V}\left(A\right)_{\bullet}\right)\right)\right).
\end{eqnarray*}
\begin{eqnarray*}
\mathbb R\mathrm{Def}_{S,V}:&\mathfrak{dgArt}_k^{\leq 0}&\xrightarrow{\hspace*{3cm}}\mathfrak{sSet} \\
&A&\longmapsto\bar W\left(\breve{\mathbf K}\left(\left(hDef_{S,V}\left(A\right),\underline{hDef_{S,V}}\left(A\right)_{\bullet}\right)\right)\right).
\end{eqnarray*}
\begin{eqnarray*}
\mathbb R\mathrm{Def}_V^{\mathcal F^{\bullet}}:&\mathfrak{dgArt}_k^{\leq 0}&\xrightarrow{\hspace*{3cm}}\mathfrak{sSet} \\
&A&\longmapsto\bar W\left(\breve{\mathbf K}\left(\left(hDef_V^{\mathcal F^{\bullet}}\left(A\right),\underline{hDef_V^{\mathcal F^{\bullet}}}\left(A\right)_{\bullet}\right)\right)\right).
\end{eqnarray*}
\begin{lemma} \label{Def_End}
In the above notations we have:
\begin{enumerate}
\item $\mathbb R\mathrm{Def}_V$ is weakly equivalent to the derived deformation functor $\mathbb R\mathrm{Def}_{\mathrm{End}^*\left(V\right)}$ associated to the dgla $\mathrm{End}^*\left(V\right)$;
\item $\mathbb R\mathrm{Def}_{S,V}$ is weakly equivalent to the derived deformation functor $\mathbb R\mathrm{Def}_{\mathrm{End}^S\left(V\right)}$ associated to the dgla $\mathrm{End}^S\left(V\right)$;
\item $\mathbb R\mathrm{Def}_V^{\mathcal F^{\bullet}}$ is weakly equivalent to the derived deformation functor $\mathbb R\mathrm{Def}_{\mathrm{End}^{\mathcal F^{\bullet}}\left(V\right)}$ associated to the dgla $\mathrm{End}^{\mathcal F^{\bullet}}\left(V\right)$;
\end{enumerate}
\end{lemma}
\begin{proof}
This result is well-known in the Derived Deformation Theory folklore: we just recall the morphism 
\begin{equation*}
\mathbb R\mathrm{Def}_{\mathrm{End}^*}\left(V\right)\longrightarrow\mathbb R\mathrm{Def}_V
\end{equation*}
giving Claim (1). \\
Note that by Theorem \ref{Del grpd} it suffices to determine such a map on the (derived) Deligne groupoid $\mathrm{BDel}_{\mathrm{End}^*\left(V\right)}$ associated to the dgla $\mathrm{End}^*\left(V\right)$, so define
\begin{eqnarray*} \label{nu}
&\qquad\qquad\qquad\qquad\nu:\mathrm{BDel}_{\mathrm{End}^*\left(V\right)}\xrightarrow{\hspace*{1cm}}\mathbb R\mathrm{Def}_V&  \\
&\text{for all }A\in\mathfrak{dgArt}^{\leq 0}_k\quad\widetilde{\mathrm{MC}}_{\mathrm{End}^*\left(V\right)}\left(A\right)\ni \sigma\longmapsto\left[\left(V\otimes A,d+\sigma\right)\right]\quad\qquad&  \\
&\qquad\qquad\qquad\widetilde{\mathrm{Gg}}_{\mathrm{End}^*\left(V\right)}\left(A\right)\ni\xi\longmapsto\left[\vcenter{\xymatrix{\left(V\otimes  A,d+\sigma_1\right)\ar[d]_{e^{\xi}} \\ \left(V\otimes A,d+\sigma_2\right)}}\right] &
\end{eqnarray*}
Map $\nu$ is known to be a weak equivalence: a very rigorous but quite abstract proof can be found in \cite{Pr0} Section 4.1, while a simpler one can be found in \cite{Man2}; see also \cite{FMar} Section 6. Claim (2) and Claim (3) are proved in an entirely analogous way.
\end{proof}
Now we are ready to define coherent derived version of formal Grassmannians and flag functors.
\begin{defn} \label{holim def cmplx grass}
Define the \emph{derived formal total Grassmannian associated to $S$ inside $V$} to be the functor
\begin{equation*}
\mathrm{hoGrass}_{S,V}=\underset{\longleftarrow}{\mathrm{holim}}\left(\mathbb R\mathrm{Def}_{S,V}\doublerightarrow{\text{forgetful map}}{0}\mathbb R\mathrm{Def}_V\right).
\end{equation*} 
In particular $\mathrm{hoGrass}_{S,V}$ is a well-defined derived deformation functor.
\end{defn}
\begin{defn} \label{holim def cmplx}
Define the \emph{formal homotopy flag variety associated to $\left(V,\mathcal F^{\bullet}\right)$} to be the functor
\begin{equation*}
\mathrm{hoFlag}_V^{\mathcal F^{\bullet}}=\underset{\longleftarrow}{\mathrm{holim}}\left(\mathbb R\mathrm{Def}_V^{\mathcal F^{\bullet}}\doublerightarrow{\text{forgetful map}}{0}\mathbb R\mathrm{Def}_V\right).
\end{equation*} 
In particular $\mathrm{hoFlag}_V^{\mathcal F^{\bullet}}$ is a well-defined derived deformation functor.
\end{defn}
\begin{prop} \label{hoFlag dgla}
In the above notation we have that:
\begin{enumerate}
\item The functors $\mathrm{hoGrass}_{S,V}$ and $\mathbb R\mathrm{Def}_{C_{S,V}}$ are weakly equivalent.\footnote{Recall from formula \eqref{tgt dgrass} that $$C_{S,V}:=\underset{\longleftarrow}{\mathrm{holim}}\left(\mathrm{End}^{S}\left(V\right)\doublerightarrow{\chi}{0}\mathrm{End}^*\left(V\right)\right).
$$}
\item The functors $\mathrm{hoFlag}_V^{\mathcal F^{\bullet}}$ and $\mathbb R\mathrm{Def}_{C_V^{\mathcal F^{\bullet}}}$ are weakly equivalent;\footnote{Recall from formula \eqref{tgt hoflag} that $$C^{\mathcal F^{\bullet}}_V:=\underset{\longleftarrow}{\mathrm{holim}}\left(\mathrm{End}^{\mathcal F^{\bullet}}\left(V\right)\doublerightarrow{\chi}{0}\mathrm{End}^*\left(V\right)\right).
$$}
\end{enumerate}
\end{prop}
\begin{proof}
We only prove Claim (1) as Claim (2) is proved in an entirely similar way.\\
We want to show that the functor $\mathrm{hoGrass}_{S,V}$ is weakly equivalent to the Hinich nerve of  
\begin{equation*}
C_{S,V}:=\underset{\longleftarrow}{\mathrm{holim}}\left(\mathrm{End}^S\left(V\right)\doublerightarrow{\mathrm{\text{incl.}}}{0}\mathrm{End}^*\left(V\right)\right).
\end{equation*}
As proved in \cite {Getz} and \cite{Hin1}, the functor $\mathbb R\mathrm{Def}$ (homotopically) commutes with homotopy limits, thus we have that
\begin{equation*} 
\mathbb R\mathrm{Def}_{\underset{\longleftarrow}{\mathrm{holim}}\left(\mathrm{End}^S\left(V\right)\doublerightarrow{\quad}{\quad}\mathrm{End}^*\left(V\right)\right)}
\end{equation*}
is weakly equivalent to 
\begin{equation*}
\underset{\longleftarrow}{\mathrm{holim}}\left(\mathbb R\mathrm{Def}_{\mathrm{End}^S\left(V\right)}\doublerightarrow{\quad}{\quad}\mathbb R\mathrm{Def}_{\mathrm{End}^*\left(V\right)}\right).
\end{equation*}
By Lemma \ref{Def_End} we have that $\mathbb R\mathrm{Def}_V$ is weakly equivalent to $\mathbb R\mathrm{Def}_{\mathrm{End}^*\left(V\right)}$, while the functor $\mathbb R\mathrm{Def}_{S,V}$ is weakly equivalent to $\mathbb R\mathrm{Def}_{\mathrm{End}^S\left(V\right)}$, thus the statement follows by applying Definition \ref{holim def cmplx grass}.
\end{proof}
As a consequence of Proposition \ref{hoFlag dgla} we have that functors $\mathrm{hoGrass}_{S,V}$ and $\mathrm{hoFlag}_V^{\mathcal F^{\bullet}}$ are really derived versions of $\mathrm{Grass}_{S,V}$ and $\mathrm{Flag}_V^{\mathcal F^{\bullet}}$: this is exactly the content of the next result.
\begin{prop} \label{Flag der enh}
In the above notations we have that:
\begin{enumerate}
\item the derived formal total Grassmannian $\mathrm{hoGrass}_{S,V}$ is a derived enhancement of the functor $\mathrm{Grass}_{S,V}$, i.e.
\begin{equation*}
\pi^0\pi_{\leq 0}\mathrm{hoGrass}_{S,V}\simeq\mathrm{Grass}_{S,V};
\end{equation*}
\item the formal homotopy flag variety $\mathrm{hoFlag}_V^{\mathcal F^{\bullet}}$ is a derived enhancement of the formal flag variety $\mathrm{Flag}_V^{\mathcal F^{\bullet}}$, i.e.
\begin{equation*}
\pi^0\pi_{\leq 0}\mathrm{hoFlag}_V^{\mathcal F^{\bullet}}\simeq\mathrm{Flag}_V^{\mathcal F^{\bullet}}.
\end{equation*}
\end{enumerate}
\end{prop}
\begin{proof}
Again, we only prove Claim (1) as Claim (2) is proved in an entirely similar way.\\
Theorem \ref{Grass=Def_C} and Proposition \ref{hoFlag dgla}.1 state that the (homotopy) dgla $C_{S,V}$ represents both the deformation functor $\mathrm{Grass}_{S,V}$ and the derived deformation functor $\mathrm{hoGrass}_{S,V}$; Theorem \ref{compare def theories} implies that the latter has to be a derived enhancement of the former.
\end{proof}
At last recall from \cite{dN2} that there exist two (locally geometric) derived stacks $\mathcal{DG}rass_k\left(V\right)$ and $\mathcal{DF}lag_k\left(V\right)$ which respectively enhance the total Grassmannian variety $\mathrm{Grass}\left(H^*\left(V\right)\right)$ and the (total) flag variety $\mathrm{Flag}\left(H^*\left(V\right)\right)$ to the derived world (see \cite{dN2} Theorem 2.42 and \cite{dN2} Theorem 2.45). As a conclusion of this section we want to show that functors $\mathrm{hoGrass}_{S,V}$ and $\mathrm{hoFlag}_V^{\mathcal F^{\bullet}}$ are formal neighbourhoods of the above derived stacks.
\begin{prop} \label{comp flag fin}
In the above notations we have that:
\begin{enumerate}
\item the derived deformation functor $\mathrm{hoGrass}_{S,V}$ is the formal neighbourhood of the derived stack $\mathcal{DG}rass_k\left(V\right)$ at $\left[S\hookrightarrow V\right]$;
\item the derived deformation functor $\mathrm{hoFlag}_V^{\mathcal F^{\bullet}}$ is the formal neighbourhood of the derived stack $\mathcal{DF}lag_k\left(V\right)$ at $\left(V,\mathcal F^{\bullet}\right)$.
\end{enumerate}
\end{prop}
\begin{proof}
Again, we only prove Claim (1) as Claim (2) is proved in an entirely similar way.\\
As a first step, recall that -- by \cite{dN2} Proposition 2.41 -- the stack $\mathcal{DG}rass_k\left(V\right)$ is an open derived substack of the \emph{big total derived Grassmannian}
\begin{equation*}
\mathcal{DGRASS}_k\left(V\right):=\underset{\longleftarrow}{\mathrm{holim}}\left(\mathbb R\mathcal Sub_k\doublerightarrow{\left[U\hookrightarrow W\right]\mapsto W}{\mathrm{const}_V}\mathbb R\mathcal Perf_k\right)
\end{equation*}
where
\begin{equation*}
\mathbb R\mathcal Perf_k=\underset{n}{\bigcup}\mathbb R\mathcal Perf_k^n\qquad\qquad\qquad\mathbb R\mathcal Sub_k=\underset{n}{\bigcup}\mathbb R\mathcal Sub_k^n
\end{equation*}
are respectively the locally geometric derived stack of perfect complexes and the locally geometric derived stack of perfect subcomplexes: for more details see \cite{dN2} Section 2.2 and 2.3. \\
As a result, it suffices to show that $\mathrm{hoGrass}_{S,V}$ is the formal neighbourhood of $\mathcal{DGRASS}_k\left(V\right)$ at $\left[S\hookrightarrow V\right]$, which by Definition \ref{holim def cmplx grass} amounts to prove that $\mathbb R\mathrm{Def}_{S,V}$ is the formal neighbourhood of $\mathbb R\mathcal Sub_k$ at $\left[S\hookrightarrow V\right]$ and $\mathbb R\mathrm{Def}_V$ is the formal neighbourhood of $\mathbb R\mathcal Perf_k$ at $V$; we show only the first assertion, as the other one is proved analogously. \\
Observe from \cite{dN2} Remark 2.34, \cite{dN2} Corollary 2.32 and \cite{dN2} Theorem 2.33 that the derived stack $\mathbb R\mathcal Sub^n_k$ is obtained as follows: given for all $A\in\mathfrak{Alg}_k$ the functorial simplicial category
\begin{eqnarray*}
&\mathbf M^n_{\mathrm{sub},k}\left(A\right):=&\text{full simplicial subcategory of $\mathfrak{FdgMod}_A$ made of pairs $(\mathscr E,\mathscr C)$} \\ 
&&\text{of perfect complexes of length at most $n$ and for which $E\subseteq C$}
\end{eqnarray*}
construct for all $A\in\mathfrak{dg}_b\mathfrak{Nil}^{\leq 0}_k$ the functorial simplicial category
\begin{eqnarray*}
&\tilde{\mathbf M}^n_{\mathrm{sub},k}\left(A\right):=&\text{full simplicial subcategory of $\mathcal W\left(Fd\mathrm{CART}_k\left(A\right)\right)$ \footnotemark} \\ &&\text{made of pairs $(\mathscr E,\mathscr C)$ such that $(\mathscr E\otimes^{\mathbb L}_AH^0\left(A\right),\mathscr C\otimes^{\mathbb L}_AH^0\left(A\right))$} \\
&&\text{is weakly equivalent to an object in $\mathbf M^n_{\mathrm{sub},k}\left(H^0\left(A\right)\right)$ }
\end{eqnarray*}
\footnotetext{For the precise meaning of $Fd\mathrm{CART}_k$ see \cite{dN2} formula (2.29).}and end up with a functor $\bar W\tilde{\mathbf M}^n_{\mathrm{sub},k}$ which turns out to be via Lurie-Pridham Representability (see \cite{Lu1} or \cite{Pr2bis} for more details) the restriction to $\mathfrak{dg}_b\mathfrak{Nil}^{\leq 0}_k$ of the derived geometric stack $\mathbb R\mathcal Sub^n_k$, which is fully determined then. 
On the other hand recall from the beginning of this section that the derived deformation functor $\mathbb R\mathrm{Def}_{S,V}$ is obtained from the formal affine $\mathrm{dg}_{\geq 0}$-groupoid
\begin{equation*}
\left(hDef_{S,V}\left(A\right),\underline{hDef_{S,V}}\left(A\right)_{\bullet}\right)\qquad\qquad A\in\mathfrak{dgArt}^{\leq 0}_k
\end{equation*}
by applying functor $\bar W$ and (affine) Dold-Kan denormalisation $\breve{\mathbf K}$. \\
Now fix $A\in\mathfrak{dgArt}^{\leq 0}_k$; clearly we have
\begin{equation} \label{neigh}
\tilde{\mathbf M}^n_{\mathrm{sub},k}\left(A\right)\times^{h}_{\tilde{\mathbf M}^n_{\mathrm{sub},k}\left(k\right)}\left\{\left(S,V\right)\right\}\simeq \breve{\mathbf K}\left(\left(hDef_{S,V}\left(A\right),\underline{hDef_{S,V}}\left(A\right)_{\bullet}\right)\right)\quad\quad A\in\mathfrak{dgArt}^{\leq 0}_k
\end{equation}
thus the equivalence of $\mathbb R\mathrm{Def}_{S,V}$ and the formal neighbourhood of $\mathbb R\mathcal Sub^n_k\subseteq\mathbb R\mathcal Sub_k$ at $\left(S,V\right)$ follows by applying functor $\bar W$ to formula \eqref{neigh}.
\end{proof}
\subsection{Derived Deformations of $k$-Schemes}
Now we want to describe the functor $\mathbb R\mathrm{Def}_X$ which parametrises derived deformations of the scheme $X$: the idea consists of deforming the scheme $X$ through derived schemes instead of ordinary schemes. There are a variety of equivalent definitions of derived scheme (in particular see \cite{Lu1} Definition 4.5.1 and \cite{TV} Chapter 2.2 for the two most standard ways to look at it); the one we are about to recall probably is not the most elegant, but it is definitely the handiest one to make actual computations. As a matter of fact, by \cite{Pr2} Theorem 6.42 a \emph{derived scheme} $S$ over $k$ can be seen as a pair $\left(\pi^0S,\mathscr O_{S,*}\right)$, where $\pi^0S$ is an honest $k$-scheme and $\mathscr O_{S,*}$ is a presheaf of differential graded commutative algebras in non-positive degrees on the site of affine opens of $\pi^0S$ such that:
\begin{itemize}
\item the (cohomology) presheaf $\mathcal H^0\left(\mathscr O_{S,*}\right)\simeq\mathscr O_{\pi^0S}$;
\item the (cohomology) presheaves $\mathcal H^n\left(\mathscr O_{S,*}\right)$ are quasi-coherent $\mathscr O_{\pi^0S}$-modules.
\end{itemize}
Also, recall from \cite{Pr2} that a morphism $f:A\rightarrow B$ in $\mathfrak{dgAlg}_k^{\leq 0}$ is said to be \emph{homotopy flat} if 
\begin{equation*}
H^0\left(f\right):H^0\left(A\right)\longrightarrow H^0\left(B\right)
\end{equation*}
is flat and the maps
\begin{equation*}
H^i\left(A\right)\otimes_{H^0\left(A\right)}H^0\left(B\right)\longrightarrow H^i\left(B\right)
\end{equation*}
are isomorphisms for all $i$; moreover a very useful characterisation says that $f$ is homotopy flat if and only if $B\otimes^{\mathbb L}_AH^0\left(A\right)$ is (weakly equivalent to) a discrete flat $H^0\left(A\right)$-algebra: for a proof see \cite{Pr3} Lemma 3.13.\footnote{\cite{Pr3} and \cite{Pr2} actually deal with homotopy flatness in terms of simplicial and $\mathrm{dg}_{\geq 0}$ chain algebras; nevertheless all definitions and arguments readily adapt to cochain algebras in non-positive degrees.}\\
Now define a derived deformation of the scheme $X$ over $A\in\mathfrak{dgArt}^{\leq 0}$ to be a homotopy pull-back diagram of derived schemes
\begin{equation*}
\xymatrix{\ar@{} |{\Box^h}[dr]X\ar@{^{(}->}[r]^i\ar[d] & \mathcal X\ar[d]^p \\
\mathrm{Spec}\left(k\right)\ar[r] & \mathbb R\mathrm{Spec}\left(A\right)}
\end{equation*}
where the map $p$ is homotopy flat; equivalently such a deformation can be seen as a morphism $\mathscr O_{A,*}\rightarrow\mathscr O_X$ of presheaves of differential graded commutative algebras over $A$ such that: 
\begin{enumerate}
\item $\mathscr O_{A,*}$ is homotopy flat;
\item the induced $k$-linear morphism $\mathscr O_{A,*}\otimes_A^{\mathbb L} k\rightarrow\mathscr O_X$ is a weak equivalence;
\item the morphism $\mathscr O_{A,*}\rightarrow\mathscr O_X$ is surjective;
\item $\mathscr O_{A,*}$ is cofibrant.
\end{enumerate}
\begin{rem} \label{fib-cof X}
In the above notations, Condition (1) and Condition (2) are proper deformation-theoretic conditions, which resemble the ones characterising underived deformations of schemes (see Section 2.1), while Condition (3) and Condition (4) are fibrancy-cofibrancy conditions, which are needed in order to ensure that certain maps of derived deformation functors which will arise in the rest of the paper are well-defined.
\end{rem}
Now consider the formal groupoid 
\begin{eqnarray*}
\mathrm{Del}_X:&\mathfrak{dgArt}_k^{\leq 0}&\xrightarrow{\hspace*{1cm}}\mathfrak{Grpd} \\
&A&\longmapsto\mathrm{Del}_X\left(A\right)
\end{eqnarray*}
defined by the formulae
\begin{equation*}
\mathrm{Del}_X\left(A\right):=\left\{\text{(derived) deformations of }X\text{ over }A\right\}
\end{equation*} 
and for all $\left(\mathscr O_{A,*}\overset{\varphi}{\rightarrow}\mathscr O_X\right),\left(\mathscr O'_{A,*}\overset{\phi}{\rightarrow}\mathscr O_X\right)\in\mathrm{Del}_X\left(A\right)$
\begin{equation} \label{Del_X}
\mathrm{Hom}_{\mathrm{Del}_X\left(A\right)}\left(\varphi,\phi\right):=\left\{\Psi\in\mathrm{Hom}_{A}^0\left(\mathscr O_{A,*},\mathscr O'_{A,*}\right)\text{ s.t. }\phi\circ\Psi=\varphi,\Psi\equiv\mathrm{Id}\;\left(\mathrm{mod}\,\mathfrak m_A\right)\right\}
\end{equation}
\begin{rem}
In the notations of formula \eqref{Del_X}, notice that the condition $\Psi\equiv\mathrm{Id}\;\left(\mathrm{mod}\,\mathfrak m_A\right)$ ensures that $\mathrm{Del}_X\left(A\right)$ is a groupoid for all $A\in\mathfrak{dgArt}_k^{\leq 0}$; roughly speaking, the formal groupoid $\mathrm{Del}_X$ can be thought as some sort of (derived) Deligne groupoid associated to the scheme $X$, meaning that its role is intended to formally resemble the one played by the (derived) Deligne groupoid associated do a differential graded Lie algebra, which we described in Section 3.1. 
\end{rem}
Now consider the functor
\begin{equation*}
\mathrm{BDel}_X:\mathfrak{dgArt}_k^{\leq 0}\xrightarrow{\hspace*{1cm}}\mathfrak{sSet}
\end{equation*}
given by the nerve of $\mathrm{Del}_X$ and define 
\begin{equation*}
\mathbb R\mathrm{Def}_X:\mathfrak{dgArt}_k^{\leq 0}\xrightarrow{\hspace*{1cm}}\mathfrak{sSet}
\end{equation*}
to be the right derived functor of $\mathrm{BDel}_X$.
The definition of $\mathrm{Del}_X$ implies immediately that this is a derived pre-deformation functor, thus -- by \cite{Pr2bis} Theorem 3.16 -- $\mathbb R\mathrm{Def}_X$ turns to be a derived deformation functor. \\
Now let us briefly look at global derived moduli of schemes; consider the assignment
\begin{eqnarray*}
\mathbf{Stack}_{n/k}^0:&\mathfrak{Alg}_k&\xrightarrow{\hspace*{1cm}}\mathfrak{sCat} \\
&A&\longmapsto\mathbf{Stack}_{n/k}^0\left(A\right):=\text{ simplicial category of algebraic $n$-spaces over } A
\end{eqnarray*} 
which is the simplicial (underived) moduli functor classifying (underived) $0$-stacks of dimension $n$ over $k$. Pridham has shown that such a functor induces a derived stack $\mathcal D\mathcal Sch_{n/k}$ parametrising derived schemes over $k$ of dimension $n$: see \cite{Pr3} Example 3.36 for a detailed construction. Unfortunately the stack $\mathcal D\mathcal Sch_{n/k}$ is far too large to be geometric, nonetheless Pridham constructed many interesting geometric substacks of its: see \cite{Pr3} Section 3 for more details.
\begin{rem} \label{DSch}
The derived deformation functor $\mathbb R\mathrm{Def}_X$ is the formal neighbourhood of the derived stack $\mathcal D\mathcal Sch_{d/k}$ at $X$; in particular, it follows -- using either \cite{Pr2} Theorem 8.8 or \cite{Pr3} Theorem 10.8 -- that
\begin{equation*}
H^i\left(\mathrm{BDel}_X\right)\simeq H^i\left(\mathbb R\mathrm{Def}_X\right)\simeq\mathrm{Ext}^{i+1}_{\mathscr O_X}\left(\mathbb L_{X/k},\mathscr O_X\right)
\end{equation*}
\end{rem}
\begin{thm} \label{main}
The functors $\mathbb R\mathrm{Def}_X$ and $\mathbb R\mathrm{Def}_{KS_X}$ are weakly equivalent.
\end{thm}
\begin{proof}
We want to construct a natural transformation
\begin{equation*}
\mathbb R\mathrm{Def}_{KS_X}\xrightarrow{\hspace*{1cm}}\mathbb R\mathrm{Def}_X
\end{equation*}
providing a weak equivalence between such derived deformation functors, i.e. an isomorphism on the level of homotopy categories. \\
Again, by Theorem \ref{Del grpd} and the definition of $\mathbb R\mathrm{Def}_X$ it is enough to define such a morphism on $\mathrm{BDel}_{KS_X}$, thus define the map
\begin{eqnarray} \label{mu}
&\qquad\qquad\mu:\mathrm{BDel}_{KS_X}&\xrightarrow{\hspace*{1.5cm}}\mathrm{BDel}_X \nonumber \\
\text{for all }A\in\mathfrak{dgArt}^{\leq 0}_k&\mathrm{MC}_{KS_X}\left(A\right)\ni x&\longmapsto\left[\tau^{\leq 0}\mathbb R\mathscr O_A\left(x\right)\rightarrow\tau^{\leq 0}\mathbb R\mathscr O_k\left(x\right)\simeq\mathscr O_X\right] \nonumber \\
&\mathrm{Gg}_{KS_X}\left(A\right)\ni\xi&\longmapsto\left[\vcenter{\xymatrix{\tau^{\leq 0}\mathbb R\mathscr O_A\left(x_1\right)\ar[dd]_{e^{\xi}}\ar[dr] & \\ & \mathscr O_X \\ \tau^{\leq 0}\mathbb R\mathscr O_A\left(x_2\right)\ar[ur] &} }\right]
\end{eqnarray}
where $\mathbb R\mathscr O_A\left(x\right):=\left(\mathscr A^{0,*}_X\otimes A,\bar{\partial}+l_{x}\right)$, $l$ being the Lie derivative (i.e. the differential of the contraction map) and the map $\tau^{\leq 0}\mathbb R\mathscr O_A\left(x\right)\rightarrow\tau^{\leq 0}\mathbb R\mathscr O_k\left(x\right)$ is induced by $A\twoheadrightarrow \nicefrac{A}{\mathfrak m_A}\simeq k$. Notice that the complex $\mathbb R\mathscr O_A\left(x\right)$ is cofibrant, since it is bounded above and the underlying graded module $\mathscr A^{0,*}_X\otimes A$ is projective. Furthermore observe that the surjectivity of the natural map $A\twoheadrightarrow \nicefrac{A}{\mathfrak m_A}$ together with the surjectivity of the canonical morphism $\tau^{\leq 0}\mathbb R\mathscr O_A\left(x\right)\twoheadrightarrow H^0\left(\tau^{\leq 0}\mathbb R\mathscr O_A\left(x\right)\right)$ -- due in turn to the fact that the complex $\tau^{\leq 0}\mathbb R\mathscr O_A\left(x\right)$ lives in non-positive degrees -- ensures that the morphism $\tau^{\leq 0}\mathbb R\mathscr O_A\left(x\right)\rightarrow\mathscr O_X$ is surjective.\\
In order to show that map \ref{mu} is well-defined, we have to check that  
\begin{equation*}
\left[\tau^{\leq 0}\mathbb R\mathscr O_A\left(x\right)\rightarrow\tau^{\leq 0}\mathbb R\mathscr O_k\left(x\right)\simeq\mathscr O_X\right]
\end{equation*}
actually determines a derived deformation of the scheme $X$, i.e. we need to prove that $\tau^{\leq 0}\mathbb R\mathscr O_A\left(x\right)$ is homotopy flat over $A$ and $\tau^{\leq 0}\mathbb R\mathscr O_A\left(x\right)\otimes^{\mathbb  L}_Ak$ is weakly equivalent to $\mathscr O_X$ as complexes of presheaves of differential graded commutative $k$-algebras: this essentially means to verify that $\mathbb R\mathscr O_A\left(x\right)\otimes^{\mathbb L}_A H^0\left(A\right)$ is flat over $H^0\left(A\right)$. \\
Let us first prove that $\tau^{\leq 0}\mathbb R\mathscr O_A\left(x\right)$ is weakly equivalent to $\mathbb R\mathscr O_A\left(x\right)$. Filter the latter complex by powers of the maximal ideal $\mathfrak m_A$ of $A$, i.e. define the filtered complex $\left(\mathbb R\mathscr O_A\left(x\right),\mathcal F^{\bullet}\right)$ through the relation
\begin{equation*}
\mathcal F^p\mathbb R\mathscr O_A\left(x\right):=\mathfrak m_A^p\mathbb R\mathscr O_A\left(x\right)
\end{equation*} 
and take the associated graded object
\begin{equation} \label{assoc grad}
\mathrm{Gr}^p\left(\mathcal F^{\bullet}\right):=\frac{\mathfrak m_A^p\mathbb R\mathscr O_A\left(x\right)}{\mathfrak m_A^{p+1}\mathbb R\mathscr O_A\left(x\right)}.
\end{equation}
Notice that formula \eqref{assoc grad}, so
\begin{equation*}
\mathrm{Gr}^p\left(\mathcal F^{\bullet}\right)\simeq\mathscr A^{0,*}_X\otimes\frac{\mathfrak m_A^p}{\mathfrak m_A^{p+1}}.
\end{equation*}
Now consider the spectral sequence
\begin{equation}\label{spec seq}
H^{p+q}\left(\mathscr A_X^{0,*}\otimes\frac{\mathfrak m_A^p}{\mathfrak m_A^{p+1}}\right)\simeq\bigoplus_{i+j=p+q}\left(H^i\left(\mathscr A^{0,*}_X\right)\otimes H^j\left(\frac{\mathfrak m_A^p}{\mathfrak m_A^{p+1}}\right)\right)\Longrightarrow H^{p+q}\left(\mathbb R\mathscr O_A\left(x\right)\right)
\end{equation}
which converges by the \emph{Classical Convergence Theorem} (see \cite{We} Theorem 5.5.1); note that $H^j\left(\frac{\mathfrak m_A^p}{\mathfrak m_A^    {p+1}}\right)=0$ when $j>0$ and, since the {} ``Dolbeaut'' resolution $\mathscr A^{0,*}_X\hookleftarrow \mathscr O_X$ provides a weak equivalence between $\mathscr A^{0,*}_X$ and $\mathscr O_X$ in the category of $\mathscr O_X$-modules in complexes, also $H^i\left(\mathscr A^{0,*}_X\right)=0$ when $i>0$: this means that at least one of these two terms vanishes whenever $i+j>0$, so the convergence of spectral sequence \eqref{spec seq} implies that
\begin{equation*}
H^n\left(\mathbb R\mathscr O_A\left(x\right)\right)=0\qquad\forall n>0.
\end{equation*} 
In particular $\mathbb R\mathscr O_A\left(x\right)$ and $\tau^{\leq 0}\mathbb R\mathscr O_A\left(x\right)$ are weakly equivalent. \\
Now we want to prove that $\tau^{\leq 0}\mathbb R\mathscr O_A\left(x\right)\otimes^{\mathbb L}_AH^0\left(A\right)$ is flat over $H^0\left(A\right)$; first notice that 
\begin{equation*}
\tau^{\leq 0}\mathbb R\mathscr O_A\left(x\right)\otimes^{\mathbb L}_AH^0\left(A\right)\approx\mathbb R\mathscr O_A\left(x\right)\otimes^{\mathbb L}_AH^0\left(A\right)\approx\mathbb R\mathscr O_A\left(x\right)\otimes_AH^0\left(A\right)\footnote{The symbol $\approx$ stands for {}``weakly equivalent''.}
\end{equation*}
so it is enough to show that $\mathbb R\mathscr O_A\left(x\right)\otimes_AH^0\left(A\right)$ is flat over $H^0\left(A\right)$. In order to prove this let $M$ be any $H^0\left(A\right)$-module, consider the complex 
\begin{equation*}
\left(\mathbb R\mathscr O_A\left(x\right)\otimes_AH^0\left(A\right)\right)\otimes_{H^0\left(A\right)}M\simeq\mathbb R\mathscr O_A\left(x\right)\otimes_AM.
\end{equation*}
and filter it by powers of the maximal ideal $\mathfrak m_{H^0\left(A\right)}$ of $H^0\left(A\right)$, i.e. define the filtered complex $\left(\mathbb R\mathscr O_A\left(x\right)\otimes_AM,\mathcal F^{\bullet}\right)$ through the relation
\begin{equation*}
\mathcal F^p\left(\mathbb R\mathscr O_A\left(x\right)\otimes_AM\right):=\mathfrak m_{H^0\left(A\right)}^p\left(\mathbb R\mathscr O_A\left(x\right)\otimes_AM\right).
\end{equation*}
As before, the associated graded object is
\begin{equation*}
\mathrm{Gr}^p\left(\mathcal F\right):=\frac{\mathfrak m_{H^0\left(A\right)}^p\left(\mathbb R\mathscr O_A\left(x\right)\otimes_AM\right)}{\mathfrak m_{H^0\left(A\right)}^{p+1}\left(\mathbb R\mathscr O_A\left(x\right)\otimes_AM\right)}\simeq\mathscr A^{0,*}_X\otimes\frac{\mathfrak m^p_{H^0\left(A\right)}M}{\mathfrak m^{p+1}_{H^0\left(A\right)}M}
\end{equation*}
and there is a spectral sequence
\begin{equation} \label{spec seq 2}
H^{p+q}\left(\mathscr A^{*,0}_X\otimes\frac{\mathfrak m_{H^0\left(A\right)}^pM}{\mathfrak m_{H^0\left(A\right)}^{p+1}M}\right)\Longrightarrow H^{p+q}\left(\mathbb R\mathscr O_A\left(x\right)\otimes_AM\right)
\end{equation}
which still converges because of the Classical Convergence Theorem. Of course 
\begin{equation*}
H^{p+q}\left(\mathscr A^{*,0}_X\otimes\frac{\mathfrak m_{H^0\left(A\right)}^pM}{\mathfrak m_{H^0\left(A\right)}^{p+1}M}\right)\simeq\bigoplus_{i+j=p+q}\left(H^i\left(\mathscr A^{0,*}_X\right)\otimes H^j\left(\frac{\mathfrak m_{H^0\left(A\right)}^pM}{\mathfrak m_{H^0\left(A\right)}^{p+1}M}\right)\right)
\end{equation*}
and $H^j\left(\frac{\mathfrak m_{H^0\left(A\right)}^pM}{\mathfrak m_{H^0\left(A\right)}^{p+1}M}\right)=0$ for all $j\neq 0$, thus the convergence of spectral sequence \eqref{spec seq 2} implies
\begin{equation*}
\mathrm{Tor}_n^A\left(\mathbb R\mathscr O_A\left(x\right),M\right)\simeq H^{-n}\left(\mathbb R\mathscr O_A\left(x\right)\otimes_AM\right)=0\qquad\forall n\neq0
\end{equation*}
which gives us the flatness of $\tau^{\leq 0}\mathbb R\mathscr O_A\left(x\right)\otimes^{\mathbb L}_AH^0\left(A\right)$ over $H^0\left(A\right)$. Notice that the same computation also ensures that the map 
\begin{equation*}
\left[\tau^{\leq 0}\mathbb R\mathscr O_A\left(x\right)\rightarrow\tau^{\leq 0}\mathbb R\mathscr O_k\left(x\right)\simeq\mathscr O_X\right]
\end{equation*} 
is quasi-smooth.\\
Now we want to prove that map \ref{mu} is a weak equivalence of derived deformation functors; by \cite{Pr1} Corollary 1.49 it suffices to check that such a map induces isomorphisms on generalised tangent spaces, so consider the morphisms
\begin{equation*}
H^i\left(\mu\right):H^i\left(\mathrm{BDel}_{KS_X}\right)\longrightarrow H^i\left(\mathrm{BDel}_X\right)\qquad\quad i\geq -1
\end{equation*}
and notice that higher tangent maps in larger negative degrees vanish as $KS_X$ lives only in non-negative degrees. \\
For all $i\geq 0$ we have the chain of canonical identifications
\begin{equation*}
H^i\left(\mathbb R\mathrm{Def}_{KS_X}\right)\simeq H^i\left(\mathrm{BDel}_{KS_X}\right)\simeq H^{i+1}\left(KS_X\right)\simeq H^{i+1}\left(X,\mathscr A_X^{0,*}\left(\mathscr T_X\right)\right)\simeq H^{i+1}\left(X,\mathscr T_X\right)
\end{equation*}
where the first and the second isomorphism come from Remark \ref{Del grpd coho}, the third one is true just by definition and the last one is given by the Dolbeaut Theorem. In the same fashion, there is also a chain of canonical isomorphisms
\begin{eqnarray*}
&H^i\left(\mathrm{BDel}_X\right)\simeq\mathrm{Ext}^{i+1}_{\mathscr O_X}\left(\mathbb L_{X/k},\mathscr O_X\right)\simeq\mathrm{Hom}_{\mathrm D\left(X\right)}\left(\Omega^1_{X/k}\otimes \mathscr O_X,\mathscr O_X\left[-i-1\right]\right)\simeq& \\
&\mathrm{Hom}_{\mathrm D\left(X\right)}\left(\mathscr O_X,\mathcal Hom\left(\Omega^1_{X/k},\mathscr O_X\right)\left[-i-1\right]\right)\simeq\mathrm{Hom}_{\mathrm D\left(X\right)}\left(\mathscr O_X,\mathscr T_X\left[-i-1\right]\right)\simeq\mathrm{Ext}_{\mathscr O_X}^{i+1}\left(\mathscr O_X,\mathscr T_X\right)&
\end{eqnarray*}
where the first isomorphism -- as we discussed before -- comes from the fact that $\mathbb R\mathrm{Def}_X$ is the formal neighbourhood of a derived stack of schemes, the third one is true by adjunction, while all the other ones directly follow from definitions. \\
Finally, for all $i\geq 0$ we see that the map $H^i\left(\mu\right)$ is
\begin{eqnarray} \label{H^i}
H^i\left(\mu\right):&H^{i+1}\left(X,\mathscr T_X\right)&\xrightarrow{\hspace*{0.3cm}}\quad\;\mathrm{Ext}^{i+1}_{\mathscr O_X}\left(\mathscr O_X,\mathscr T_X\right) \nonumber \\
&\xi&\longmapsto\left(\mathscr O_X\overset{\xi}{\longrightarrow}\mathscr T_X\left[-i-1\right]\right)
\end{eqnarray} 
where the (cohomology class of the) degree $i$ morphism $\mathscr O_X\overset{\xi}{\longrightarrow}\mathscr T_X\left[-i-1\right]$ is nothing but the map\footnote{Of course, there is some abuse of notation in this sentence.} induced in $\mathrm D\left(X\right)$ by the cocycle $\xi$;
on the other hand -- again by using Remark \ref{Del grpd coho} -- the map $H^{-1}\left(\mu\right)$ turns out to be
\begin{eqnarray} \label{H^-1}
H^{-1}\left(\mu\right):&H^0\left(KS_X\right)\simeq\mathrm{Stab}_{\mathrm{Gg}_{\Gamma\left(X,\mathscr A^{0,0}_X\left(\mathscr T_X\right)\right)}\left(\frac{k\left[\varepsilon\right]}{\left(\varepsilon^2\right)}\right)}\left(0\right)&\xrightarrow{\hspace*{0.3cm}} H^0\left(X,\mathscr T_X\right)\simeq\mathrm{Ext}^1_{\mathscr O_X}\left(\mathscr O_X,\mathscr T_X\right)\nonumber \\
&\qquad\qquad\qquad\quad\mathrm{Id}+\xi&\longmapsto \qquad\;\xi.
\end{eqnarray}
Both map \eqref{H^-1} and map \eqref{H^i} are clearly isomorphisms, so this completes the proof.
\end{proof}
\begin{rem} \label{Def_X der enh}
$\mathbb R\mathrm{Def}_X$ is a derived enhancement of $\mathrm{Def}_X$, i.e.
\begin{equation*}
\pi^0\pi_{\leq 0}\mathbb R\mathrm{Def}_X\simeq\mathrm{Def}_X.
\end{equation*}
\end{rem}
\subsection{The Geometric Fiorenza-Manetti-Martinengo Period Map}
Now we have all the ingredients to give a geometric interpretation of the map $\mathrm{FMM}$ described in Definition \ref{alg Fio-Man-Mar}.
\begin{defn} \label{g FMM map}
Define the \emph{(universal) geometric Fiorenza-Manetti-Martinengo local period map} to be the morphism of derived deformation functors
\begin{equation*}
\mathbb R\mathcal P:\mathbb R\mathrm{Def}_X\xrightarrow{\hspace*{0.75cm}}\mathrm{hoFlag}^{F^{\bullet}}_{\mathbb R\Gamma\left(X,\Omega^*_{X/k}\right)}
\end{equation*}
identified by the universal morphism of derived deformation functors under the map of derived pre-deformation functors given for all $A\in\mathfrak{dgArt}^{\leq 0}_k$ by
\begin{eqnarray*} 
&\mathrm{BDel}_X&\xrightarrow{\hspace*{0.75cm}}\qquad\mathrm{hoFlag}^{F^{\bullet}}_{\mathbb R\Gamma\left(X,\Omega^*_{X/k}\right)} \\
&\left[\mathscr O_{A,*}\overset{\varphi}{\longrightarrow}\mathscr O_X\right]\approx\left[\vcenter{\xymatrix{\ar@{} |{\Box^h}[dr]X\ar@{^{(}->}[r]\ar[d] & \mathcal X\ar[d] \\
\mathrm{Spec}\left(k\right)\ar[r] & \mathbb R\mathrm{Spec}\left(A\right)}}\right]&\mapsto \left[\left(\left(\mathbb R\Gamma\left(\pi^0\mathcal X,\Omega^*_{\mathcal X/A}\right),F^{\bullet}\right),\tilde{\varphi}\right)\right] \\
&\left[\vcenter{\xymatrix{\mathscr O_{A,*}'\ar[dd]_{\Psi}\ar[dr]^{\varphi'} & \\ & \mathscr O_X \\ \mathscr O_{A,*}''\ar[ur]_{\varphi''} &} }\right]& \mapsto\left[\vcenter{\xymatrix{\left(\left(\mathbb R\Gamma\left(\pi^0\mathcal X,\Omega^*_{\mathcal X'/A}\right),F^{\bullet}\right),\tilde{\varphi'}\right)\ar[d]^{\tilde{\Psi}} \\ \left(\left(\mathbb R\Gamma\left(\pi^0\mathcal X,\Omega^*_{\mathcal X''/A}\right),F^{\bullet}\right),\tilde{\varphi''}\right)}}\right]
\end{eqnarray*}
where 
\begin{itemize}
\item $\tilde{\varphi}$ is the derived globalisation of the natural $A$-linear map extending $\varphi$ to the algebraic De Rham complex;
\item $\tilde{\Psi}$ is constructed by using the same universal property;
\item the complex 
\begin{equation*}
\mathbb R\Gamma\left(\pi^0\mathcal X,\Omega^*_{\mathcal X/A}\right):=\underset{i}{\prod}\left(\bigwedge^i\mathbb L_{\mathcal X/A}\right) 
\end{equation*}
is sometimes known as \emph{derived de Rham complex} and the Hodge filtration over it is just
\begin{equation*}
F^p\mathbb R\Gamma\left(\pi^0\mathcal X,\Omega^*_{\mathcal X/A}\right):=\underset{i\geq p}{\prod}\left(\bigwedge^i\mathbb L_{\mathcal X/A}\right).
\end{equation*}
\end{itemize}
\end{defn}
\begin{rem}
The fibrant-cofibrant replacement properties pointed out in Remark \ref{fib-cof flag} and Remark \ref{fib-cof X} ensure that the geometric Fiorenza-Manetti-Martinengo local period map described in Definition \ref{g FMM map} is well-defined.
\end{rem}
In the end all constructions and results we have discussed so far sum up in the following theorem.
\begin{thm} \label{univ lpm ddf}
The diagram of derived deformation functors and (Schlessinger's) deformation functors
\begin{equation*}
\xymatrix{& & \mathbb R\mathrm{Def}_{KS_X}\ar[rrr]^{\mathrm{FMM}\qquad\quad}\ar[dll]^{\sim}\ar[ddd]_{\pi^0\pi_{\leq 0}} & & &  \mathbb R\mathrm{Def}_{\frac{\mathrm{End}^*\left(\mathbb R\Gamma\left(X,\Omega^*_{X/k}\right)\right)}{\mathrm{End}^{\geq 0}\left(\mathbb R\Gamma\left(X,\Omega^*_{X/k}\right)\right)}\left[-1\right]}\ar[dl]^{\sim}\ar[ddd]^{\pi^0\pi_{\leq 0}} \\
\mathbb R\mathrm{Def}_X\ar[rrrr]^{\mathbb R\mathcal P}\ar[ddd]_{\pi^0\pi_{\leq 0}} & & & & \mathrm{hoFlag}^{F^{\bullet}}_{\mathbb R\Gamma\left(X,\Omega^*_{X/k}\right)}\ar[ddd]^{\pi^0\pi_{\leq 0}} & \\
& & & & &\\
& &\mathrm{Def}_{KS_X}\ar[rrr]^{\mathrm{FM}\qquad\quad}\ar[dll]^{\sim} & & & \mathrm{Def}_{\frac{\mathrm{End}^*\left(H^*\left(X,k\right)\right)}{\mathrm{End}^{\geq 0}\left(H^*\left(X,k\right)\right)}\left[-1\right]}\ar[dl]^{\sim} \\
\mathrm{Def}_X\ar[rrrr]^{\mathcal P} & & & & \mathrm{Flag}_{H^*\left(X,k\right)}^{F^{\bullet}} &
}
\end{equation*}
commutes up to isomorphism; in particular the morphisms $\mathbb R\mathcal P$ and $\mathrm{FMM}$ are equivalent.
\end{thm}
\begin{proof}
Notice that:
\begin{itemize}
\item the commutativity of the bottom diagram follows from Theorem \ref{Fiorenza-Manetti} and Theorem \ref{Fiorenza-Martinengo}; 
\item the commutativity (up to isomorphism) of the back diagram corresponds to Theorem \ref{Fiorenza-Martinengo};
\item the commutativity (up to isomorphism) of the front diagram follows immediately from Remark \ref{Def_X der enh}, Proposition \ref{Flag der enh}.2 and the definitions of the maps $\mathcal P$ and $\mathbb R\mathcal P$;
\item the commutativity (up to isomorphism) of the left hand diagram is obtained by combining Theorem \ref{Donatella}, Theorem \ref{main} and Remark \ref{Def_X der enh};
\item the commutativity (up to isomorphism) of the right hand diagram is obtained by combining Corollary \ref{und Flag dgla}, Proposition \ref{hoFlag dgla}.2 and Proposition \ref{Flag der enh}.2. 
\end{itemize}
As regards the top diagram, again by Theorem \ref{Del grpd} it suffices to verify its commutativity up to isomorphism on $\mathrm{BDel}_{KS_X}$; moreover Proposition \ref{hoFlag dgla}.2 and Definition \ref{holim def cmplx} say that the derived deformation functors $\mathbb R\mathrm{Def}_{\frac{\mathrm{End}^*\left(\mathbb R\Gamma\left(X,\Omega^*_{X/k}\right)\right)}{\mathrm{End}^{\geq 0}\left(\mathbb R\Gamma\left(X,\Omega^*_{X/k}\right)\right)}\left[-1\right]}$ and $\mathrm{hoFlag}^{F^{\bullet}}_{\mathbb R\Gamma\left(X,\Omega^*_{X/k}\right)}$ are homotopy fibres, so it is enough to check that the diagrams\footnote{here maps $\mu$ and $\nu$ are the morphisms defined in Theorem \ref{main} and Lemma \ref{Def_End} respectively.}
\begin{equation} \label{diagr1}
\xymatrix{\mathrm{BDel}_{KS_X}\ar[rrrr]^{\mathbb R\mathrm{Def}\left(l\right)\qquad}\ar[d]^{\wr}_{\mu} & & & & \mathrm{BDel}_{\mathrm{End}^*\left(\mathbb R\Gamma\left(X,\Omega^*_{X/k}\right)\right)}\ar[d]_{\wr}^{\nu} \\
\mathbb R\mathrm{Def}_X\ar[rrrr]^{\left(X,\mathscr O_{A,*}\right)=:\mathcal X\mapsto\mathbb R\Gamma\left(\pi^0\mathcal X,\Omega^*_{\mathcal X/A}\right)\qquad\qquad} & & & & \mathbb R\mathrm{Def}_{\mathbb R\Gamma\left(X,\mathbb R\Gamma\left(X,\Omega^*_{X/k}\right)\right)} }
\end{equation}
and
\begin{equation} \label{diagr2}
\xymatrix{\mathrm{BDel}_{KS_X}\ar[rrrrr]^{\mathbb R\mathrm{Def}\left(l\right)\qquad}\ar[d]^{\wr}_{\mu} & & & & & \mathrm{BDel}_{\mathrm{End}^{\geq 0}\left(\mathbb R\Gamma\left(X,\Omega^*_{X/k}\right)\right)}\ar[d]_{\wr}^{\nu} \\
\mathbb R\mathrm{Def}_X\ar[rrrrr]^{\left(X,\mathscr O_{A,*}\right)=:\mathcal X\mapsto\left(\mathbb R\Gamma\left(\pi^0\mathcal X,\Omega^*_{\mathcal X/A}\right),F^{\bullet}\right)\qquad\qquad\qquad\qquad} & & & & &  \mathbb R\mathrm{Def}_{F^{\bullet}\mathbb R\Gamma\left(X,\mathbb R\Gamma\left(X,\Omega^*_{X/k}\right)\right),\mathbb R\Gamma\left(X,\mathbb R\Gamma\left(X,\Omega^*_{X/k}\right)\right)} } \qquad\quad
\end{equation} 
commute up to isomorphism. We are only going to show the commutativity of diagram \eqref{diagr1}, as the commutativity of diagram \eqref{diagr2} is verified by a similar argument.\\
Let us walk along its arrows: for all $A\in\mathfrak{dgArt}_k$ an element $x\in\mathrm{MC}_{KS_X}\left(A\right)$ maps through $\mu$ to $\left[\mathcal X\rightarrow\mathbb R\mathrm{Spec}\left(A\right)\right]$ -- where $\mathcal X=\left(X,\mathbb R\mathscr O_A\left(x\right)\right)$ -- and in turn this is sent to the complex 
\begin{equation} \label{cmplx 1}
\mathbb R\Gamma\left(\pi^0\mathcal X,\Omega^*_{\mathcal X/A}\right)\approx\mathbb R\Gamma\left(X,\Omega^*_{\mathbb R\mathscr O_A\left(x\right)/A}\right)
\end{equation}
which is an honest derived deformation over $A$ of the algebraic De Rham complex $\mathbb R\Gamma\left(X,\Omega^*_{X/k}\right)$; on the other side, the vector $x$ is sent to the derivation $l_x$ and -- proceeding down along map $\nu$ -- this determines the complex
\begin{equation} \label{cmplx 2}
\left(\mathbb R\Gamma\left(X,\Omega^*_{X/k}\right)\otimes A,d+l_x\right). 
\end{equation}
We claim that complexes \eqref{cmplx 1} and \eqref{cmplx 2} are quasi-isomorphic: more precisely, we assert that the natural zig-zag
\begin{equation} \label{resol}
\xymatrix{\mathbb R\Gamma\left(X,\Omega^*_{\mathbb R\mathscr O_A\left(x\right)/A}\right) & \left(\Gamma\left(X,\mathscr A^{*,*}_X\otimes A\right),\partial + \left(\bar{\partial}+l_x\right)\right)\ar[l] & \left(\mathbb R\Gamma\left(X,\Omega^*_{X/k}\right)\otimes A,d+l_x\right)\ar@{=}[l] }
\end{equation}
is a chain of quasi-isomorphisms. The right-hand morphism in diagram \eqref{resol} is essentially given by the resolution \eqref{RGamma}. As regards the left-hand one, this is constructed in the following way: consider the standard Dolbeaut resolution $\mathscr A^{0,*}_X\hookleftarrow\mathscr O_X$ and twist it through the derivation $l_x$, so get a map $\mathbb R\mathscr O_A\left(x\right)\leftarrow\mathscr O_X\otimes A$ and hence a morphism $\Omega^*_{\mathbb R\mathscr O_A\left(x\right)/A}\leftarrow\Omega^*_X\otimes A$; now just recall that $\mathscr A^{*,*}_X\simeq\Omega^*_X\otimes_{\mathscr O_X}\mathscr A^{0,*}_X$: this provides us with a natural map $\Omega^*_{\mathbb R\mathscr O_A\left(x\right)}\leftarrow\mathscr A^{*,*}_X\otimes A$, whose globalisation finally gives us the left-hand map in diagram \eqref{resol}.\\
Now denote 
\begin{equation*}
\mathbb R\mathscr O_A\left(x\right)\left(n\right):=\left(\mathscr A^{n,*}_X\otimes A,\bar{\partial}+l_x\right)
\end{equation*}
and observe that to show that the zig-zag \eqref{resol} is really a chain of quasi-isomorphisms it suffices to prove that the complexes $\mathbb R\mathscr O_A\left(x\right)\left(n\right)$ and $\Omega^n_{\mathbb R\mathscr O_A\left(x\right)}$ are weakly equivalent. As already done in the proof of Theorem \ref{main}, filter them by powers of the maximal ideal $\mathfrak m_A$, i.e. consider the filtrations\footnote{There is some abuse of notation in these formulae.}
\begin{eqnarray*}
&\mathcal F^p\left(\mathbb R\mathscr O_A\left(x\right)\left(n\right)\right):=\mathfrak m_A^p\mathbb R\mathscr O_A\left(x\right)\left(n\right)\quad\Rightarrow\quad\mathrm{Gr}^p\left(\mathcal F\right)\simeq\mathscr A^{n,*}_X\otimes\frac{\mathfrak m_A^p}{\mathfrak m_A^{p+1}}& \\
&\mathcal F^p\left(\Omega^n_{\mathbb R\mathscr O_A\left(x\right)}\right):=\mathfrak m_A^p\Omega^n_{\mathbb R\mathscr O_A\left(x\right)}\quad\Rightarrow\quad\mathrm{Gr}^p\left(\mathcal F\right)\simeq\Omega^n_{\mathscr A^{0,*}_X}\otimes\frac{\mathfrak m_A^p}{\mathfrak m_A^{p+1}}&
\end{eqnarray*}
which kill the twisting $l_x$. Now observe that 
\begin{equation*}
\Omega^n_{\mathscr A^{0,*}}\approx\Omega^n_{\mathscr O_X}\approx\mathscr A_X^{n,*}
\end{equation*}
where the first quasi-isomorphism is induced by the Dolbeaut resolution $\mathscr A^{0,*}_X\hookleftarrow\mathscr O_X$ and the second one is true basically by definition of $\Omega^{n}_{\mathscr O_X}$, so in particular $H^m\left(\Omega^n_{\mathscr A^{0,*}_X}\right)=H^m\left(\mathscr A^{n,*}_X\right)$ for all $m$. Finally, look at the induced spectral sequences: we have
\begin{equation*}
H^{p+q}\left(\mathscr A^{n,*}_X\otimes\frac{\mathfrak m_A^p}{\mathfrak m_A^{p+1}}\right)\simeq\bigoplus_{i+j=p+q}\left(H^i\left(\mathscr A^{n,*}_X\right)\otimes H^j\left(\frac{\mathfrak m_A^p}{\mathfrak m_A^{p+1}}\right)\right)\Longrightarrow H^{p+q}\left(\mathbb R\mathscr O_A\left(x\right)\left(n\right)\right)
\end{equation*}
and
\begin{equation*}
H^{p+q}\left(\Omega^n_{\mathscr A^{0,*}_X}\otimes\frac{\mathfrak m_A^p}{\mathfrak m_A^{p+1}}\right)\simeq\bigoplus_{i+j=p+q}\left(H^i\left(\mathscr A^{n,*}_X\right)\otimes H^j\left(\frac{\mathfrak m_A^p}{\mathfrak m_A^{p+1}}\right)\right)\Longrightarrow H^{p+q}\left(\Omega^n_{\mathbb R\mathscr O_A\left(x\right)}\right)
\end{equation*}
so the complexes $\Omega^n_{\mathbb R\mathscr O_A\left(x\right)}$ and $\mathbb R\mathscr O_A\left(x\right)\left(n\right)$ are quasi-isomorphic as their cohomologies are computed by the same spectral sequence. \\
Now let us look at diagram \eqref{diagr1} on the level of morphisms; a gauge element $\xi$ in the Kodaira-Spencer differential graded Lie algebra associated to $X$ maps through $\mathbb R\mathrm{Def}\left(l\right)$ to $l_{\xi}$, which in turn induces by $\nu$ the morphism of complexes
\begin{equation*}
\mathbb R\Gamma\left(X,e^{l_{\xi}}\right):\left(\mathbb R\Gamma\left(X,\Omega^*_{X/k}\right)\otimes A,d+l_{x_1}\right)\longrightarrow\left(\mathbb R\Gamma\left(X,\Omega^*_{X/k}\right)\otimes A,d+l_{x_2}\right).
\end{equation*}
In a similar way, the gauge $\xi$ determines via $\mu$ the morphism of complexes 
\begin{equation*}
e^{\xi}:\mathbb R\mathscr O_A\left(x_1\right)\longrightarrow\mathbb R\mathscr O_A\left(x_2\right)
\end{equation*}
which in turn induces through the bottom arrow in diagram \eqref{diagr1} the morphism
\begin{equation*}
\mathbb R\Gamma\left(X,\Omega^*_{e^{\xi}}\right):\mathbb R\Gamma\left(X,\Omega^*_{\mathbb R\mathscr O_A\left(x_1\right)/A}\right)\longrightarrow\mathbb R\Gamma\left(X,\Omega^*_{\mathbb R\mathscr O_A\left(x_2\right)/A}\right)
\end{equation*}
therefore we end up with a diagram
\begin{equation} \label{diagr3}
\xymatrix@C=1em{ \mathbb R\Gamma\left(X,\Omega^*_{\mathbb R\mathscr O_A\left(x_1\right)/A}\right)\ar[d]^{\mathbb R\Gamma\left(X,\Omega^*_{e^{\xi}}\right)} & \left(\Gamma\left(X,\mathscr A_X^{*,*}\otimes A\right),\partial+\left(\bar{\partial}+l_{x_1}\right)\right)\ar[l]\ar[d]^{\Gamma\left(X,e^{l_{\xi}}\right)} & \left(\mathbb R\Gamma\left(X,\Omega^*_{X/k}\right)\otimes A,d+l_{x_1}\right)\ar[l]\ar[d]^{\mathbb R\Gamma\left(X,e^{l_{\xi}}\right)} \\ 
\mathbb R\Gamma\left(X,\Omega^*_{\mathbb R\mathscr O_A\left(x_2\right)/A}\right) & \left(\Gamma\left(X,\mathscr A_X^{*,*}\otimes A\right),\partial+\left(\bar{\partial}+l_{x_2}\right)\right)\ar[l] & \left(\mathbb R\Gamma\left(X,\Omega^*_{X/k}\right)\otimes A,d+l_{x_2}\right).\ar[l] }
\end{equation}
Notice that the right hand square of diagram \eqref{diagr3} commutes because the morphism $\Gamma\left(X,e^{l_{\xi}}\right)$ is induced by $\mathbb R\Gamma\left(X,e^{l_{\xi}}\right)$ via the standard Dolbeaut resolution; as regards the left hand square, consider for all $n$ the unglobalised diagram
\begin{equation} \label{diagr4}
\xymatrix{\Omega^n_{\mathbb R\mathscr O_A\left(x_1\right)}\ar[d]^{\Omega^n_{e^{\xi}}} & \left(\mathscr A_X^{n,*}\otimes A,\bar{\partial}+l_{x_1}\right)\ar[l]\ar[d]^{e^{l_{\xi}}} \\
\Omega^*_{\mathbb R\mathscr O_A\left(x_2\right)} & \left(\mathscr A_X^{n,*}\otimes A,\bar{\partial}+l_{x_2}\right)\ar[l] }
\end{equation}
and again filter all complexes by powers of the maximal ideal $\mathfrak m_A$ in order to kill the derivations $l_{x_1}$, $l_{x_2}$ and hence the gauge $l_{\xi}$: we end up with a sequence of commutative diagrams
\begin{equation*}
\xymatrix{\Omega^n_{\mathscr A^{0,*}_X}\otimes\frac{\mathbf m_A^p}{\mathbf m_A^{p+1}}\ar@{=}[d] & \left(\mathscr A_X^{n,*}\otimes\frac{\mathbf m_A^p}{\mathbf m_A^{p+1}},\bar{\partial}\right)\ar[l]\ar@{=}[d] \\
\Omega^n_{\mathscr A^{0,*}_X}\otimes\frac{\mathbf m_A^p}{\mathbf m_A^{p+1}} & \left(\mathscr A_X^{n,*}\otimes\frac{\mathbf m_A^p}{\mathbf m_A^{p+1}},\bar{\partial}\right)\ar[l] }
\end{equation*}
therefore diagram \eqref{diagr4} has to commute and so does diagram \eqref{diagr3}, as well. This observation completes the proof. 
\end{proof}

\section{The Period Map as a Morphism of Derived Stacks}

Theorem \ref{univ lpm ddf} gives the ultimate picture of the local period map as a deformation-theoretic morphism, since it explains how the Fiorenza-Manetti map lifts naturally to the context of Derived Deformation Theory. Anyway, despite being entirely canonical, the Fiorenza-Manetti-Martinengo map\footnote{Again, Theorem \ref{univ lpm ddf} allows us to drop any further adjective.} is still a local morphism: concretely this means that it provides a fully satisfying description of the behaviour of {} ``derived variations of the Hodge structures'' associated to some nice $k$-scheme $X$ with respect to the infinitesimal derived deformations of the scheme itself, but this map is not able to give us any global information, i.e. it does not provide significant relations between the associated global (derived) moduli stacks. In this section we will set a path towards a coherent notion of global derived period map, which will be further studied in \cite{dNH}.
\subsection{A Non-Geometric Global Period Map}
Fix again $X$ to be a smooth proper scheme over $k$ of dimension $d$ and define the (non-geometric) derived stack 
\begin{equation}\label{new sch attempt}
\mathcal{DS}ch^X_{d/k}:=\mathcal{DS}ch_{d/k}\times^h_{\mathbb R\mathcal Perf_k}\left\{\mathbb R\Gamma\left(X,\Omega^*_{X/k}\right)\right\}
\end{equation}
where the map defining the homotopy fibre product in formula \eqref{new sch attempt} is
\begin{eqnarray*}
&\mathcal D\mathcal Sch_{d/k}&\xrightarrow{\hspace*{1cm}}\mathbb R\mathcal Perf_k \nonumber\\
&Y&\longmapsto \quad\quad\mathbb R\Gamma\left(\pi^0 Y,\Omega^*_{Y/A}\right)
\end{eqnarray*}
for all $A\in\mathfrak{dgAlg}_k^{\leq 0}$.
\begin{defn}
Define the \emph{non-geometric (universal) global period map} to be the morphism of derived stacks
\begin{eqnarray} \label{final mapping}
\underline{\mathbb R\mathcal P}:&\mathcal D\mathcal Sch^X_{d/k}&\xrightarrow{\hspace*{1cm}}\mathcal{DF}lag^0_k\left(\mathbb R\Gamma\left(X,\Omega^*_{X/k}\right)\right) \nonumber\\
&\left[Y,\theta:\mathbb R\Gamma\left(\pi^0 Y,\Omega^*_{Y/A}\right)\overset{\sim}{\rightarrow}\mathbb R\Gamma\left(X,\Omega^*_{X/k}\right)\otimes A\right]&\longmapsto\quad\quad\left(\mathbb R\Gamma\left(\pi^0 Y,\Omega^*_{Y/A}\right),\mathcal F\right)
\end{eqnarray}
for all $A\in\mathfrak{dgAlg}_k^{\leq 0}$.
\end{defn}
\begin{prop} \label{very final}
Consider the diagram of derived stacks and derived deformation functors
\begin{equation*}
\xymatrix{\mathcal D\mathcal Sch^X_{d/k}\ar[r]^{\underline{\mathbb R\mathcal P}\qquad\qquad} & \mathcal{DF}lag^0_k\left(\mathbb R\Gamma\left(X,\Omega^*_{X/k}\right)\right) \\
\mathbb R\mathrm{Def}_X\ar[u]\ar[r]^{\mathbb R\mathcal P\quad\quad} & \mathrm{hoFlag}^{\mathcal F}_{\mathbb R\Gamma\left(X,\Omega^*_{X/k}\right)}\ar[u]}
\end{equation*}
where the right-hand-side vertical arrow is the formal neighbourhood inclusion and the left-hand-side one is the composite
\begin{equation*}
\resizebox{.85\hsize}{!}{\xymatrix{ \mathbb R\mathrm{Def}_X\ar[rrrr]^{\left(\mathrm{Id}_{\mathbb R\mathrm{Def}_X},\mathrm{const}_{\mathbb R\Gamma\left(X,\Omega^*_{X/k}\right)}\right)\qquad\qquad\qquad\quad} & & & & \mathbb R\mathrm{Def}_X\times^h_{\mathrm{RDef}_{\mathbb R\Gamma\left(X,\Omega^*_{X/k}\right)}}\left(\mathbb R\Gamma\left(X,\Omega^*_{X/k}\right)\right)\ar[rrrr]^{\;\;\quad\qquad\qquad{\quad\scriptstyle{formal\;neighbourhood\;inclusion}}} & & & & \mathcal D\mathcal Sch^X_{d/k}.} }
\end{equation*}
Then the above square is well-defined and commutes.
\end{prop}
\begin{proof}
The fact that the diagram is well-defined is precisely the content of Remark \ref{DSch} and Theorem \ref{comp flag fin}.2; the commutativity is readily verified just walking along the arrows, as done in the proof of Theorem \ref{univ lpm ddf}.
\end{proof}
\subsection{Towards a Derived Analytic Period Mapping}
Proposition \ref{very final} is certainly a first step towards a globalisation of Griffiths period map, but it is certainly not sufficiently satisfying since map \eqref{final mapping} is highly non-geometric; let us outline here some of the ideas that are to be developed in \cite{dNH} and which should allow us to identify a reasonable global derived period map in the complex-analytic context. \\
Fix $k$ to be the field of complex numbers $\mathbb C$ and the base scheme $X$ to be smooth and projective; consider a family $p:\mathcal X\rightarrow \mathcal S$ of derived schemes globally deforming it, i.e. a homotopy pull-back diagram
\begin{equation*}
\xymatrix{\ar@{} |{\Box^h}[dr]X\ar@{^{(}->}[r]^i\ar[d] & \mathcal X\ar[d]^p \\
\mathrm{Spec}\left(\mathbb C\right)\ar[r] & \mathcal S}
\end{equation*}
where $p$ is flat and $H^0\left(p\right)$ is smooth. For any point $s\in\mathcal S$, we can consider the (homotopy) fibre of $p$ at $s$, which is a derived scheme we will denote as $\mathcal X_s$. Unlike the infinitesimal case, the associated filtered (derived) De Rham complex $\left(\mathbb R\Gamma\left(\pi^0\mathcal X_s,\Omega^*_{\mathcal X_s/\mathbb C}\right),F^{\bullet}\right)$ does not determine a point in the homotopy flag variety associated to $\mathbb R\Gamma\left(X,\Omega^*_{X/\mathbb C}\right)$, because the group $\Omega\left(\pi^0\mathcal S\right)$ -- where $\Omega\left(\pi^0\mathcal S\right)$ stands for the (simplicial) loop group attached to the topological space $\pi^0\mathcal S$ -- acts by (higher) monodromy on the fibres of $\mathbb Rp_{*}\Omega_{\mathcal X/S}$ and thus on $\mathcal{DF}lag_{\mathbb C}\left(\mathbb R\Gamma\left(X,\Omega^*_{X/\mathbb C}\right)\right)$. Therefore we would like to define the global period map associated to the family $p$ as a morphism of derived stacks
\begin{eqnarray} \label{finalissima}
&\mathcal S&\longrightarrow\nicefrac{\mathcal{DF}lag_{\mathbb C}\left(\mathbb R\Gamma\left(X,\Omega^*_{X/\mathbb C}\right)\right)}{\Omega\left(\pi^0S\right)} \nonumber \\
&s&\longmapsto\left[\left(\mathbb R\Gamma\left(\pi^0\mathcal X_s,\Omega^*_{\mathcal X_s/\mathcal S}\right),F^{\bullet}\right)\right].
\end{eqnarray}
However it is not clear that the quotient $\nicefrac{\mathcal{DF}lag_{\mathbb C}\left(\mathbb R\Gamma\left(X,\Omega^*_{X/\mathbb C}\right)\right)}{\Omega\left(\pi^0\mathcal S\right)}$ exists as a derived geometric stack in the algebraic sense -- and in general we do not expect it to do -- but it is very likely to exist in the analytic setting. As a matter of fact in the underived case the monodromy action on the Grassmannian (or flag variety) is properly discontinuous and this ensures that the quotient has a complex structure; it is then reasonable to expect that the action of $\Omega\left(\pi^0\mathcal S\right)$ on the {} ``analytification'' of $\mathcal{DF}lag_{\mathbb C}\left(\mathbb R\Gamma\left(X,\Omega^*_{X/\mathbb C}\right)\right)$ (or a suitable open substack of it recovering the notion of period domain in the derived setting) should be regular enough for the quotient to exist as a derived analytic stack and map \eqref{finalissima} to be (derived) holomorphic. Thus the ultimate study of the global derived period mapping involves massively the newly-born theory of Derived Analytic Geometry, as well as a good notion of analytification functor in the derived context (see \cite{Lu4}, \cite{Porta1}, \cite{Porta2}, \cite{Porta3}, \cite{PortaYue} for foundational work). The goal of \cite{dNH} is to analyse thoroughly the above ideas.
\section*{Notations and conventions}

\begin{itemize}
\item If $i\geq 0$ $\Delta^i$ is the $i$-th standard simplicial simplex
\item $\mathrm{diag}\left(-\right)=$ diagonal of a bisimplicial set
\item $k=$ fixed field of characteristic $0$, unless otherwise stated
\item If $A$ is a (possibly differential graded) local Artin ring, $\mathfrak m_A$ will be its unique maximal (possibly differential graded) ideal
\item $R=$ fixed (possibly differential graded) commutative unital $k$-algebra, unless otherwise stated
\item If $R$ is a commutative unital ring then $\underline R$ is the constant sheaf of stalk $R$
\item If $\left(V^*,d\right)$ is a cochain complex (in some suitable category) then $\left(V\left[n\right]^*,d_{\left[n\right]}\right)$ will be the cochain complex such that $V\left[n\right]^j:=V^{j+n}$ and $d_{\left[n\right]}^j=d^{j+n}$
\item $\mathbb G_m=$ multiplicative group scheme over $k$
\item $X=$ smooth proper scheme over $k$ of finite dimension, unless otherwise stated
\item $\mathscr O_X=$ structure sheaf of $X$
\item $\mathscr T_X=$ tangent sheaf of $X$
\item $\mathscr A^{0,*}_X=$ {}``Dolbeaut'' complex of $X$
\item $\mathscr A^{*,*}_X=$ double complex of {}``$\bar k$''-valued forms on $X$
\item $\Omega^*_{X/k}=$ algebraic De Rham complex of $X$
\item $F^{\bullet}=$ Hodge filtration on $\Omega^*_{X/k}$ or cohomology, unless otherwise stated
\item $\mathbb L_{X/k}=$ (absolute) cotangent complex of $X$ over $k$
\item $\mathfrak{Sh}\left(X\right)=$ category of sheaves of abelian groups over $X$
\item $\mathrm D\left(X\right)=$ derived category of $X$
\item $\Delta=$ category of finite ordinal numbers
\item $\mathfrak{Alg}_k=$ category of commutative associative unital algebras over $k$
\item $\mathfrak{Aff}_k=$ category of (linear) affine spaces over $k$  
\item $\mathfrak{Art}_k=$ category of local Artin algebras over $k$
\item $\mathfrak{Ch}_{\geq 0}\left(\mathfrak{Vect}_k\right)=$ model category of chain complexes of vector spaces over $k$ in non-negative degrees
\item $\mathfrak{Ch}_{\geq 0}\left(\mathfrak{Aff}_k\right)$ model category of (chain) $\mathrm{dg}_{\geq 0}$-affine spaces over $k$
\item $\mathfrak{Def}^{\mathrm{Hin}}_k=$ $\infty$-category of Hinich derived deformation functors (over $k$)
\item $\mathfrak{Def}^{\mathrm{Man}}_k=$ $\infty$-category of Manetti extended deformation functors (over $k$)
\item $\mathfrak{dg}_{\geq 0}\mathfrak{Alg}_k=$ model category of (chain) differential graded commutative algebras over $k$ in non-negative degrees
\item $\mathfrak{dg}_{\geq 0}\mathfrak{Cat}_k=$ model category of (chain) differential graded categories over $k$
\item $\mathfrak{dg}_{\geq 0}\mathfrak{Cat}^{\mathfrak{Aff}}_k=$ $\infty$-category of affine (chain) differential graded categories over $k$
\item $\mathfrak{dg}_{\geq 0}\mathfrak{Grpd}^{\mathfrak{Aff}}_k=$ $\infty$-category of affine (chain) differential graded groupoids over $k$
\item $\mathfrak{dgAlg}^{\leq 0}_k=$ model category of (cochain) differential graded commutative algebras over $k$ in non-positive degrees
\item $\mathfrak{dgArt}_k=$ model category of (cochain) differential graded local Artin algebras over $k$
\item $\mathfrak{dgArt}^{\leq 0}_k=$ model category of (cochain) differential graded local Artin algebras over $k$ in non-positive degrees
\item $\mathfrak{dgLie}_k=$ model category of (cochain) differential graded Lie algebras over $k$
\item $\mathfrak{dgMod}_R=$ model category of $R$-modules in (cochain) complexes
\item $\mathfrak{dg}_b\mathfrak{Nil}^{\leq 0}_k=$ $\infty$-category of bounded below differential graded commutative $k$-algebras in non-positive degrees such that the canonical map $A\rightarrow H^0\left(A\right)$ is nilpotent
\item $\mathfrak{dgVect}_k^{\leq 0}=$ model category of (cochain) differential graded vector spaces over $k$ in non-positive degrees
\item $\mathfrak{FdgMod}_R=$ model category of filtered $R$-modules in (cochain) complexes
\item $\mathfrak{Grpd}=$ 2-category of groupoids
\item $\mathfrak{Set}=$ category of sets
\item $\mathfrak{Sch}_k=$ category of schemes over $k$
\item $\mathfrak{sAff}_k=$ model category of simplicial affine spaces over $k$
\item $\mathfrak{sAlg}_k=$ model category of simplicial commutative associative unital algebras over $k$ 
\item $\mathfrak{sCat}=$ model category of simplicial categories
\item $\mathfrak{sCat}_k=$ $\infty$-category of $k$-simplicial categories over $k$
\item $\mathfrak{sCat}_k^{\mathfrak{Aff}}=$ $\infty$-category of affine simplicial categories over $k$
\item $\mathfrak{sGrpd}=$ model category of simplicial groupoids
\item $\mathfrak{sGrpd}_k^{\mathfrak{Aff}}=$ $\infty$-category of affine simplicial groupoids over $k$
\item $\mathfrak{sSet}=$ simplicial model category of simplicial sets
\item $\mathfrak{sVect}_k=$ model category of simplicial vector spaces over $k$
\item $\mathfrak{Vect}_k=$ category of vector spaces over $k$
\end{itemize}


\begin{thebibliography}{99}
\bibitem{CR}{A. M. Cegarra, J. Remedios,}
            \emph{The relationship between the diagonal and the bar constructions on a bisimplicial set}, { Topology and its Applications, $\mathbf{153}$ (1) (2005): 21-51.}
\bibitem{Del1}{P. Deligne,}
              \emph{Th\' eor\` eme de Lefschetz et crit\` eres de d\' eg\' enerescence de suites spectrales}, { Publications math\' ematiques de l'I.H.\' E.S., $\mathbf{35}$ (1968): 107-126.}
\bibitem{Del2}{P. Deligne,}
              \emph{Th\' eorie de Hodge, II}, { Publications Math\' ematiques de l'I.H.\' E.S., $\mathbf{40}$ (1971): 5-57.}
\bibitem{Del3}{P. Deligne,}
              \emph{Th\' eorie de Hodge, III}, { Publications Math\' ematiques de l'I.H.\' E.S., $\mathbf{44}$ (1974): 5-77.}
\bibitem{dN1}{C. Di Natale,}
             \emph{An Overview of Derived Deformation Theory}, { https://www.dpmms.cam.ac.uk/~cd475/GAeL.pdf, 2013.}
\bibitem{dN2}{C. Di Natale,}
            \emph{Derived Moduli of Complexes and Derived Grassmannians}, { $\mathbf{arXiv:1407.5900}$, 2014.}
\bibitem{dNH}{C. Di Natale, J. Holstein,}
             \emph{A Global Derived Analytic Period Map}, { in preparation.}
\bibitem{FMan1}{D. Fiorenza, M. Manetti,}
               \emph{$L_{\infty}$-Algebras, Cartan Homotopies and Period Maps}, { $\mathbf{arXiv:math/0605297}$, 2006.}
\bibitem{FMan2}{D. Fiorenza, M. Manetti,}
               \emph{$L_{\infty}$-Structures on Mapping Cones}, { Algebra and Number Theory, $\mathbf{1}$ (2007): 301-330.}
\bibitem{FMan3}{D. Fiorenza, M. Manetti,}
               \emph{A Period Map for Generalised Deformations}, { Journal of Noncommutative Geometry (3), $\mathbf 4$ (2009): 579-597.}
\bibitem{FMar}{D. Fiorenza, E. Martinengo,}
              \emph{A Short Note on $\infty$-Groupoids and the Period Map for Projective Manifolds}, { Publications of the nLab (2), $\mathbf 1$ (2012): 1-13.}
\bibitem{Getz}{E. Getzler,}
             \emph{Lie Theory for Nilpotent $L_{\infty}$-Algebras}, { Annals of Mathematics (2), $\mathbf{170}$ (1) (2009): 271-301.}
\bibitem{GS}{P. G. Goerss, K. Schemmerhon,}
            \emph{Model Categories and Simplicial Methods}, { in {}``Interactions Between Homotopy Theory and Algebra'': 3-49, Contemporary Mathematics, $\mathbf{436}$, American Mathematical Society, Providence, RI, 2007.}
\bibitem{GJ}{P. G. Goerss, J. F. Jardine,}
            \emph{Simplicial Homotopy Theory}, { Birk\" auser, 1999.}
\bibitem{Grif}{P. A. Griffiths,} 
              \emph{Periods of Integrals on Algebraic Manifolds II - Local Study of the Period Mapping}, { American Journal of Mathematics, $\mathbf{90}$ (1968): 805-865.}
\bibitem{Hin1}{V. Hinich,}
             \emph{Descent of Deligne Groupoids}, { International Mathematical Research Notices, $\mathbf 5$ (1997): 223-239.}
\bibitem{Hin2}{V. Hinich,}
             \emph{DG Coalgebras as Formal Stacks}, { Journal of Pure and Applied Algebra, $\mathbf{162}$ (2001): 209-250.}
\bibitem{Hin4}{V. Hinich,}
             \emph{Deformations of Homotopy Algebras}, { Communications in Algebra, $\mathbf{32}$ (2004): 473-494.}
\bibitem{Iac1}{D. Iacono,}
             \emph{Differential Graded Lie Algebras and Deformations of Holomorphic Maps}, { PhD Thesis, $\mathbf{arXiv:math.AG/0701091}$, 2006.}
\bibitem{Iac2}{D. Iacono,}
             \emph{$L_{\infty}$-Algebras and Deformations of Holomorphic Maps}, { International Mathematical Research Notices, $\mathbf{2008}$ (2008):article ID:mn013, 36 pages.}
\bibitem{IacMan}{D. Iacono, M. Manetti,}
             \emph{An Algebraic Proof of Bogomolov-Tian-Todorov Theorem}, { in \emph{Deformation Spaces}, $\mathbf{39}$, Vieweg Verlag, 2010.}
\bibitem{Kon}{M. Kontsevich,}
             \emph{Deformation Quantization of Poisson Manifolds}, { Letters of Mathematical Physics, $\mathbf{66}$ (3) (2003): 157-216.}
\bibitem{KS}{M. Kontsevich, Y. Soibelman,}
            \emph{Deformations of Algebras over Operads and Deligne's Conjecture}, { G. Dito and D. Sternheimer (eds) Conf\' erence Mosh\' e Flato 1999, Vol. I (Dijon 1999), Kluwer Academic Publications, Dordrecht (2000): 255-307.}
\bibitem{LV}{J.-L. Loday, B. Vallette,}
            \emph{Algebraic Operads}, { Grundlehren der Mathematischen Wissenschaften $\mathbf{346}$, Springer, Heidelberg, 2012.}
\bibitem{Lie}{M. Lieblich,}
            \emph{Moduli of Complexes on a Proper Morphism}, { Journal of Algebraic Geometry, $\mathbf{15}$ (2006): 175-206.}
\bibitem{Lu1}{J. Lurie,}
            \emph{Derived Algebraic Geometry}, { PhD thesis, www.math.harvard.edu/~lurie/papers/DAG.pdf, 2004.}
\bibitem{Lu2}{J. Lurie,}
            \emph{Higher Topos Theory}, { Annals of Mathematics Studies, 2009.}
\bibitem{Lu4}{J. Lurie,}
            \emph{Derived Algebraic Geometry V -- Structured Spaces}, { www.math.harvard.edu/~lurie/papers/DAG-V.pdf, 2011.}
\bibitem{Lu3}{J. Lurie,}
            \emph{Derived Algebraic Geometry X -- Formal Moduli Problems}, { www.math.harvard.edu/~lurie/papers/DAG-X.pdf, 2011.}
\bibitem{Man1}{M. Manetti,}
             \emph{Extended Deformation Functors}, { International Mathematical Research Notices, $\mathbf{14}$ (2002): 719-756.}
\bibitem{Man2}{M. Manetti,}
             \emph{Lectures on Deformations of Complex Manifolds}, { Rendiconti di Matematica, VII, $\mathbf 4$ (2004):1-183.}
\bibitem{Porta1}{M. Porta,}
              \emph{Derived Complex Analytic Geometry I: GAGA Theorems}, { $\mathbf{arXiv:1506.09042}$, 2015.}
\bibitem{Porta2}{M. Porta,}
              \emph{Comparison Results for Derived Deligne-Mumford Stacks}, { $\mathbf{arXiv:1507.00573}$, 2015.}
\bibitem{Porta3}{M. Porta,}
              \emph{Derived Complex Analytic Geometry II: Square-Zero Extensions}, {$\mathbf{arXiv:1507.06602}$, 2015.}
\bibitem{PortaYue}{M. Porta, T. Yue Yu}
              \emph{Higher Analytic Stacks and GAGA Theorems}, { $\mathbf{arXiv:1412.5166}$, 2014.}
\bibitem{Pr0}{J. P. Pridham,}
             \emph{Derived Deformations of Artin Stacks}, { $\mathbf{arXiv:0805.3130}$, 2008.}
\bibitem{Pr1}{J. P. Pridham,}
             \emph{Unifying Derived Deformation Theories}, { Advances in Mathematics, $\mathbf{224}$ (3) (2010): 772-826.}
\bibitem{Pr2}{J. P. Pridham,}
             \emph{Presenting Higher Stacks as Simplicial Schemes}, { Advances in Mathematics, $\mathbf{238}$ (2013): 184-245.}
\bibitem{Pr2bis}{J. P. Pridham,}
             \emph{Representability of Derived Stacks}, { Journal of $K$-Theory (10), $\mathbf 2$ (2012): 413-453}
\bibitem{Pr3}{J. P. Pridham,}
             \emph{Derived Moduli of Schemes and Sheaves}, { Journal of $K$-Theory (10), $\mathbf 1$ (2012): 41-85, $\mathbf{arXiv:1011.2742v6}$.}
\bibitem{Pr4}{J. P. Pridham,}
             \emph{Constructing Derived Stacks}, { Geometry and Topology, $\mathbf{17}$ (3) (2013): 1417-1495.}
\bibitem{Tab}{G. Tabuada,}
             \emph{Differential Graded Versus Simplicial Categories}, { Topology and its Applications, $\mathbf{157}$ (3) (2010): 563-593.}
\bibitem{Toe}{B. To\" en,}
             \emph{Higher and Derived Stacks: a Global Overview}, { in \emph{Algebraic Geometry - Seattle 2005}, Proceedings of Symposia in Pure Mathematics, $\mathbf{80}$, Part 1, American Mathematical Society, Providence, RI, 2009.}
\bibitem{TVaq}{B. To\" en, M. Vaqui\' e,}
             \emph{Moduli of Objects in DG-Categories}, { Annales Scientifiques de l' \' Ecole Normale Sup\' erieure (4) $\mathbf{40}$ (2007): 387-444.}
\bibitem{TV}{B. To\" en, G. Vezzosi,}
            \emph{Homotopical Algebraic Geometry II: Geometric Stacks and Applications}, { Memoirs of the American Mathematical Society, $\mathbf{193}$ (902), 2008.}
\bibitem{Schl}{M. J. Schlessinger},
             \emph{Functors of Artin Rings}, { Transactions of the American Mathematical Society, $\mathbf{130}$ (1968): 208-222.}
\bibitem{SchShi}{S. Schwede, B. Shipley,}
            \emph{Equivalences of Monoidal Model Categories}, {Algebraic and Geometric Topology, $\mathbf 3$ (2003): 287-334.}
\bibitem{Va}{B. Vallette,}
            \emph{Algebra$\,+\,$Homotopy$\,=\,$Operad}, { to appear in the Proceedings of the Mathematical Science Research Institute and Research Institute for Mathematical Science {}``Symplectic Geometry, Noncommutative Geometry and Physics'', $\mathbf{arXiv:1202.3245}$, 2012.}
\bibitem{Vo}{C. Voisin,}
            \emph{Hodge Theory and Complex Algebraic Geometry Vol. I}, { Cambridge University Press, 2002.}
\bibitem{We}{C. Weibel,}
            \emph{An Introduction to Homological Algebra}, { Cambridge University Press, 1994.}
\end{thebibliography}
\end{document}